\documentclass[reqno,a4paper,11pt]{article}
\pdfoutput=1
\usepackage{CJKutf8}

\usepackage{xcolor}
\definecolor{MyDarkBlue}{rgb}{0.15,0.25,0.45}
\usepackage{graphicx}
\usepackage[textwidth = 430 pt, textheight = 630 pt]{geometry}

\usepackage{amsmath,amssymb,amsthm}
\usepackage{amsmath, amsthm, amssymb}
\usepackage{amsfonts}
\usepackage{mathrsfs}
\usepackage{bbm}
\usepackage[normalem]{ulem}
\usepackage{booktabs,longtable}
\usepackage{latexsym}
\usepackage{aliascnt}     
\usepackage{lmodern}
\usepackage{microtype}

\usepackage{hyperref}
\hypersetup{
    hypertexnames=false,
    colorlinks=true,
    citecolor=MyDarkBlue,
    linkcolor=MyDarkBlue,
    urlcolor=MyDarkBlue,
    pdfauthor={},
    pdftitle={},
    pdfsubject={hep-th math-ph},
    breaklinks=true
}

\newcommand{\sharedtheorem}[2]{%
    \newaliascnt{#1}{theorem}%      alias counter sharing theorem's register
    \newtheorem{#1}[#1]{#2}%        environment stepping its own alias counter
    \aliascntresetthe{#1}%          fix \the#1 so \newtheorem is not fooled
}

\theoremstyle{plain}       % italic body -- assertions
\newtheorem{theorem}{Theorem}[section]   % <-- the one master counter
\sharedtheorem{lemma}{Lemma}
\sharedtheorem{proposition}{Proposition}
\sharedtheorem{corollary}{Corollary}
\sharedtheorem{conjecture}{Conjecture}

\theoremstyle{definition}  % upright body -- definitions / examples
\sharedtheorem{definition}{Definition}
\sharedtheorem{example}{Example}
\sharedtheorem{question}{Question}
\sharedtheorem{notation}{Notation}

\theoremstyle{remark}      % upright, unemphasised -- asides
\sharedtheorem{remark}{Remark}

\numberwithin{equation}{section}

\usepackage[nameinlink,noabbrev]{cleveref}

\crefname{theorem}{Theorem}{Theorems}
\crefname{lemma}{lemma}{Lemmata}
\crefname{proposition}{Proposition}{Propositions}
\crefname{corollary}{Corollary}{Corollaries}
\crefname{conjecture}{Conjecture}{Conjectures}
\crefname{definition}{Definition}{Definitions}
\crefname{example}{Example}{Examples}
\crefname{notation}{Notation}{Notations}
\crefname{question}{Question}{Questions}
\crefname{remark}{Remark}{Remarks}

\Crefname{theorem}{Theorem}{Theorems}
\Crefname{lemma}{Lemma}{Lemmata}
\Crefname{proposition}{Proposition}{Propositions}
\Crefname{corollary}{Corollary}{Corollaries}
\Crefname{conjecture}{Conjecture}{Conjectures}
\Crefname{definition}{Definition}{Definitions}
\Crefname{example}{Example}{Examples}
\Crefname{notation}{Notation}{Notations}
\Crefname{question}{Question}{Questions}
\Crefname{remark}{Remark}{Remarks}

\usepackage{tikz}
\usetikzlibrary{matrix,cd,arrows}
\usepackage{mathtools}

\tikzcdset{scale cd/.style={every label/.append style={scale=#1},
        cells={nodes={scale=#1}}}}

\let\fn\footnote
\renewcommand{\footnote}[1]{\linespread{1.1}\fn{#1}\linespread{1.29}}

\makeatletter
\renewcommand{\section}{\@startsection
    {section}{1}{\z@}{-3.5ex plus -1ex minus
        -.2ex}{2.3ex plus .2ex}{\mathversion{bold}\bf }}
\renewcommand{\subsection}{\@startsection{subsection}{2}{\z@}{-3.25ex
        plus -1ex minus
        -.2ex}{1.5ex plus .2ex}{\mathversion{bold}\bf }}
\renewcommand{\subsubsection}{\@startsection{subsubsection}{3}{-2.45ex}{-3.25ex
        plus -1ex minus -.2ex}{1.5ex plus .2ex}{\it }}
\renewcommand\paragraph{\@startsection{paragraph}{4}{\z@}%
    {3.25ex \@plus1ex \@minus.2ex}%
    {-1em}%
    {\normalfont\normalsize\bfseries\mathversion{bold}}}

\renewcommand{\@seccntformat}[1]{\@nameuse{the#1}.~~}

\makeatletter
\newcommand{\groupappendicesintoc}{%
    \addtocontents{toc}{\protect\addvspace{1em}}%
    \phantomsection\addcontentsline{toc}{section}{Appendices}%
    \let\oldaddcontentsline\addcontentsline
    \renewcommand\addcontentsline[3]{%
        \def\@tempa{##2}\def\@tempb{section}%
        \ifx\@tempa\@tempb
        \oldaddcontentsline{##1}{subsection}{##3}%
        \else
        \oldaddcontentsline{##1}{##2}{##3}%
        \fi
    }%
}
\makeatother

\newcommand{\appendices}{
    \section*{Appendix}\label{appendices}\setcounter{subsection}{0}
    \appendix
    \crefalias{section}{appendix}
    \crefalias{subsection}{appendix}
    \crefname{appendix}{appendix}{appendices}
    \Crefname{appendix}{Appendix}{Appendices}
    \groupappendicesintoc
}

\newcommand{\makecommand}[3]{%
    \foreach \i in #3 {%
        \expandafter\xdef\csname #1\i\endcsname{\noexpand#2{\unexpanded\expandafter{\i}}}%
    }%
}

\makecommand{fr}{\mathfrak}{{c,der,g,h,A,X,gl,sl,o,so,osp,u,su,at,brst,sp,spin,string,Poisson,td,tb}}

\makecommand{sf}{\mathsf}{{id,String,Lie,Hom,SU,SO,hom,CE,Bibun,Cat,Set,Spin,Pin,inv,CE,Diff,HSymp,TD,TB,pr,Aut,GL,SL,IO,Mat,Sym,Inn,Out,AUT,Vect}}

\makecommand{rm}{\mathrm}{{Ham,diag,im,dim,Ad,deg,id}}

\newcommand{\latinalphabet}{A,a,B,b,C,c,d,D,E,e,F,f,G,g,H,h,I,i,J,j,K,k,L,l,M,m,N,n,O,o,P,p,Q,q,R,r,S,s,T,t,U,u,V,v,W,w,X,x,Y,y,Z,z}
\newcommand{\caplatinalphabet}{A,B,C,D,E,F,G,H,I,J,K,L,M,N,O,P,Q,R,S,T,U,V,W,X,Y,Z}
\makecommand{I}{\mathbbm}{\latinalphabet}
\makecommand{bf}{\mathbf}{\latinalphabet}
\makecommand{bm}{\bm}{\latinalphabet}
\makecommand{ca}{\mathcal}{\caplatinalphabet}
\makecommand{fr}{\mathfrak}{\latinalphabet}
\makecommand{rm}{\mathrm}{\latinalphabet}
\makecommand{sf}{\mathsf}{\latinalphabet}
\makecommand{sc}{\mathscr}{\latinalphabet}
\makecommand{tt}{\mathtt}{\latinalphabet}

\newcommand{\acton}{\triangleright}

\def\slasha#1{\setbox0=\hbox{$#1$}#1\hskip-\wd0\hbox to\wd0{\hss\sl/\/\hss}}

\def\periodb#1{\setbox0=\hbox{$#1$}#1\hskip-\wd0\hbox to\wd0{-}}

\newcommand{\dpar}{\partial}

\newcommand{\unit}{\mathbbm{1}}   			% identity map/matrix
\newcommand{\catdgMfd}{\mathsf{dgMfd}}
\newcommand{\catdgGrpd}{\mathsf{dgGrpd}}
\newcommand{\CatdgGrpd}{\mathsf{dg2Grpd}}
\newcommand{\catSet}{\mathsf{Set}}
\newcommand{\CatCat}{\mathsf{Cat}}
\newcommand{\CatNQMfd}{\mathsf{NQMfd}}

\newcommand{\eand}{~~~\mbox{and}~~~}

\newcommand{\ewith}{~~~\mbox{with}~~~}
\newcommand{\eps}{\varepsilon}

\makeatletter
\AtBeginDocument{%
    \let\oldref\ref
    \renewcommand*{\ref}{\@ifstar{\oldref*}{\cref}}%
    \renewcommand{\eqref}[1]{{\rm (\oldref{#1})}}
}
\makeatother

\newcommand\fibtimes[2]{\mathbin{_{#1}\times_{#2}}}
\newcommand\racton{\mathbin{\vartriangleleft}}
\newcommand\ractonB{\mathbin{\blacktriangleleft}}

\newcommand{\ourodot}{\left.\bigodot\right.}

\newcommand{\ihom}{\underline{\sfhom}}

\newcommand{\inv}{\mathsf{inv}}
\newcommand\connectiononeform{\vartheta^{\mathrm{tot}}}

\begin{document}
    \begin{titlepage}
        \vspace*{2cm}
        \begin{center}
            {\LARGE \bf
                Adjusted Connections I:
                \\[0.2cm]
                Differential Cocycles for\\[0.2cm]
                Principal Groupoid Bundles with Connection
            }
            \vskip1cm
            {\Large {Simon-Raphael \textsc{Fischer}\textsuperscript1\qquad Mehran \textsc{Jalali Farahani}\textsuperscript2\qquad Hyungrok \textsc{Kim} (\begin{CJK*}{UTF8}{bkai}金炯錄\end{CJK*})\textsuperscript3\qquad Christian \textsc{Saemann}\textsuperscript2}}\\[0.5cm]
            \vskip0.5cm
            {
                \textsuperscript1 National Center for Theoretical Sciences (\begin{CJK*}{UTF8}{bkai}國家理論科學研究中心\end{CJK*}),\\ Mathematics Division, National Taiwan University\\
                Room 407, Cosmology Building, No.\ 1, Sec.\ 4, Roosevelt Rd., Taipei City 106319, Taiwan\\
                \begin{CJK*}{UTF8}{bkai}106319 臺北市羅斯福路四段1號　(國立臺灣大學次震宇宙館407室)\end{CJK*}
                \\[0.3cm]
                \textsuperscript2 Maxwell Institute for Mathematical Sciences\\
                Department of Mathematics, Heriot--Watt University\\
                Colin Maclaurin Building, Riccarton, Edinburgh EH14 4AS, U.K.\\[0.3cm]
                \textsuperscript3
                Department of Mathematics and Theoretical Physics\\ University of Hertfordshire, Hatfield,
                Herts.\ AL10 9AB, U.K.}\\[0.3cm]
            {\href{mailto:sfischer@ncts.tw}{\texttt{sfischer@ncts.tw}}\qquad\href{mailto:mj2020@hw.ac.uk}{\texttt{mj2020@hw.ac.uk}}\\
                \href{mailto:h.kim2@herts.ac.uk}{\texttt{h.kim2@herts.ac.uk}}\qquad\href{mailto:c.saemann@hw.ac.uk}{\texttt{c.saemann@hw.ac.uk}}}
        \end{center}
        \vskip0.8cm
        \begin{center}
            \textbf{Abstract}
        \end{center}
        \begin{quote}
            We develop a new perspective on principal bundles with connection as morphisms from the tangent bundle of the underlying manifold to a classifying dg-Lie groupoid. This groupoid can be identified with a lift of the inner homomorphisms groupoid arising in Ševera's differentiation procedure of Lie quasi-groupoids. Our new perspective readily extends to principal groupoid bundles, but requires an adjustment, an additional datum familiar from higher gauge theory. We show that for Lie groupoids, the additional adjustment data amounts to a Cartan connection. The resulting adjusted connections naturally provide a global formulation of the kinematical data of curved Yang--Mills--Higgs theories as described by Kotov--Strobl~\cite{Kotov:2015nuz} and Fischer~\cite{Fischer:2021glc}.
        \end{quote}
    \end{titlepage}
    
    \tableofcontents
    
    \newpage
    
    \section{Introduction and results}
    
    \paragraph{Motivation.}
    In this paper, we provide a generalization of the notion of connection on a principal fiber bundle to a principal groupoid bundle\footnote{or {\em principal bundle over Lie groupoids} or {\em groupoid principal bundle}} as defined, e.g., in~\cite[Section 5.7]{Moerdijk:2003bb}. In this construction, we encounter similar features as in the context of higher gauge theory, which makes the construction non-trivial and mathematically interesting. 
    
    A key motivation for our discussion, however, stems directly from physics, and in particular from (quantum) field theory. Recall that {\em gauge field theories} underlie our modern understanding of particle physics. {\em Gauge fields}, which are part of the kinematical data of such theories, are connections on principal fiber bundles. In many gauge field theories, the kinematical data also includes {\em matter fields}, which are usually sections of associated vector bundles. 
    
    There are, however, more general situations, in which the matter fields do not take values in a vector space, but in a general manifold $M$, usually called the {\em target space}. If this manifold carries the action of a Lie group $\sfG\curvearrowright M$, then the matter fields can be {\em gauged}. That is, the gauge and matter field content are combined into a connection on a principal bundle with the action groupoid $M\times\sfG\rightrightarrows M$ as structure groupoid. There is no other reason than technical simplicity for restricting to action groupoids, and the technology developed in this paper allows indeed for very general\footnote{and there are good reasons to believe that there is no more general formulation} Lie groupoids as structure Lie groupoids. This leads directly to new forms of the famous Higgs mechanism that have so far been mathematically inaccessible and which may even have applications in phenomenology.
    
    \paragraph{The problem.} The construction of principal bundles with higher structure groups or even higher structure groupoids $\scG$ is straightforward. In particular, one encounters only technical but no fundamental difficulties in constructing the relevant classifying space $\sfB\scG$.
    
    Endowing such bundles with connections, however, is more subtle, even locally, i.e.~over a coordinate patch of the base manifold over which the bundle trivializes. Consider, for example, the local description of connections as morphisms of differential graded \mbox{(dg-)}commu\-tative algebras~\cite{Cartan:1949aaa,Cartan:1949aab,Kotov:2007nr} and in particular~\cite{Sati:2008eg}. For ordinary principal bundles, this description is obtained by restricting the global picture of a connection as a section of the Atiyah algebroid sequence~\cite{Atiyah:1957} to a patch $U$ of the base space. A connection is then a graded morphism $\caA\colon \rmT[1]U\rightarrow \frg[1]$ from the grade-shifted tangent bundle to the grade-shifted structure Lie algebra\footnote{See \ref{app:graded_manifolds} for our conventions and more details on graded manifolds}. Such morphisms are defined dually as morphisms of graded commutative algebras\footnote{Recall that the Chevalley--Eilenberg algebra of a Lie algebra $\frg$ is the differential graded commutative algebra $\sfCE(\frg)=\ourodot^\bullet\frg[1]^*$ with differential $\sfd a=-\tfrac12[a,a]$ for $a\in \frg[1]^*$, see also~\ref{ssec:dg-manifolds_and_Lie_algebroids}.} $\caA^\sharp\colon\sfCE(\frg)\rightarrow \Omega^\bullet(M)$. The curvature is then the failure of $\caA^\sharp$ to respect the differential, and dg-commutative algebra morphisms correspond to flat connections. 
    
    This construction generalizes to higher Lie algebras as well as (higher) Lie algebroids $\frg$, but in many of these cases, the closure of the Lie algebra of gauge transformations restricts the curvatures strongly, and both these restrictions as well as the Bianchi identities arising in this construction contradict the expectations from physics, see~e.g.~\cite{Borsten:2024gox} for a discussion.
    
    \paragraph{The solution.} In the case of higher structure groups, a solution to this problem has been found~\cite{Sati:2008eg,Fiorenza:2012tb,Saemann:2019dsl,Kim:2019owc,Rist:2022hci}: the definition of curvature has to be corrected by what has been called a {\em Chern--Simons term} or, more generally, an {\em adjustment}. Note that higher flat connections are readily defined by restricting the above-mentioned map $\caA^\sharp:\sfCE(\frg)\rightarrow \Omega^\bullet(M)$ to a morphism of differential graded commutative algebras. The key observation is then that non-flat connections can be regarded as flat connections for a higher structure algebra, as pointed out for Lie 2-algebras in~\cite{Roberts:0708.1741}. Here, the curvature forms are identified with additional higher connection forms by the flatness condition. To a certain extent, this observation already underlies the use of the Weil algebra of a Lie algebra $\frg$ for ordinary gauge theories in~\cite{Cartan:1949aaa,Cartan:1949aab}, which is the Chevalley--Eilenberg algebra of the Lie 2-algebra of inner derivations of $\frg$. 
    
    In the case where $\frg$ is concentrated in more than one degree, e.g.~for $\frg$ a Lie algebroid or $\frg$ a higher Lie algebra, there is now an ambiguity in the identification of the generators that are mapped by $\caA^\sharp$ to the additional higher connection forms. This is crucial, because different choices yield different definitions of curvatures, which induce different forms of gauge transformations. An adjustment is now such a choice with particularly nice properties. Importantly, the gauge transformations close without the above-mentioned strong constraints on the curvatures, and in many cases, the curvatures gauge-transform in a nice, ``covariant'' manner, which is important for the construction of action principles for field theories involving them.
    
    The full description of a higher principal bundle with connection is then obtained by rendering the above global. The main goal of the present paper is to achieve this for Lie algebroids that integrate to a Lie groupoid\footnote{Recall that not all Lie algebroids admit integrating Lie groupoids; one may have to extend to higher Lie groupoids.}.
    
    \paragraph{State of the art.} So far, adjustments have been given mostly locally, and then for particular examples of higher Lie algebras, such as string-like higher Lie algebras~\cite{Sati:2008eg}, their strict loop space models~\cite{Saemann:2019dsl}, twisted string and fivebrane structures~\cite{Sati:2009ic}, and the higher Lie algebras appearing in the tensor hierarchies of gauged supergravity~\cite{Borsten:2021ljb}. Global descriptions have been given either abstractly, see e.g.~\cite{Fiorenza:2012tb} for a weak string 2-group model, or explicitly, but only for particular examples, see~\cite{Kim:2019owc,Rist:2022hci} for a strict string 2-group model; see also~\cite{Borsten:2024gox} for more details.
    
    For principal groupoid bundles, the local description of adjusted connections was first proposed (albeit in a slightly different language) in~\cite{Strobl:2004im}, with further understanding of the involved gauge transformations in~\cite{Bojowald:0406445, Mayer:2009wf}. Eventually, this type of local kinematical data for ``curved Yang--Mills--Higgs gauge theories'' was summarized and finalized in~\cite{Kotov:2015nuz}. In Fischer's Ph.D.\ thesis~\cite{Fischer:2021glc} these local connections were redeveloped in a coordinate-free manner, and their geometry was studied. This thesis also provides several new examples of adjusted connections and explains (using a different nomenclature) that locally, an adjusted connection is equivalent to a connection with flat adjustment if one assumes that $\frg$ is either a Lie algebra bundle or tangent algebroid; see \ref{sec:examples} for a detailed explanation and an extension of this statement.
    
    In the above-mentioned literature, the starting point was the construction of a gauge-invariant action functional that provides a dynamical principle for a Lie-algebroid-valued connection. To this end, one has to define a covariantly transforming curvature, which is essentially achieved by adjusting the local connections as discussed above. This local adjustment is obtained from a connection on the Lie algebroid which satisfies several \emph{compatibility conditions}, including a cohomological property of the connection and a certain curvature equation; flatness of the latter is essential for recovering ordinary gauge theory which is the reason why one speaks of a \emph{curved} gauge theory. 
    
    As pointed out in~\cite{Fischer:2020lri, Fischer:2021glc}\footnote{Later, Alexei Kotov and Thomas Strobl found another proof of the following argument about the exactness condition in the infinitesimal setting, which has not been published yet.}, the definition of such connections is indeed of a cohomological type: The curvature condition is in fact an exactness condition, and gauge covariance amounts to the connection being closed and its curvature being exact. This generalizes the Maurer--Cartan form and its flatness condition. The differential defining the relevant cohomology here is a natural simplicial differential on Lie groupoids as defined in~\cite[beginning of \S 1.2]{Crainic:2003:681-721}, which is an important notion for introducing Cartan, or more general, multiplicative connections. 
    
    Making use of this observation,~\cite{Fischer:2022sus} formulated a global version of local adjusted connections in the case of Lie algebroids that are Lie algebra bundles. This resulted in adjusted connections on principal groupoid bundles with structure Lie groupoids that are Lie group bundles, called \emph{(multiplicative) Yang--Mills connections}. These connections form the kinematical data of ``curved Yang--Mills gauge theory''.
    
    We note that the above-mentioned simplicial differential is formulated on general Lie groupoids. Therefore, it should be straightforward to define Lie-groupoid-based connections, following similar constructions as in~\cite{Fischer:2022sus}, essentially integrating curved Yang--Mills--Higgs gauge theories. Furthermore,~\cite[Rem.\ 6.67]{Fischer:2022sus} points out that the compatibility conditions on the connection, which ensure the closure of gauge transformations, may in fact be related to the existence of the connection. We intend to discuss all this in a separate work~\cite{future:2024ac}.
    
    Laurent-Gengoux and Fischer already observed in~\cite{Fischer:2401.05966} that the above-mentioned Yang--Mills connections help to classify (singular) foliations $\scF$. Here, these connections are induced as $\scF$-connections such that in the context of~\cite{Fischer:2401.05966} the compatibility conditions are indeed related to the existence of $\scF$. That is, the compatibility conditions are necessary and sufficient for the existence of a singular foliation in this context. In fact, as mentioned in~\cite{Fischer:2401.05966}, and as we will work out here in greater detail, the results of~\cite{Fischer:2401.05966} actually classify certain curved adjusted connections.
    
    One of the difficult tasks in defining generalizations of principal bundles with adjusted connections is the concrete description of the classifying stack $\sfB\scG^{\rm conn}$ of higher principal $\scG$-bundles with connections. Usually, this stack is captured very indirectly and on a case-by-case basis, e.g.~in terms of differential cocycles\footnote{i.e.~transition functions and gluing prescriptions for local connection forms} as in~\cite{Rist:2022hci}, and it is our goal to remedy this situation.
    
    We note that there are also constructions of Lie groupoid gauge theories without the use of adjustments~\cite{Signori:2008jd}. Instead, a \emph{Moerdijk-Mrčun condition} is imposed, i.e.~the Ehresmann connection on the principal bundle as a horizontal distribution has to be in the kernel of what one calls the moment map $\Phi$.\footnote{The action $p \racton g$ of a Lie groupoid element $g$ on a principal bundle element $p$ is defined if $\Phi(p) = \sft(g)$, where $\sft$ is the target arrow of the Lie groupoid.} This condition, however, excludes a number of useful constructions. First, it excludes structural Lie group bundles as structural Lie groupoid as developed in~\cite{Fischer:2022sus} because $\Phi$ is often the bundle projection of the principal bundle itself in such situations. In particular, we would lose many of the examples we present in \ref{sec:examples}. The Moerdijk--Mrčun condition would also exclude the previously-mentioned idea of covariantising Yang--Mills gauge theory. Covariantisation here includes two steps: first observing that a theory with a structural Lie group is equivalent to a theory with structural trivial Lie group bundle over the space-time (not just over a point) equipped with trivial connection, and second, extending the theory to non-trivial bundle structures, usually by extending the description to a richer family of connections, not just flat ones on the trivial group bundle. In particular, Yang--Mills gauge theory should be equivalent to a description with a structural trivial Lie group bundle as provided in~\cite{Fischer:2022sus}, but such a description is also excluded for the same reasons as mentioned above. Last but not least, our notion of Ehresmann connection generalises the one in~\cite{Signori:2008jd}, so that we effectively enlarge known gauge theory.
    
    Furthermore, \cite{Chatterjee:2502.02284} has a similar approach introducing an adjustment to define connections on a principal groupoid bundle. However, in their context they look at an invariance w.r.t.\ representations up to homotopy, essentially avoiding the curvature condition observed in curved Yang--Mills--Higgs theories. In a future work, one might investigate whether the work in~\cite{Chatterjee:2502.02284} is a kind of ``curved Yang--Mills--Higgs theory up to homotopy'', similar to how representations up to homotopy are related to representations. (Observe that the same research team also works on higher gauge theory~\cite{Chatterjee:2021sif}).
    
    \paragraph{Results.} This paper is the first one in a series of papers giving an explicit and useful definition of the classifying stack $\sfB\scG^{\rm conn}$ of principal bundles with connections as a differential graded higher Lie groupoid. In particular, the cocycle of a principal $\scG$-bundle over a manifold $X$ subordinate to a surjective submersion $\sigma:Y\rightarrow X$ (e.g.~a suitable cover of $X$) endowed with higher adjusted connection is simply obtained as a higher functor from the grade-shifted tangent groupoid $\rmT[1]\check\scC(\sigma)$ of the Čech groupoid $\check\scC(\sigma)$ of $\sigma$ to $\sfB\scG^{\rm conn}$ in the category of differential graded higher Lie groupoids.
    
    In this paper, we present the definition of $\sfB\scG^{\rm conn}$ for $\scG$ a Lie groupoid. We show that the required adjustment datum, i.e.~the additional algebraic information necessary for the general definition of a connection, is given by a Cartan connection\footnote{see \ref{ssec:connections} for the definitions} on $\scG$. Using this adjustment datum, we then construct an adjusted inner action groupoid $\scA^\sfW(\scG)$, which takes the above-mentioned role of $\sfB\scG^{\rm conn}$.
    
    We introduce three forms of adjustment: plain adjustments, covariant adjustments, and strict covariant adjustments, with the latter two requiring additional data. Covariant adjustments locally reproduce the notion of Lie algebroid-valued connections as developed in~\cite{Kotov:2015nuz} and~\cite{Fischer:2021glc}. As a new result at the local and infinitesimal level, we show that an adjustment here amounts to a Cartan connection on the relevant Lie algebroid.
    
    We also present a collection of detailed, interesting, and concrete examples, focusing on fundamental Lie groupoids, Lie group bundles regarded as Lie groupoids with source and target both equal the projection, and action Lie groupoids. The last class contains all the examples known from the physics literature in the context of gauge-matter field theories. In particular, we present several examples which cannot be treated without introducing adjustments, such as \ref{ex:PairGroupoidsEx}, \ref{ex:Hopf}, and \ref{ex:Hopf2}.
    
    \paragraph{Outlook.} In future work~\cite{future:2024aa}, we will discuss the extension of our definition of adjusted connections to higher bundles with a focus on principal 2-bundles. We also plan to address the total space perspective, including the notion of parallel transport for principal groupoid bundles with adjusted connection in~\cite{future:2024ac}.
    
    It is also clearly interesting to explore the phenomenological implications and use cases for our construction in more detail. We will give a first concrete application in~\cite{Fischer:2024hqy}. After this, we plan to explore if some of the hopes voiced in~\cite{Strobl:2004im} regarding phenomenological applications are justified.   
    
    \section{Lie groupoids, Lie algebroids, and differential graded manifolds}
    
    In the following, we briefly recall some definitions and results on Lie groupoids and Lie algebroids, and extend these to dg-Lie groupoids and dg-Lie algebroids. Particularly important to us is the differentiation procedure for quasi-groupoids proposed in~\cite{Severa:2006aa}. For more classical details on Lie groupoids and Lie algebroids, see also~\cite{Mackenzie:1987aa,0521499283,Moerdijk:2003bb}.
    
    \subsection{Lie groupoids and their morphisms}
    
    A \uline{groupoid} is a small category with invertible morphisms, and we write $\scG=(\sfG\rightrightarrows M)$ with $M$ and $\sfG$ the set of objects and morphisms, respectively. A \uline{Lie groupoid} is a groupoid $\scG$ with $M$ and $\sfG$ smooth manifolds, where the source and target maps $\sfs,\sft\colon\sfG\rightrightarrows M$ are surjective submersions, and the identity $\sfe\colon M\rightarrow \sfG$ and the composition $\circ\colon\sfG\fibtimes{\sfs}{\sft}\sfG\rightarrow \sfG$ are all smooth maps; as a consequence, the inverse $-^{-1}\colon\sfG\rightarrow \sfG$ is then a smooth map, too. In the following, all manifolds are smooth.
    
    \begin{example}[Čech groupoid]\label{ex:Cech_groupoid}
        An important example is the Čech groupoid $\check\scC(\sigma)$  of a surjective submersion $\sigma\colon Y\rightarrow X$ between manifolds $Y$ and $X$, which is the category $\check \scC(\sigma)=(Y^{[2]}\rightrightarrows Y)$, where 
        \begin{equation}
            Y^{[2]}\coloneqq\{(y_1,y_2)\in Y\times Y\,|\,\sigma(y_1)=\sigma(y_2)\}~,
        \end{equation}
        and the structure maps are given by 
        \begin{equation}
            \begin{gathered}
                \sft(y_1,y_2)=y_1~,~~~\sfs(y_1,y_2)=y_2~,~~~(y_1,y_2)\circ (y_2,y_3)=(y_1,y_3)~,
                \\
                \sfe_y=(y,y)~,~~~(y_1,y_2)^{-1}=(y_2,y_1)
            \end{gathered}
        \end{equation}
        for $(y_1,y_2),(y_2,y_3)\in Y^{[2]}$ and $y\in Y$.
    \end{example}
    
    \begin{example}[Pair groupoid]\label{ex:pair_groupoid}
        Another useful example is the pair groupoid
        \begin{equation}
            \scP\mathsf{air}(M)\coloneqq(M\times M\rightrightarrows M)
        \end{equation}
        of a manifold $M$ with the evident structure maps 
        \begin{equation}
            \begin{gathered}
                \sft(m_1,m_2)=m_1~,~~~\sfs(m_1,m_2)=m_2~,~~~(m_1,m_2)\circ (m_2,m_3)=(m_1,m_3)~,
                \\
                \sfe_m=(m,m)~,~~~(m_1,m_2)^{-1}=(m_2,m_1)
            \end{gathered}
        \end{equation}
        for $m,m_{1,2,3}\in M$. We have $\scP\mathsf{air}(M)=\check\scC(\sigma)$ for $\sigma=M\rightarrow *$.
    \end{example}
    
    \begin{example}[Tangent groupoid]\label{ex:tangent_groupoid}
        Given a groupoid $\scG=(\sfG\rightrightarrows M)$, we define the tangent groupoid, or rather the tangent prolongation Lie groupoid as the groupoid  
        \begin{equation}
            \rmT\scG=(\rmT\sfG\rightrightarrows\rmT M)~,
        \end{equation}
        where the structure maps are the differentials of the structure maps on $\scG$.
    \end{example}

    Recall that Lie groupoids are objects in a 2-category $\sfBibun$ of Lie groupoids, bibundles, and bibundle morphisms. A \uline{left-bibundle} between two Lie groupoids $\scG=(\sfG\rightrightarrows M)$ and $\scH=(\sfH\rightrightarrows N)$ is a manifold $B$ together with a smooth map $\tau\colon B\rightarrow M$ and a surjective submersion $\sigma\colon B\rightarrow N$, depicted as
    \begin{equation}
        \begin{tikzcd}
            \sfG \arrow[d,shift left]\arrow[d,shift right] & B \arrow[dl,"{\tau}"'] \arrow[dr,"{\sigma}"] & \sfH \arrow[d,shift left]\arrow[d,shift right]
            \\
            M & & N
        \end{tikzcd}
    \end{equation}
    together with mutually commuting left- and right-actions
    \begin{equation}
        \sfG\fibtimes{\sfs}{\tau}B\rightarrow B\eand B\fibtimes{\sigma}{\sft}\sfH\rightarrow B
    \end{equation}
    that respect morphism composition and identities in $\scG$ and $\scH$ and for which the left-action is principal, i.e.~the map $\sfG\fibtimes{\sfs}{\tau}B\rightarrow B\times_N B$, $(g,b)\mapsto (gb,b)$ is an isomorphism. Right-bibundles are defined analogously, cf., e.g.,~\cite{Metzler:0306176,Mrcun:1996aa,Blohmann:2007ez}. 
    
    The usual strict morphisms of Lie groupoids (i.e.~functors) can be lifted to bibundles. Bibundles are composable, but this composition is not associative, and therefore Lie groupoids, bibundles, and bibundle morphisms form the weak 2-category $\sfBibun$. The need for bibundles arises from the fact that the axiom of choice fails for Lie groupoids, i.e.~fully faithful and essentially surjective functors between Lie groupoids can fail to give rise to equivalences~\cite{Roberts:1101.2363}. Generically in such situations, we work with anafunctors, and a bibundle manifests the ``irreducible information'' of such an anafunctor of Lie groupoids.
    
    If a bibundle is both a left- and right-bibundle, we call the bibundle \underline{invertible} and the two groupoids $\scG$ and $\scH$ \underline{Morita-equivalent}. An important example of a pair of Morita-equivalent groupoids is a manifold $X$, trivially regarded as the Lie groupoid $(X\rightrightarrows X)$, and the Čech groupoid $\check \scC(\sigma)$ for any surjective submersion $\sigma\colon Y\rightarrow X$. We will use this to switch between local and global descriptions of principal fiber bundles. Another example is that of the action groupoid $(\sfG\times \sfG\rightrightarrows \sfG)$ of a group $\sfG$ acting on itself by left-multiplication, which is Morita-equivalent to the trivial Lie groupoid $(*\rightrightarrows *)$. For further details, see again~\cite{Metzler:0306176,Mrcun:1996aa,Blohmann:2007ez}. 
    
    \subsection{Lie algebroids and differential graded manifolds}\label{ssec:dg-manifolds_and_Lie_algebroids}
    
    Consider now a Lie groupoid $\scG$. The kernel of the map $\rmd \sft\colon \rmT\sfG\rightarrow \rmT M$ is a vector bundle over $\sfG$, and the pullback along $\sfe$ yields a vector bundle $\sfLie(\scG)=\frg\rightarrow M$, the \uline{Lie algebroid} of $\scG$. The embedding of $\frg$ into $\rmT \sfG$ also induces a Lie bracket $[-,-]$ on sections of $\frg$, which satisfies the obvious Leibniz identity for products of functions and vector fields. Furthermore, we have the \uline{anchor map} $\rho$, which is given by the restriction of $\rmd \sfs$ to $\frg$, and which is a Lie algebra morphism. 
    
    As an example, consider the pair groupoid $\scP {\rm air}(M)$ from \ref{ex:pair_groupoid}. The Lie algebroid $\sfLie(\scP {\rm air}(M))$ is the tangent algebroid $\rmT M$ with trivial anchor map.
    
    The \uline{Chevalley--Eilenberg algebra} $\sfCE(\frg)$ of a Lie algebroid $\frg\rightarrow M$ is the free graded commutative algebra\footnote{We will always assume that the vector bundle $\frg$ allows for clear dualization, e.g., because it is finite-dimensional. For more general descriptions, see e.g.~\cite{Abad:0901.0322}.}
    \begin{equation}
        \ourodot^\bullet_{C^\infty(M)}\Gamma(\frg[1]^*)~,
    \end{equation}
    where $\frg[1]$ is the vector bundle $\frg$ with fiber elements having degree\footnote{See \ref{app:graded_manifolds} for our conventions on (differential) graded manifolds.} $-1$, together with the differential 
    \begin{equation}
        \begin{aligned}
            \sfd_\sfCE \omega(\nu_0,\ldots \nu_n)& \coloneqq \sum_{i=0}^n(-1)^i\rho(\nu_i)\omega(\nu_0,\ldots,\nu_{i-1},\hat \nu_i,\nu_{i+1}, \ldots,\nu_n)
            \\
            &\hspace{1cm}+ \sum_{i<j} (-1)^{i+j}\omega([\nu_i,\nu_j],\nu_0,\ldots,\hat \nu_i,\ldots, \hat \nu_j,\ldots, \nu_n)
        \end{aligned}
    \end{equation}
    for $\nu_{0,\ldots,n}\in \Gamma(\frg[1])$, where $\hat-$ denotes an omission.
    
    In terms of some local trivialization over a patch $U\subset M$ given by coordinates $m^a$, $a=1,\ldots,\rmdim(M)$ on the base $M$ and $\xi^\alpha$, $\alpha=1,\ldots, \rmdim(\frg)-\rmdim(M)$, which freely generate the graded commutative algebra $\sfCE(\frg)$, the action of the Chevalley--Eilenberg differential $\sfd_{\sfCE}$ for a general Lie algebroid $\frg$ is given by\footnote{Here and in the following, we use the Einstein sum convention, i.e.~a sum over indices appearing twice is always implied.}
    \begin{equation}\label{eq:local_CE_algebra}
        \sfd_{\sfCE}f=\ttr^a_\alpha \xi^\alpha\frac{\dpar}{\dpar m^a} f\eand \sfd_{\sfCE} \xi^\alpha=-\tfrac12 \ttf^\alpha_{\beta\gamma}\xi^\beta\xi^\gamma
    \end{equation}
    for all $f\in C^\infty(M)$, where $\ttf^\alpha_{\beta\gamma},\ttr^a_\alpha\in C^\infty(M)$ are the structure functions of the Lie bracket on $\frg$ and the anchor map, respectively. The Leibniz rule then continues this differential uniquely to all of $C^\infty(\frg[1])$.
    
    The description~\eqref{eq:local_CE_algebra} also makes it clear that the Chevalley--Eilenberg differential can be seen as a nilquadratic vector field $Q$ of degree~$1$ on $\frg[1]$.
    
    Altogether, a Lie algebroid $\frg$ is conveniently regarded as a differential graded manifold $\frg[1]$ concentrated in degrees $0$ and $-1$, see e.g.~\cite{16298602}, and we will often speak of the Lie algebroid $\frg[1]$. As a side remark, we note that dg-manifolds without this restriction correspond to $L_\infty$-algebroids.
    
    A useful and ubiquitous example is the following.
    \begin{example}[Tangent algebroid]\label{ex:tangent_algebroid}
        The tangent algebroid $\rmT M$ of some manifold $M$ is identified with the differential graded manifold $\rmT[1]M$, i.e.~the grade-shifted tangent bundle with fiber elements concentrated in degree~$-1$, which is endowed with the de~Rham differential on $M$. The Chevalley--Eilenberg algebra of $\rmT M$ is the de~Rham complex on $M$, $\sfCE(\rmT M)=(\Omega^\bullet(M),\rmd_M)$. 
    \end{example}

    \subsection{Connections on Lie groupoids and Lie algebroids}\label{ssec:connections}
    
    The definition of principal groupoid bundle connections and their infinitesimal form requires additional structure, namely connections on Lie algebroids and Lie groupoids.
    
    Concretely, we need the notion of Cartan connection on Lie groupoids, and we will follow closely~\cite{Crainic:1210.2277}, but exchange $\ker(\rmd \sfs)$ by $\ker(\rmd \sft)$; the literature for Cartan connections and their properties is quite advanced, and historically, Cartan connections were not always known as Cartan connections, see~\cite{Abad:0911.2859, Behrend:0410255, Blaom:0404313, Blaom:0509071, Blaom:1605.04365, Crainic:1307.7979, Tang:0405378}.
    
    \begin{definition}[Multiplicative distribution]
        A multiplicative distribution on a groupoid $\scG=(\sfG\rightrightarrows M)$ is a distribution $\caD\subset \rmT\sfG$ such that $(\caD\rightrightarrows \rmT M)$ is a subgroupoid of $\rmT\scG=(\rmT\sfG\rightrightarrows \rmT M)$.
    \end{definition}
    
    \begin{definition}[Cartan connection]\label{def:Cartan_connection}
        A Cartan connection on a Lie groupoid $\scG=(\sfG\rightrightarrows M)$ is a multiplicative distribution $\caH\subset \rmT\sfG$ which is complementary to $\ker(\rmd \sft)$, i.e.~at each $g\in \sfG$ we have $\rmT_g\sfG=\caH_g\oplus \ker(\rmd \sft)_g$.
    \end{definition}
    
    As usual, a connection defines a connection 1-form. Recall that for a groupoid $\scG=(\sfG\rightrightarrows M)$, left- and right-multiplication by an element $g\in \sfG$ form maps
    \begin{equation}
        L_g\colon\sfG^{\sfs(g)}\rightarrow \sfG^{\sft(g)}\eand R_g\colon\sfG_{\sft(g)}\rightarrow \sfG_{\sfs(g)}~,
    \end{equation}
    where $\sfG^m$ and $\sfG_m$ are the sets of elements in $\sfG$ with target and source $m$, respectively. The differentials of these maps act on the tangent spaces:
    \begin{equation}
        \begin{aligned}
            \rmd L_g\colon&&\rmT\left(\sfG^{\sfs(g)}\right)=\ker(\rmd \sft)|_{\sfG^{\sfs(g)}}&\rightarrow \rmT\left(\sfG^{\sft(g)}\right)=\ker(\rmd \sft)|_{\sfG^{\sft(g)}}~,
            \\
            \rmd R_g\colon&&\rmT\left(\sfG_{\sft(g)}\right)=\ker(\rmd \sfs)|_{\sfG_{\sft(g)}}&\rightarrow \rmT\left(\sfG_{\sfs(g)}\right)=\ker(\rmd \sfs)|_{\sfG_{\sfs(g)}}~.
        \end{aligned}
    \end{equation}
    This allows us to define the vector bundle isomorphisms\footnote{Here and in the following, $\sfp_M$ always denotes the projection onto the base in some fiber bundle $E\rightarrow M$. Also, we are always very explicit in our notation of points in fiber bundles, and e.g.~$(\sfs(g),\nu)$ is the point $\nu$ in the fiber over the base point $\sfs(g)$. This is to avoid confusion when the discussion gets more complicated later, where e.g.~iterated tangent bundles appear.}
    \begin{equation}
        \begin{aligned}
            \scL\colon\sfs^*\frg=\sfG\fibtimes{\sfs}{\sfp_M}\frg&\rightarrow \ker(\rmd \sft)
            ~,
            \\
            \bigl(g;\sfs(g),\nu\bigr)&\mapsto \rmd L_{g}\left(\sfe_{\sfs(g)},\nu\right)
            ~,
        \end{aligned}
    \end{equation}
    and
    \begin{equation}
        \begin{aligned}
            \scR\colon\sft^*\sfe^*{\rm ker}(\rmd \sfs)=\sfe^*{\rm ker}(\rmd \sfs)\fibtimes{\sfp_M}{\sft}\sfG&\rightarrow \ker(\rmd \sfs)~,
            \\
            \bigl(\sft(g),\nu;g\bigr)&\mapsto \rmd R_g\left(\sfe_{\sft(g)},\nu\right)~,
        \end{aligned}
    \end{equation}
    which, together with the maps
    \begin{equation}
        \begin{aligned}
            \widehat{\rmd\sft}\colon\rmT\sfG&\rightarrow \sft^*\rmT M
            ~,~~~&
            \widehat{\rmd\sfs}\colon\rmT\sfG&\rightarrow \sfs^*\rmT M~,
            \\
            (g,\gamma)&\mapsto \bigl(g, \rmd\sft(g,\gamma)\bigr)
            ~,~~~&
            (g,\gamma)&\mapsto \bigl(g, \rmd\sfs(g,\gamma)\bigr)~,
        \end{aligned}
    \end{equation}    
    fit into two short exact sequences of vector bundles over $\sfG$, 
    \begin{equation}\label{eq:ShortSequenceForHorLifts}
        \begin{tikzcd}[row sep = small]
            \sfs^*\frg= 
            \sfs^*\sfe^*\ker(\rmd \sft) \arrow[hook]{r}{\scL} 
            & \rmT\sfG\arrow[two heads]{r}{\widehat{\rmd \sft}} 
            & \sft^*\rmT M~,
            \\
            \phantom{\sfs^*\frg=} 
            \sft^*\sfe^*\ker(\rmd \sfs)\arrow[hook]{r}{\scR} 
            & \rmT\sfG\arrow[two heads]{r}{\widehat{\rmd \sfs}} 
            & \sfs^*\rmT M~,
        \end{tikzcd}
    \end{equation}
    cf.~\cite{Abad:0911.2859} as well as~\cite[Prop.~3.5.3]{0521499283}. We denote the inverse of the bundle isomorphism $\scL:\sfs^*\frg\rightarrow \ker(\rmd \sft)\subset \rmT\sfG$ by $\vartheta$:
    \begin{equation}\label{eq:def_vartheta}
        \begin{aligned}
            \vartheta\colon\ker(\rmd \sft)&\rightarrow \sfs^*\frg~,
            \\
            (g,\gamma)&\mapsto \vartheta(g,\gamma)\coloneqq\rmd L_{g^{-1}}\left( g, \gamma\right)~.
        \end{aligned}
    \end{equation}
    With the above, we can make the following definition.
    
    \begin{definition}[Connection 1-form on $\scG$]\label{def:Conn1FormOnGroupoid}
        Let $\scG=(\sfG\rightrightarrows M)$ be a Lie groupoid equipped with a Cartan connection $\caH$, and let $\frg$ be the associated Lie algebroid. The connection 1-form $\connectiononeform$ is the following map:
        \begin{equation}
            \begin{tikzcd}
                \sfs^*\frg & \ker(\rmd \sft) \lar[bend right, "\vartheta"'above] & \rmT\sfG \lar[bend right, "\sfp^{\rm vert}"'above] \ar[ll, bend left, "\connectiononeform"]
            \end{tikzcd}
        \end{equation}
        where $\sfp^{\rm vert} \colon \rmT\sfG \to \ker(\rmd \sft)$ is the projection to the vertical subbundle implied by $\caH$.
    \end{definition}
    
    For later purposes, we note the following.
    \begin{lemma}\label{lem:connection_one_form_properties}
        Consider a Cartan connection $\caH$ on a groupoid $\scG$ and let $\vartheta^{\rm tot}$ be the corresponding connection 1-form. Then $\vartheta^{\rm tot}$ is unital in the sense that $\vartheta^{\rm tot}(\rmd\sfe(m,\mu))=0$ for all $(m,\mu)\in \rmT M$.
    \end{lemma}
    \begin{proof}
        This is a direct consequence of $\caH$ being a multiplicative distribution, containing the identities on $\rmT M$. 
    \end{proof}
    
    Also, a connection on a groupoid defines an adjoint quasi-representation~\cite{Abad:0911.2859},\footnote{The reference speaks of quasi-actions, but these are actually linear in the reference's context, which is why we decided to write ``quasi-representation'' to put an emphasis on the linearity.} which is an ordinary representation for a Cartan connection.
    \begin{definition}[Adjoint action]\label{def:adjoint_action}
        Given a Cartan connection $\caH$ on a groupoid $\scG$ with Lie algebroid $\frg$, we define the adjoint action or representation as 
        \begin{equation}
            \begin{aligned}
                \rmAd_{\caH} \colon \sfs^*\frg &\rightarrow \sft^* \frg      \\
                \left(g,\hat \nu\right)&\mapsto -\left((\sfinv^{\star})^{-1}\circ \vartheta^{\rm tot}\circ \rmd\,\sfinv \circ \scL \right) \left(g, \hat \nu \right)~,
            \end{aligned}
        \end{equation}
        where $\sfinv\colon\sfG\rightarrow \sfG$ is the groupoid inverse $\sfinv(g)=g^{-1}$, $\sfinv^\star$ is the isomorphism $\sfinv^\star\colon\sft^* \frg \cong \sfs^* \frg$ induced by $\sfinv$. We may drop the subscript $\caH$ if it is clear from the context or if the adjoint representation is independent of the Cartan connection as in the case of Lie group bundles.
    \end{definition}
    \noindent The connection, $\caH$, can also be introduced as a \uline{horizontal lift} $\sigma_\caH$, a splitting  of the short exact vector bundle sequence
    \begin{equation}
        \begin{tikzcd}
            \ker(\rmd\sft)\arrow[r] & \rmT \sfG \arrow[r] & \sft^*\rmT M \arrow[l,bend right,"\sigma_\caH"']~,
        \end{tikzcd}
    \end{equation}
    and as described in~\cite[Rem.~2.10]{Abad:0911.2859} one has simultaneously a splitting of the second sequence in~\eqref{eq:ShortSequenceForHorLifts}.
    The previously mentioned properties can then be expressed as
    \begin{equation}
        \begin{aligned}
            \sigma_\caH(g_1g_2;x, V)&=\sigma_\caH(g_1 ;x, V)\circ\sigma_\caH\bigl(g_2;\rmd \sfs(\sigma_\caH(g_1 ;x, V))\bigr)~,      \\
            \sigma_\caH(\sfe_x;x, V)&=\rmd \sfe(x, V)
        \end{aligned}
    \end{equation}
    for all $(g_1, g_2) \in \sfG \fibtimes\sfs\sft \sfG$ and $(x, V) \in \rmT M$ ($x = \sft(g_1)$ in the first equation). For any vector field $V$ on M, the corresponding \uline{projectable horizontal vector field} $\mathsf{hor}_{\caH}(V)$ is given by 
    \begin{equation}\label{eq:horizontal_lift}
        \left.\mathsf{hor}_{\caH}(V)\right|_g
        =
        \sigma_\caH\left(g; \sft(g) , V_{\sft(g)}\right)~.
    \end{equation}
    
    \begin{proposition}\label{eq:adjoint_simplified}
        We can write the adjoint action as 
        \begin{equation}
            \rmAd_{\caH}(g; \sfs(g), \nu)
            =
            \left(
            g;\sft(g), \rmd L_g \left(\sfe_{\sfs(g)}, \nu\right) \circ \sigma_\caH\left(g^{-1};\sfs(g), \rho(\nu)\right)
            \right)
        \end{equation}
        for all $(g; \sfs(g), \nu) \in \sfs^*\frg$.
    \end{proposition}
    
    \begin{proof}
        The maps $\sigma_\caH$ and $\sfp^{\rm vert}$ are related by the equation
        \begin{equation}\label{eq:spliting}
            \sfp^{\rm vert} + \sigma_\caH \circ \rmd \sft = \sfid_{\rmT \sfG}~.
        \end{equation}
        Consider the isomorphism $\inv^\star \colon \sft^* \frg \cong \sfs^* \frg$ induced by $\sfinv$. Then we have 
        \begin{align}
            &(\scL \circ \inv^\star) \left( g;\sft(g), \rmd L_g \left(\sfe_{\sfs(g)}, \nu\right) \circ \sigma_\caH\left(g^{-1};\sfs(g), \rho(\nu)\right)\right)
            \nonumber\\
            &\hspace{1cm}=
            \left(\sfe_{\sfs(g)}, \nu\right)\circ \sigma_\caH \left(g^{-1}; \sfs(g), \rho(\nu)\right)\nonumber\\
            &\hspace{1cm}=
            \left(\sfe_{\sfs(g)},\nu\right) \circ \left( \left[\left(\sfe_{\sfs(g)},\nu\right)^{-1} \circ \left(g^{-1},0\right)\right]- \sfp^{\rm vert} \left[\left(\sfe_{\sfs(g)},\nu\right)^{-1} \circ \left(g^{-1},0\right)\right]\right)
            \nonumber\\
            &\hspace{1cm}=
            \left(\sfe_{\sfs(g)},\nu\right) \circ \left(\sfe_{\sfs(g)},\nu\right)^{-1} \circ \left(g^{-1},0\right)
            - \left(\sfe_{\sfs(g)},0\right) \circ  \sfp^{\rm vert} \left[\left(\sfe_{\sfs(g)},\nu\right)^{-1} \circ \left(g^{-1},0\right)\right]
            \nonumber\\
            &\hspace{1cm}=
            - \sfp^{\rm vert} \left[\left(\sfe_{\sfs(g)},\nu\right)^{-1} \circ \left(g^{-1},0\right)\right]
            \nonumber\\
            &\hspace{1cm}=
            \scL \circ \inv^\star\circ \rmAd_{\caH}(g, ;\sfs(g),\nu)~,
        \end{align}
        where we used the canonical tangent groupoid structure of the tangent groupoid $\rmT\scG$, cf.~\ref{ex:tangent_groupoid}, denoting its composition also by $\circ$, and we used~\eqref{eq:spliting} in the second equation. As $\scL$ and $\inv^\star$ are isomorphisms, this proves the proposition. 
    \end{proof}
    
    \begin{remark}[Simplified notation]
        Most of the time, the base point information in the pullback structures and of all the natural isomorphisms is implied by the context and the structure of the corresponding equation. In such a case one may simply write the adjoint action as 
        \begin{equation}
            \rmAd_{\caH}(g, \nu)
            =
            - \left(\left.\vartheta^{\rm tot}\right|_{g^{-1}} \circ \rmd R_{g^{-1}} \right)\bigl(\rmd\, \sfinv (\nu)\bigr)
            =
            - \left(\rmd L_{g} \circ \sfp^{\rm vert} \circ \rmd R_{g^{-1}} \right)\bigl(\rmd\, \sfinv (\nu)\bigr)
        \end{equation}
        for all $g \in \sfG$ and $\nu \in \frg_{\sfs(g)}$. Works like~\cite{Abad:0911.2859} also omit denoting $\rmd\, \sfinv$, and observe that one recovers the ordinary adjoint representation, if $\scG$ is a Lie group bundle, regardless of the chosen Cartan connection.
    \end{remark}
    
    A connection on a groupoid $\scG$ also induces a vector bundle connection on the corresponding Lie algebroid $\frg$. Recall that a vector bundle connection on a Lie algebroid $\frg$ is a bilinear map
    \begin{equation}
        \nabla:\frX(M) \times \Gamma(\frg)\rightarrow \Gamma(\frg)~, ~~~(V, \nu) \mapsto \nabla_V \nu~,
    \end{equation}
    such that the identities $\nabla_{fV} \nu=f\nabla_V \nu$ and $\nabla_{V} (f\nu) =f\nabla_V \nu+V(f) ~ \nu$ hold for any smooth function $f$ on $M$.
    
    This induces the \uline{basic $\frg$-connection on $\frg$}\footnote{This basic connection can be seen as the infinitesimal form of the adjoint representation of the integrating Lie groupoid $\scG$ on $\frg$.},
    \begin{equation}\label{eq:basic_connection_1}
        \nabla^{\rm bas}_{\nu_2}{\nu_1}\coloneqq\nabla_{\rho(\nu_1)}\nu_2+[\nu_2,\nu_1]~,
    \end{equation}
    as well as the \uline{basic $\frg$-connection on $\rmT M$}\footnote{This can be seen as the infinitesimal version of the adjoint representation of $\scG$ on $\rmT M$.},
    \begin{equation}
        \nabla_{\nu}^{\rm bas}V\coloneqq \rho(\nabla_V\nu)+[\rho(\nu),V]
    \end{equation}
    for $\nu,\nu_{1,2}\in \Gamma(\frg)$ and $V\in \frX(M)$. The \uline{basic curvature} of $\nabla$ is then defined as 
    \begin{equation}\label{eq:basic_curvature}
        R_\nabla^{\rm bas}(\nu_1,\nu_2)(V)\coloneqq \nabla_V([\nu_1,\nu_2])-[\nabla_V\nu_1,\nu_2]-[\nu_1,\nabla_V\nu_2]-\nabla_{\nabla_{\nu_2}^{\rm bas}V}\nu_1+\nabla_{\nabla_{\nu_1}^{\rm bas}V}\nu_2
    \end{equation}
    for $\nu_1,\nu_2\in \Gamma(\frg)$ and $V\in \frX (M)$, see~\cite[Def.~2.9]{Abad:0901.0319} and references therein for details.
    
    \begin{proposition}[{\cite[Thm.~3.11]{Crainic:1210.2277}}]\label{prop:associated_Lie_algebroid_connection}
        Any Cartan connection $\caH$ on a Lie groupoid $\scG$ induces a vector bundle connection on the Lie algebroid $\frg$ via
        \begin{equation}
            \left.\nabla_V \nu\right|_x\coloneqq \left.[\mathsf{hor}_\caH(V), \nu_l]\right|_{\sfe_x}
        \end{equation}
        for all $\nu \in \Gamma(\frg)$ and $V \in \frX(M)$, where $\nu_l$ is the left-invariant vector field on $\sfG$ induced by $\nu$, $\nu_l|_g = \scL(g, \nu_{\sfs(g)})$.
        This connection is Cartan, i.e.\ its basic curvature vanishes.
    \end{proposition}
    
    \begin{remark}
        Analogous to \cite{Fischer:2022sus}, one can express the invariance and the previous formulas also with the notion of the \uline{modified right-pushforward} which may be more suitable for the dynamics in the related gauge theory.
    \end{remark}
    
    \subsection{Ševera differentiation of quasi-groupoids}\label{ssec:Severa_differentiation}
    
    A crucial ingredient in our discussion is Ševera's differentiation of quasi-groupoids to $L_\infty$-algebroids~\cite{Severa:2006aa}, see also~\cite{Jurco:2016qwv} and in particular~\cite{Li:2014} for details. Recall that there is a general Lie theory for $\infty$- or quasi-Lie groupoids, i.e.~Kan simplicial manifolds, and the $L_\infty$-algebroid of a quasi-groupoid can be seen as a representation of the 1-jet of a simplicial manifold.
    
    Let us give the specialization of this differentiation procedure to the case of Lie group\-oids.\footnote{In the general construction, there is an issue regarding representability, which has been addressed in~\cite{Li:2014}. Moreover, in all our cases, we can explicitly see that the internal homomorphisms are representable.} This differentiation makes use of a special class of surjective submersions.
    \begin{definition}[Tangent submersion]\label{def:tangent_submersion}
        Let $\Theta\coloneqq\IR[-1]$. The tangent submersion $\sigma_\Theta(\caX)$ for $\caX$ a dg-manifold is the evident projection $\sigma_\Theta(\caX):\caX\times \Theta\rightarrow \caX$. We write $\sigma_\Theta(*)$ for the trivial projection $\Theta\rightarrow *$, where $*$ is the singleton set.
    \end{definition}
    
    The differentiation itself now works as follows.
    \begin{proposition}[\cite{Severa:2006aa}]\label{prop:Severa_diff}
        The grade-shifted Lie algebroid $\sfLie(\scG)[1]$ of a Lie groupoid $\scG$ is the internal hom\footnote{see \ref{app:graded_manifolds} for details} $\ihom(\check \scC(\sigma_\Theta(*)),\scG)$ in the category of graded manifolds, which naturally carries the structure of a differential graded manifold. The differential in the corresponding Chevalley--Eilenberg algebra is induced by the canonical action of $\ihom(\Theta,\Theta)$.
    \end{proposition}
    
    We briefly demonstrate this proposition by computation, as similar computations will become relevant later. We start by considering functors of graded Lie groupoids 
    \begin{equation}\label{eq:param_functors}
        \Phi:\check\scC(\sigma_\Theta(\caX))\rightarrow \scG=(\sfG\rightrightarrows M)
    \end{equation}
    for $\caX$ some graded manifold. Such functors contain maps\footnote{Later in \ref{sec:principal_groupoid_bundles}, we will identify these maps with the description of principal $\scG$-bundles over $\caX$ subordinate to $\sigma_\Theta(\caX)$ in terms of differential cocycles.} 
    \begin{equation}
        F_1:\caX\times \Theta\times \Theta \rightarrow \sfG
        \eand 
        F_0:\caX\times \Theta \rightarrow M
    \end{equation}
    such that 
    \begin{equation}
        \begin{gathered}
            F_1(x,\theta_1,\theta_3)=F_1(x,\theta_1,\theta_2)F_1(x,\theta_2,\theta_3)~,
            \\
            \sft(F_1(x,\theta_1,\theta_2))=F_0(x,\theta_1)~,\eand \sfs(F_1(x,\theta_1,\theta_2))=F_0(x,\theta_2)~.
        \end{gathered}
    \end{equation}
    As a consequence of these relations, we find that 
    \begin{equation}
        \begin{aligned}
            \sfe_{m(x)}&=F_1(x,0,0)=F_1(x,0,\theta)F_1(x,\theta,0)~,
            \\
            F_1(x,\theta_1,\theta_2)&=F_1(x,\theta_1,0)F_1(x,0,\theta_2)
        \end{aligned}
    \end{equation}
    for all $x\in \caX$ and $\theta_{1,2}\in \Theta$, where $m \in C^\infty(\caX; M)$, $x \mapsto F_0(x, 0)$. Therefore, the functor $\Phi$ is fully defined by a map
    \begin{equation}
        \check F:\caX\times \Theta\rightarrow \sfG
    \end{equation}
    with 
    \begin{equation}
        \check F(x,\theta)=F_1(x,0,\theta)~,
    \end{equation}
    which we can write suggestively as\footnote{That is, using the local diffeomorphism of $\sfG$ with $\rmT[1]_{\sfe_{\hat m(x)}}\sfG$ in a neighborhood of $\sfe_{\check m(x)}$; or alternatively for this case, making use of the local diffeomorphism induced by the exponential map for Lie groupoids as in~\cite[\S 3.6]{0521499283}.}
    \begin{equation}
        F_1(x,0,\theta)=\sfe_{m(x)}+\theta \nu(x)
    \end{equation}
    with $\nu (x)\in \rmT[1]_{\sfe_{\hat m(x, 0)}}\sfG$ and $\sfd\sft(\nu(x))=0$. This makes it clear that the functor $\Phi$ is representable, i.e.~it can be seen as a morphism of graded manifolds from $\caX$ to $\sfLie(\scG)[1]$. Moreover, the action of the Chevalley--Eilenberg differential is given by the component-wise action of the operation
    \begin{equation}\label{eq:differential}
        F_1(x,\theta_1,\theta_2)\mapsto \frac{\rmd}{\rmd \eps}F_1(x,\theta_1+\eps,\theta_2+\eps)~,
    \end{equation}
    and for the detailed, general computation in coordinates, see e.g.~\cite[Section 8.4.3]{Li:2014}. In this sense $\sfLie(\scG)[1]=\ihom(\check \scC(\sigma_\Theta(*)),\scG)$, and this generalizes in an evident manner to quasi-groupoids.
    
    \subsection{Differential graded Lie groupoids and Lie algebroids}
    
    For Ševera differentiation, it is very convenient to work directly in the category $\catdgMfd$ of differential graded (dg-)manifolds, because here, the differential is already included in the representation of the functor.\footnote{For a discussion of super, i.e.~$\IZ_2$-graded, Lie groupoids and Lie algebroids, see~\cite{Mehta:0605356}.} We briefly develop this picture in the following.
    \begin{definition}[Grade-shifted tangent bundle]\label{def:grade-shifted_tangent_bundle}
        The odd line $\Theta=\IR[-1]$, as a dg-manifold, is naturally endowed with the differential $\sfd$
        \begin{equation}\label{eq:theta_differential}
            \sfd(a+\theta b)\coloneqq b~,
        \end{equation}
        where $a,b\in \IR$ and $\theta\in \IR[-1]^*$. The grade-shifted tangent bundle $\rmT[1]\caM$ of a dg-manifold $\caM$ is a representation of the functor $\catdgMfd(-\times \Theta,\caM):\catdgMfd^{\rm op}\rightarrow \catSet$.
    \end{definition}
    \noindent Given a split dg-manifold $E=(E\rightarrow M)$, i.e.~a graded vector bundle $E$ over a manifold $M$, the grade-shifted tangent bundle $\rmT[1]E$ is indeed given by the tangent bundle of $E$ with the grading of the tangent fibers shifted appropriately. Moreover, the differential on $\rmT[1]E$ acts as the sum of the original differential on $E$ and the de~Rham differential on $E$, which continues uniquely to all of $\rmT[1]E$. 
    
    To work fully internally to the category $\catdgMfd$ of differential graded manifolds, we need the following definition.
    \begin{definition}[Dg-Lie groupoid and algebroid]\label{def:dgLiegroupoids}
        A dg-Lie groupoid is a groupoid internal to the category of dg-manifolds $\catdgMfd$ with source and target maps surjective submersions.\footnote{For the definition of surjective submersions for dg-manifolds, see e.g.~\cite{Vysoky:2021wox}.} A dg-Lie algebroid is a dg-vector bundle (in the evident sense) with fibers graded vector spaces concentrated in non-negative degrees. Together with the evident morphisms, dg-Lie groupoids form the category $\catdgGrpd$, and together with their weak and higher morphisms, the 2-category $\CatdgGrpd$.
    \end{definition}
    
    \ref{def:grade-shifted_tangent_bundle} can now be extended to dg-Lie groupoids.
    \begin{definition}[Tangent prolongation dg-Lie groupoid and Tangent prolongation dg-Lie algebroid]\label{def:tangent_dg_Lie_algebroid}
        Given a dg-Lie groupoid $\scG=(\caG\rightrightarrows \caM)$ with $\caG$ and $\caM$ dg-manifolds, its tangent prolongation dg-Lie groupoid $\rmT[1]\scG$ is the dg-Lie groupoid
        \begin{equation}
            \rmT[1]\scG=(\rmT[1]\caG\rightrightarrows \rmT[1]\caM)~,
        \end{equation}
        and the structure maps are the differentials of the structure maps in $\scG$.
        
        The tangent prolongation dg-Lie algebroid of a dg-Lie algebroid $\frg[1]$ over a dg-manifold $\caM$ is the dg-Lie algebroid\footnote{This graded manifold should be regarded as a Lie 2-algebroid, since a local coordinate description involves coordinate functions in degrees $0,1,2$. It describes the inner derivations of $\frg$.} $\rmT[1]\frg[1]$ over $\rmT[1] \caM$.
    \end{definition}
    
    \begin{remark}
        Following~\cite{0521499283}, we will usually just speak of a tangent groupoid $\rmT[1]\scG$, however, to avoid confusion with the tangent algebroid $\rmT[1]\frg[1] \to \frg[1]$ we will not omit ``prolongation'' when speaking of the algebroid structure of $\rmT[1]\frg[1]$ over $\rmT[1]\caM$.
    \end{remark}
    
    An example of a dg-Lie groupoid which is important for our discussion is the following.
    \begin{example}[Tangent Čech groupoid]\label{ex:tangent_Cech_groupoid}
        The tangent Čech groupoid of a surjective submersion $\sigma:Y\rightarrow X$ of ordinary manifolds is the dg-Lie groupoid
        \begin{equation}
            \rmT[1]\check \scC(\sigma)=\big(~\rmT[1]Y^{[2]}\rightrightarrows\rmT[1]Y~\big)~,
        \end{equation}
        where all the structure maps in $\rmT[1]\check \scC(\sigma)$ are simply the differentials of the structure maps in $\check \scC(\sigma)$.
    \end{example}
    Another useful example is the following.
    \begin{example}[Odd-line Čech groupoid]
        The differential~\eqref{eq:theta_differential} on $\Theta$ extends to $\Theta^{[2]}$ as 
        \begin{equation}
            \sfd (a_1+\theta b_1,a_2+\theta b_2)\coloneqq(b_1,b_2)~,
        \end{equation}
        cf.~\eqref{eq:differential}. This renders the Čech groupoid $\check\scC(\sigma_\Theta(*))$ a dg-Lie groupoid.
    \end{example}
    
    As one would expect, Ševera differentiation of Lie groupoids as discussed in \ref{ssec:Severa_differentiation} then straightforwardly extends to $\catdgGrpd$, the category of dg-Lie groupoids from \ref{def:dgLiegroupoids}.
    
    \begin{definition}[Dg-Lie algebroid of a dg-Lie groupoid]\label{def:dg-Lie-algebroid}
        The dg-Lie algebroid $\sfLie(\scG)[1]$ of a dg-Lie groupoid $\scG$ is the dg-manifold representing the internal homomorphisms\linebreak $\ihom(\check \scC(\sigma_\Theta(*)),\scG)$ in the category of dg-manifolds.
    \end{definition}
    \noindent In other words, $\sfLie(\scG)[1]$ is a representation of the functor 
    \begin{equation}
        \catdgGrpd(\check \scC(\sigma_\Theta(-)),\scG):\catdgMfd^{\rm op}\rightarrow \catSet~,
    \end{equation}
    sending a dg-manifold $\scM$ to the set of dg-functors from $\check \scC(\sigma_\Theta(\scM))$ to $\scG$.
    
    We then have the following, evident link between the tangent dg-Lie groupoid and the tangent prolongation Lie algebroid: 
    \begin{lemma}
        Consider a tangent dg-Lie groupoid $\rmT[1]\scG=(\rmT[1]\sfG\rightrightarrows \rmT[1]M)$. Its dg-Lie algebroid is $\sfLie(\rmT[1]\scG)[1]=\rmT[1]\sfLie(\scG)[1]$.
    \end{lemma}
    
    We will also need the notion of action of a dg-Lie groupoid onto a dg-manifold, which is the straightforward extension of a Lie groupoid action on a manifold, cf.~\cite[Def.~1.6.1]{0521499283}.
    \begin{definition}[Right groupoid action]\label{def:right_action}
        A right action of a dg-Lie groupoid \(\scG=(\sfG\rightrightarrows M)\) on a dg-manifold \(N\) along a dg-map \(\psi\colon N\to M\) is a dg-map
        \begin{subequations}\label{eq:relations_right_action}
            \begin{equation}
                \begin{aligned}
                    (-\racton-)\colon N\fibtimes\psi\sft\sfG &\to N \\
                    (p,g)&\mapsto p\racton g
                \end{aligned}
            \end{equation}
            such that
            \begin{equation}\label{eq:right_action_explicit_relations}
                \psi(p\racton g_1) = \sfs(g_1)
                ~,~~~
                p\racton\sfe_{\psi(p)} = p
                ~,~~~
                (p\racton g_1)\racton g_2=p\racton (g_1g_2)
            \end{equation}
        \end{subequations}        
        for all $(p,g_1,g_2)\in N\fibtimes\psi\sft\sfG\fibtimes\sfs\sft\sfG$.
    \end{definition}
    
    Given such a right action, we can construct the corresponding action groupoid, cf., e.g., \cite[Def.\ 1.6.10]{0521499283} for the case of ordinary groupoids. 
    \begin{definition}[Action groupoid]\label{def:action_groupoid}
        Consider a right action $(\racton,\psi)$ of a dg-Lie groupoid $\scG=(\sfG\rightrightarrows M)$ on a dg-manifold $N$. The corresponding action groupoid is given by 
        \begin{subequations}
            \begin{equation}
                \scK\coloneqq \Big(~N\fibtimes{\psi}{\sft}\sfG~\rightrightarrows~N~\Big)
            \end{equation}
            with structure maps
            \begin{equation}
                \begin{gathered}
                    \sft(n,g_1)=n~,~~~\sfs(n,g_1)=n\racton g_1~,~~~(n,g_1)\circ (n\racton g_1,g_2)=(n,g_1\circ g_2)~,
                    \\
                    \sfe_n=\left(n,\sfe_{\psi(n)}\right)~,~~~(n,g_1)^{-1}=\left(n \racton g_1,g^{-1}_1\right)
                \end{gathered}
            \end{equation}
            for $g_{1,2}\in \scG$ and $n\in N$ with $\psi(n)=\sft(g)$ and $\sfs(g_1)=\sft(g_2)$.            
        \end{subequations}
    \end{definition}
    \noindent Similarly, one can construct the action groupoid for left actions $(\acton,\psi)$. Infinitesimally, an action groupoid differentiates to an action Lie algebroid.

    \subsection{The Weil algebra of a Lie algebroid}\label{ssec:Weil_algebra}
    
    An important extension of the Chevalley--Eilenberg algebra $\sfCE(\frg)$ of a Lie algebroid $\frg$ is its \uline{Weil algebra} $\sfW(\frg)$, which is the Chevalley--Eilenberg algebra of the tangent prolongation Lie algebroid\footnote{as given in \ref{def:tangent_dg_Lie_algebroid}, for an ordinary Lie algebroid $\frg$ trivially regarded as a dg-Lie algebroid}:
    \begin{equation}
        \sfW(\frg)=\sfCE(\rmT[1]\frg)~.
    \end{equation}
    The latter is the algebra of functions $C^\infty(\rmT[1]\frg[1])$, together with the differential $\sfd_\sfW$ given by the sum of the Chevalley--Eilenberg and de~Rham differentials on $C^\infty(\frg)$, which uniquely extends to $C^\infty(\rmT[1]\frg[1])$. Explicitly, we have 
    \begin{equation}
        \begin{aligned}
            \sfd_\sfW f&=\rmd f \circ\rho+ \rmd f~,
            \\
            \sfd_\sfW \omega&=\sfd_{\sfCE} \omega+ \nabla^*\omega+ \omega[1]~,
            \\
            \sfd_\sfW (\rmd f)&=-\left(\rmd f \circ\rho\right)[1]- \nabla^*\left(\rmd f \circ\rho\right)~,
            \\
            \sfd_\sfW \omega[1]&= -\sfd_\sfW(\rmd_{\sfCE}\omega+\nabla^*\omega)~
        \end{aligned}
    \end{equation}
    for all $f \in C^\infty(M)$, $\omega\in \Gamma(\frg[1]^*)$, where $[1]$ is the degree shift $\Gamma(\frg[1]^*)\to \Gamma(\frg[2]^*)$ and $\nabla^*$ is the induced connection on the dual vector bundle defined by 
    \begin{equation}
        \left( \nabla_V^* \omega\right)(\alpha)= -\omega(\nabla_V\alpha)+V(\omega(\alpha))~,
    \end{equation}
    for vector fields $V$ and $\alpha \in \Gamma(\frg)$.
    
    A more detailed discussion is found in~\cite{Abad:0901.0322}, see also~\cite{Mehta:0605356,Lean:2001.01101} for Lie $n$-algebroids. 
    Truncating the latter description to $n=1$, we regard the graded commutative algebra $C^\infty(\rmT[1]\frg[1])$ as the algebra of differential forms $\Omega^\bullet(\frg[1])$. As explained above, the Chevalley--Eilenberg differential of $\frg$ can be regarded as a nilquadratic vector field $Q$ of degree~$1$ on $\frg[1]$, and the differential in $\sfW(\frg)$ is defined as the sum of the de~Rham differential on $\frg[1]$ and the Lie derivative $\caL_Q$.
    
    For our purposes, we will require the Weil algebra together with a splitting into Chevalley--Eilenberg and tangent generators. That is, we require an isomorphism
    \begin{equation}\label{eq:bundle_isomorphism}
        \rmT[1]\frg[1] \cong \rmT[1]M\oplus \frg[1]\oplus \frg[2]
    \end{equation}
    for the Lie algebroid $\frg=(\frg\rightarrow M)$, where we view $\rmT[1]\frg[1]$ as a bundle over $M$ with projection given by the composition of the projections $\rmT[1]\frg[1] \to \frg[1]$ and $\frg[1] \to M$, and such an isomorphism is provided by a connection $\nabla$ on $\frg$. More precisely, such a connection yields a splitting of the tangent bundle into horizontal and vertical bundles, $\rmT\frg=H\frg\oplus V\frg$. The connection $\nabla$ induces a vector bundle isomorphism $H\frg\cong \sfp_M^*\rmT M$, where $\sfp_M$ is the bundle projection $\frg \to M$; and because $\frg$ is a vector bundle, we have the canonical isomorphism $V\frg \cong \sfp_M^*\frg$. Thus, $\rmT\frg \cong \sfp_M^*\rmT M \oplus \sfp_M^*\frg$ via $\nabla$. Equivalently, $\rmT\frg \cong \rmT M \times_M \frg \times_M \frg$, the middle factor as the aligned base point factor of $\sfp_M^*\rmT M$ and $\sfp_M^*\frg$ over $\frg$, thus, with gradings $\rmT[1]\frg[1] \cong \rmT[1] M \times_M \frg[1] \times_M \frg[2]$. Now, the composition of the natural bundle projections $\rmT[1] \frg[1] \to \frg$ and $\frg[1] \to M$ equips $\rmT[1]\frg[1]$ with a canonical vector bundle structure over $M$ via that isomorphism induced by $\nabla$. By construction this aligns with the vector bundle structure of $\rmT[1]M\oplus \frg[1]\oplus \frg[2]$. For later purposes, we define the projection
    \begin{equation}\label{eq:p_curv}
        \sfp_{\rm curv}\colon\rmT[1]\frg[1]\rightarrow \rmT[1]M\oplus \frg[2]~.
    \end{equation}
    
    The Weil algebra $\sfW(\frg)$ can now be written with respect to sections of the dual bundle $\rmT[1]^*M\oplus \frg[1]^*\oplus\frg[2]^*$, as explained (without the use of graded manifolds) in~\cite{Abad:0901.0322}. Explicitly we have 
    \begin{equation}
        \sfW_\nabla(\frg)\cong \ourodot^\bullet\Gamma\Big(\rmT[1]^*M\oplus\frg[1]^*\oplus\frg[2]^*\Big)
    \end{equation}
    for the underlying graded commutative algebra.
    
    It is convenient to describe the differential $\sfd_\sfW$ in terms of local coordinates on a chart $U_M$ of $M$ over which the bundle $\rmT[1]\frg[1]$ trivializes. If $m^a$ are coordinates on $U_M$ and $\bar m^a$, $\xi^\alpha$, and $\tilde{\bar{\xi}}^\alpha$ are the coordinates in the fibers\footnote{For now we choose a canonical frame, that is, $\bar{m}^a$ related to $m^a$ in the usual way, and both, $\xi^a$ and $\tilde{\bar\xi}^a$, correspond to the same frame in $\frg$.} of $\rmT[1]M$, $\frg[1]$, and $\frg[2]$, respectively, then the differential reads as 
    \begin{subequations}\label{eq:coordinate_changed_pre_Weil}
        \begin{equation}
            \begin{aligned}
                \sfd_\sfW f&=\left(\ttr^a_\alpha \xi^\alpha+\bar{m}^a\right)\frac{\dpar}{\dpar m^a} f~,
                \\
                \sfd_\sfW \xi^\alpha&=-\tfrac12 \ttf^\alpha_{\beta\gamma}\xi^\beta\xi^\gamma+\tilde{\bar{\xi}}^\alpha-
                \omega_{a\beta}^\alpha \bar m^a \xi^\beta
                ~,
                \\
                \sfd_\sfW\bar m^a&=-\ttr^a_\alpha \left(\tilde{\bar{\xi}}^\alpha-
                \omega_{b\beta}^\alpha \bar m^b \xi^\beta\right)+\frac{\dpar\ttr^a_\alpha}{\dpar m^b} \xi^\alpha \bar m^b~,
                \\
                \sfd_\sfW \tilde{\bar{\xi}}^\alpha&=-\left(\ttf^\alpha_{\beta\gamma}+\ttr^a_\gamma \omega^\alpha_{a\beta}\right)\xi^\beta \tilde{\bar{\xi}}^\gamma-\omega^\alpha_{a\beta}\bar m^a \tilde{\bar{\xi}}^\beta+\tfrac12 \left(R_\nabla^{\rm bas}\right)^\alpha_{\beta\gamma a}\xi^\beta\xi^\gamma\bar m^a+\tfrac12 \left(R_\nabla\right)^\alpha_{ab\beta}\bar m^a\bar m^b\xi^\beta~
            \end{aligned}
        \end{equation}
        for all $f \in C^\infty(U_M)$,
        where $\ttr^a_\alpha$, $\ttf^\alpha_{\beta\gamma}$, and $\omega^\alpha_{a\gamma}$ are the structure functions of the anchor map $\rho$, the Lie bracket $[-,-]$, and the local connection one-form of $\nabla$, respectively. Moreover,
        \begin{equation}
            \begin{aligned}
                (R_\nabla)^\alpha_{ab\beta}
                &=
                \partial_a\omega^\alpha_{b \beta}
                - \partial_b\omega^\alpha_{a \beta}
                + \omega^\alpha_{a \gamma}\omega^\gamma_{b \beta} 
                - \omega^\alpha_{b \gamma}\omega^\gamma_{a \beta}~,
                \\
                \left(R^{\rm bas}_\nabla\right)^\alpha_{\beta\gamma a}
                &=
                \partial_a\ttf^\alpha_{\beta\gamma}
                + \ttf^\delta_{\beta\gamma} \omega_{a \delta}^\alpha 
                + 2\omega^\alpha_{b [\beta} \partial_a\ttr^b_{\gamma]} 
                - 2\omega^\alpha_{b [\beta}\omega_{a|\gamma]}^\delta\ttr^b_\delta  
                + 2\left(\partial_b\omega^\alpha_{a [\beta}\right)\ttr^b_{\gamma]}
                - 2\ttf^\alpha_{[\beta|\delta}\omega^\delta_{a|\gamma]}
            \end{aligned}
        \end{equation}
    \end{subequations}
    are the components of the curvature $R_\nabla$ and the basic curvature~\eqref{eq:basic_curvature} of $\nabla$.
    
    In the definition of the bundle isomorphism~\eqref{eq:bundle_isomorphism}, there is further freedom which we can use to make our later discussion more elegant. This freedom is parameterized by a $\frg$-valued 2-form $\zeta\in \Omega^2(M,\frg)$, and amounts to a coordinate change
    \begin{equation}
        \left(m^a,\bar m^a,\xi^\alpha,\tilde{\bar{\xi}}^\alpha\right)\mapsto \left(\tilde m^a,\tilde{\bar{m}}^a,\tilde \xi^\alpha,\bar{\xi}^\alpha\right)=\left(m^a,\bar m^a,\xi^\alpha,\tilde{\bar{\xi}}^\alpha+\tfrac{1}{2}\zeta^\alpha_{ab}\bar m^a\bar m^b\right)~.
    \end{equation}
    
    We denote the resulting form of the Weil algebra by $\sfW_{(\nabla,\zeta)}$, and its underlying commutative graded algebra agrees with that of $\sfW_\nabla$. The differential acts as follows:
    \begin{subequations}\label{eq:coordinate_changed_Weil}
        \begin{equation}
            \begin{aligned}
                \sfd_\sfW f&=\left(\ttr^a_\alpha \xi^\alpha+\bar{m}^a\right)\frac{\dpar}{\dpar m^a} f~,
                \\
                \sfd_\sfW  \xi^\alpha&=-\tfrac12 \ttf^\alpha_{\beta\gamma}\xi^\beta\xi^\gamma+\bar{\xi}^\alpha-
                \omega_{a\beta}^\alpha \bar m^a \xi^\beta-\tfrac12 \zeta^\alpha_{ab} \bar m^a \bar m^b
                ~,
                \\
                \sfd_\sfW \bar m^a&=-\ttr^a_\alpha \left(\bar{\xi}^\alpha-
                \omega_{b\beta}^\alpha \bar m^b \xi^\beta-\tfrac12 \zeta^\alpha_{bc} \bar m^b \bar m^c\right)+\frac{\dpar\ttr^a_\alpha}{\dpar m^b} \xi^\alpha \bar m^b~,
                \\
                \sfd_\sfW \bar{\xi}^\alpha&=-\left(\ttf^\alpha_{\beta\gamma}+\ttr^a_\gamma \omega^\alpha_{a\beta}\right)\xi^\beta \bar{\xi}^\gamma-\left(\omega^\alpha_{a\beta}-\zeta^\alpha_{ab}\ttr^b_\beta\right)\bar m^a \bar{\xi}^\beta
                +\tfrac12 \left(R_\nabla^{\rm bas}\right)^\alpha_{\beta\gamma a}\xi^\beta\xi^\gamma\bar m^a
                \\
                &~~~~+\left(\tfrac16\left(\rmd^\nabla\zeta\right)^\alpha_{abc}-\tfrac12 \zeta^\alpha_{ad}\ttr^d_\beta \zeta^\beta_{bc}\right)\bar m^a\bar m^b\bar m^c+\tfrac12 \left(R_\nabla+\nabla^{\rm bas}\zeta\right)^\alpha_{ab\beta}\bar m^a\bar m^b\xi^\beta
            \end{aligned}
        \end{equation}
        for $f\in C^\infty(M)$, where we used the same component functions as in~\eqref{eq:coordinate_changed_pre_Weil} as well as 
        \begin{equation}
            \begin{aligned}
                \left(\rmd^\nabla\zeta\right)^\alpha_{abc}
                &=
                \partial_a\zeta^\alpha_{bc}
                +\partial_b\zeta^\alpha_{ca}
                +\partial_c\zeta^\alpha_{ab}
                +\omega^\alpha_{a \beta}\zeta^{\beta}_{bc}
                +\omega^\alpha_{b \beta}\zeta^{\beta}_{ca}
                +\omega^\alpha_{c \beta}\zeta^{\beta}_{ab}~,
                \\
                \left(\nabla^\mathrm{bas}\zeta\right)^\alpha_{ab\beta}
                &=
                \ttf^\alpha_{\beta\gamma} \zeta^\gamma_{ab}
                + \rho_\beta^c \partial_c \zeta^\alpha_{ab}
                + \omega^\alpha_{c \beta} \rho^c_\gamma \zeta^\gamma_{ab}
                + 2 \zeta^\alpha_{[a|c} 
                \left( 
                \partial_{b]} \rho^c_\beta - \rho^c_\gamma \omega^\gamma_{b] \beta}
                \right)~,
            \end{aligned}
        \end{equation}
        where $\nabla^{\rm bas}$ denotes the basic connection~\eqref{eq:basic_connection_1}.
    \end{subequations}    
    
    There are now two extreme examples. First, a Lie algebroid $\frg$ over a point is a Lie algebra, and its Weil algebra is indeed the Weil algebra of this Lie algebra. Note that here, the connection $\nabla$ and $\zeta$ are necessarily trivial. Second, we have the case when the fiber has rank 0, which again forces $\nabla$ and $\zeta$ to be trivial. This is a rephrasing of \ref{ex:tangent_algebroid}:
    \begin{example}\label{ex:Weil_algebra_of_manifold}
        We can trivially regard a manifold $X$ as a rank 0 Lie algebroid $X$ over itself. Its Weil algebra $\sfW(X)$ is then the Chevalley--Eilenberg algebra of $\rmT[1]X$ and hence the de~Rham complex $(\Omega^\bullet(X),\rmd_X)$ on $X$.
    \end{example}
    
    Before we close this section, let us briefly explain the relation between the Chevalley--Eilenberg algebra and the Weil algebra. There are two natural bundle maps between $\rmT[1]\frg[1]$ and $\frg[1]$: the bundle projection and the embedding as zero section. The former is not an algebroid morphism, and therefore merely translates to a morphism of graded algebras $\sfCE(\frg)\hookrightarrow \sfW(\frg)$ which fails to respect the differentials. The latter is an algebroid morphism and dually, it leads to the natural projection of dg-commutative algebras
    \begin{equation}\label{eq:Weil_to_CE_projection}
        \sfpr:\sfW(\frg)\rightarrow \sfCE(\frg)~,
    \end{equation}
    by setting all generators not contained in $\Gamma(\frg[1]^*)$ to zero. This projection with $\frg$ an $L_\infty$-algebra is heavily used e.g.~in~\cite{Sati:2008eg}.
    
    \section{Local connections on principal groupoid bundles}\label{sec:local_description}
    
    Before introducing the global picture, it is helpful to reproduce the local picture developed in~\cite{Strobl:2004im} and~\cite{Kotov:2015nuz} from a slightly different point of view. This will provide a clear motivation and intuition of the global constructions of later sections. Moreover, this local description of connections, together with their infinitesimal gauge transformations, is sufficient for many physical models and their quantization.
    
    We follow mostly the ideas of~\cite{Cartan:1949aaa,Cartan:1949aab,Kotov:2007nr} and in particular~\cite{Sati:2008eg} and~\cite{Saemann:2019dsl}, and extend them in the evident way to the case of local connections on Lie groupoid bundles. As a new result, we show that adjustments amount to Cartan connections.
    
    \subsection{Flat and higher flat connections}
    
    Flat Lie algebroid-valued connections are readily defined as follows.
    \begin{definition}[Flat local $\frg$-valued connection]
        Consider a Lie algebroid $\frg=(\frg\rightarrow M)$. A flat local connection with values in $\frg$ on a patch $U_X\subset X$ of a manifold $X$ is given by an algebroid morphism 
        \begin{equation}
            \caA:\rmT[1]U_X\rightarrow \frg[1]~.
        \end{equation}
    \end{definition}
    \noindent Unpacking this definition, we obtain a smooth map $\phi$ together with a $\phi^*\frg$-valued $1$-form,
    \begin{subequations}\label{eq:flat_g-valued_connection}
        \begin{equation}
            \phi\in C^\infty(U_X,M)\eand A\in \Omega^1(U_X,\phi^*\frg)~.
        \end{equation}
        Compatibility with the differentials on $\rmT[1]U_X$ and $\frg[1]$ implies\footnote{Here and in the following, $\rho$, $[-,-]$, $\nabla$, and $t_{\nabla_\rho}$ (defined below) may denote, in fact, the maps $\phi^*\rho$, $\phi^*[-,-]$, $\Phi^*\nabla$, and $\Phi^*t_{\nabla_\rho}$ respectively. Since this will always be clear from the context, we allow ourselves to use this simpler but slightly imprecise notation.}
        \begin{equation}\label{eq:flat_connections_curvatures}
            \rmd \phi-\rho(A)=0\eand
            \mathrm{d}^{\nabla} A
            - \frac{1}{2} t_{\nabla_\rho}( A , A )
            =0~,
        \end{equation}
    \end{subequations}    
    where $\nabla$ is any vector bundle connection on $\frg$, $\rmd^{\nabla}$ denotes the exterior covariant derivative of $\nabla$, and $t_{\nabla_\rho}$ is the torsion of $\nabla_\rho$. That is,
    \begin{align}
        \left(\rmd^\nabla A\right)(V, W)
        &=
        \nabla_V \bigl( A(W) \bigr)
        - \nabla_W \bigl( A(V) \bigr)
        - A \bigl( [V, W] \bigr)~,
        \\
        t_{\nabla_\rho}(\mu, \nu)
        &=
        \nabla_{\rho(\mu)} \nu
        - \nabla_{\rho(\nu)} \mu
        - [\mu, \nu]
    \end{align}
    for all $V, W \in \frX(U_X)$ and $\mu, \nu \in \Gamma(\frg)$, where $\rho$ and $[-,-]$ are the anchor and the Lie bracket in $\frg$, respectively; observe that the second equation in~\eqref{eq:flat_connections_curvatures} is independent of the choice of $\nabla$, in particular, with respect to a frame given by a basis $(e_\alpha)$ in $\frg$ and making use of the first equation in~\eqref{eq:flat_connections_curvatures}, the second equation in~\eqref{eq:flat_connections_curvatures} reads
    \begin{equation}
        \rmd A^\alpha +\frac12 \ttf^\alpha_{\beta\gamma} A^\beta A^\gamma =0~.
    \end{equation}
    There is now an action of infinitesimal bundle isomorphisms on the flat $\frg$-valued connections, which are called infinitesimal gauge transformations in the physics literature. To expose also these, we construct the BRST complex following a construction familiar from the AKSZ-model~\cite{Alexandrov:1995kv}, see also~\cite[Section 4.2]{Saemann:2019dsl}.
    \begin{definition}[BRST complex for flat connections]\label{def:BRST-complex}
        Given a patch $U_X$ of a manifold $X$ and an N$Q$-manifold\footnote{i.e.~a dg-manifold concentrated in non-positive degrees, cf.~\ref{app:graded_manifolds}} $\caG$, the BRST complex $\frbrst(U_X,\caG)$ is given by the internal hom $\ihom(\rmT[1]U_X,\caG)$ in the category $\CatNQMfd$ of N$Q$-manifolds, where the degree of the contained morphisms is called the ghost degree.
    \end{definition}
    \noindent This definition assumes that $\ihom(\rmT[1]U_X,\caG)$ is representable, which is a consequence of $\caG$ being concentrated in non-positive degrees. In particular, we will calculate the corresponding dg-manifold below.
    
    In the case of flat local $\frg$-valued connections, we can now define infinitesimal gauge transformations as follows.
    \begin{definition}[Infinitesimal gauge transformations for flat local $\frg$-valued connections]
        Consider the BRST complex $\frbrst(U_X,\frg[1])$ for $\frg[1]$ a Lie algebroid. The generators of this complex are the fields (ghost degree~0) and the gauge parameters or ghosts (ghost degree~1). The infinitesimal gauge transformations are given by the action of the differential on the fields.
    \end{definition}
    
    Let us compute $\frbrst(U_X,\frg[1])$ explicitly. The dg-commutative algebra is generated by the duals of 
    \begin{equation}
        \phi\in C^\infty(U_X,M)
        ~,~~~
        A\in \Omega^1(U_X,\phi^*\frg)
        ~,\eand
        c\in \Omega^0(U_X,\phi^*\frg)~,
    \end{equation}
    which encode maps of degree~$0$, $0$, and~$-1$, respectively. With respect to a frame $(e_\alpha)$ in $\frg$, the differential $\sfd$ in $\frbrst(U_X,\frg[1])$ acts on these according to
    \begin{equation}
        \sfd \uline\phi=\rho(\uline c)~,~~~\sfd \uline A^\alpha=\rmd \uline c^\alpha+ \ttf^\alpha_{\beta \gamma} \uline A^\beta \uline c^\gamma~,\eand \sfd \uline c^\alpha=-\tfrac12 \ttf^\alpha_{\beta \gamma}\uline c^\beta\uline c^\gamma~,
    \end{equation}
    where $\uline \phi$, $\uline A$, and $\uline c$ are ``coordinate functions'' on the spaces that $\phi$, $A$, and $c$ take values in, respectively.\footnote{Note that the physics literature usually does not make this distinction and simply identifies $\uline c$ with $c$, etc.}
    
    We conclude that the infinitesimal gauge transformations are parameterized by $c\in \Omega^0(U_x,\phi^*\frg)$ and act on the fields $\phi$ and $A$ according to
    \begin{equation}
        \delta \phi=\rho(c)\eand \delta A^\alpha=\rmd c^\alpha+\ttf^\alpha_{\beta \gamma} A^\beta c^\gamma~.
    \end{equation}
    The differential on $c$ describes the Chevalley--Eilenberg algebra of the Lie algebra of infinitesimal gauge transformations.
    \begin{remark}[Coordinate-free expressions]
        Here and in the following, a coordinate-free version is not as straightforward as before: many terms have values in $\phi^*\frg$ so that a gauge transformation of $\phi$ has a much wider effect than in ordinary gauge theory. For example, the one-form potential $A$ projects to $\phi$ (via $\frg \to M$) so that $\delta A$ projects to $\delta \phi$ via $\rmT \frg \to \rmT M$; in particular, $\delta A$ is not vertical in $\rmT\frg$ anymore such that it cannot have values in $\frg$ in general, it has instead values in the tangent prolongation of $\frg$. For technical explanations of this and related terms like the Lie algebra expression of $c$ see~\cite{Fischer:2021glc, Fischer:2021yoy}.
    \end{remark}
    
    Two remarks are in order.
    \begin{remark}
        Equivalently, one can also regard these infinitesimal gauge symmetries as infinitesimal homotopies between flat $\frg$-valued connections, cf.~e.g.~\cite{Sati:2008eg}.
    \end{remark}
    
    \begin{remark}
        For $M=*$, the Lie algebroid $\frg$ becomes a Lie algebra, and the above reproduces the usual description of an ordinary local flat Lie-algebra-valued connection 1-form.
    \end{remark}

    As mentioned before, an observation that proved useful in higher gauge theory is that non-flat connections can be regarded as special flat higher connections. Here, we are interested in the dg-manifold $\rmT[1]\frg[1]$. 
    \begin{definition}[{Flat local $\rmT[1]\frg$-valued connection}]\label{eq:flat_local_Tg_connections}
        Given a Lie algebroid $\frg=(\frg\rightarrow M)$, a flat local connection with values in $\rmT[1]\frg$ on a patch $U_X\subset X$ of a manifold $X$ is given by an algebroid morphism 
        \begin{equation}\label{eq:t1g1-connection}
            \caA:\rmT[1]U_X\rightarrow \rmT[1]\frg[1]~.
        \end{equation}
    \end{definition}
    \noindent To unpack this definition, we assume for convenience that we have a connection $\nabla$ on $\frg$ together with a $\frg$-valued 2-form $\zeta\in \Omega^2(M,\frg)$, which allows us to work with the global Weil differential $\sfW_{(\nabla,\zeta)}(\frg)$ as discussed in \ref{ssec:Weil_algebra}. Note that we can always choose a connection on $\frg$ and put $\zeta=0$. 
    
    The map $\caA$ of graded manifolds is defined by a map $\phi\in C^\infty(U_X,M)$ between the bodies of both graded manifolds together with maps $C^\infty(\rmT[1]\frg[1])\rightarrow \phi_*\Omega^1(U_X)$, where $\phi_*$ is the direct image functor on sheaves. A $\rmT[1]\frg$-valued connection then amounts to the data
    \begin{subequations}
        \begin{equation}\label{eq:kinematical_data}
            \phi\in C^\infty(U_X,M)~,~~~A_M\in\Omega^1(U_X,\phi^*\rmT M)~,~~~A_\frg\in \Omega^1(U_X,\phi^*\frg)~,~~~B\in \Omega^2(U_X,\phi^*\frg)~.
        \end{equation}
        In terms of components with respect to the basis used in \ref{ssec:Weil_algebra}, these satisfy
        \begin{equation}\label{eq:flat_T1g1-connection}
            \begin{aligned}
                0&=(\rmd \phi)^a-\ttr^a_\alpha A_{\frg}^\alpha-A_{M}^a~,
                \\
                0&=\rmd A^a_M+\ttr^a_\alpha \left(B^\alpha-
                \omega_{a\beta}^\alpha  A_M^a  A_\frg^\beta-\tfrac12 \zeta^\alpha_{ab}  A_M^a  A_M^b\right)-\frac{\dpar\ttr^a_\alpha}{\dpar m^b}  A_\frg^\alpha  A_M^b~,
                \\
                0&=\rmd A^\alpha_\frg+\tfrac12 \ttf^\alpha_{\beta\gamma}A_\frg^\beta A_\frg^\gamma-B^\alpha+
                \omega_{a\beta}^\alpha A_M^a A_\frg^\beta+\tfrac12 \zeta^\alpha_{ab} A_M^a A_M^b~,
                \\
                0&=\rmd B^\alpha +\left(\ttf^\alpha_{\beta\gamma}+\ttr^a_\gamma \omega^\alpha_{a\beta}\right) A_\frg^\beta B^\gamma+\left(\omega^\alpha_{a\beta}-\zeta^\alpha_{ab}\ttr^b_\beta\right) A_M^a B^\beta
                -\tfrac12 \left(R_\nabla^{\rm bas}\right)^\alpha_{\beta\gamma a} A_\frg^\beta A_\frg^\gamma A_M^a
                \\
                &~~~~ -\left(\tfrac16\left(\rmd^\nabla\zeta\right)^\alpha_{abc}-\tfrac12 \zeta^\alpha_{ad}\ttr^d_\beta \zeta^\beta_{bc}\right)A_M^a A_M^b A_M^c-\tfrac12 \left(R_\nabla+\nabla^{\rm bas}\zeta\right)^\alpha_{ab\beta} A_M^a A_M^b A_\frg^\beta~,
            \end{aligned}
        \end{equation}
    \end{subequations}    
    which follows from $\caA$ respecting the differential.
    
    Infinitesimal gauge transformations are again captured by the BRST complex. 
    \begin{definition}[{Infinitesimal gauge transformations for local flat $\rmT[1]\frg$-valued connections}]\label{def:BRST_Tg}
        Consider the BRST complex $\frbrst(U_X,\rmT[1]\frg[1])$. The generators of this complex are the fields (ghost degree~0), the gauge parameters or ghosts (ghost degree~1), and the higher gauge parameters or ghosts-for-ghosts (ghost degree~2). The infinitesimal gauge transformations are given by the action of the differential on the fields. The infinitesimal higher gauge transformations are given by the action of the differential on the ghosts, truncated to terms containing ghosts-for-ghosts.
    \end{definition}
    \noindent In the dg-commutative algebra $\frbrst(U_X,\rmT[1]\frg[1])$, the fields are given by 
    \begin{equation}
        \phi\in C^\infty(U_X,M)~,~~~A_M\in\Omega^1(U_X,\phi^*\rmT M)~,~~~A_\frg\in \Omega^1(U_X,\phi^*\frg)~,~~~B\in \Omega^2(U_X,\phi^*\frg)~,
    \end{equation}
    the gauge parameters by 
    \begin{equation}
        c_M\in \Omega^0(U_X,\phi^*\rmT M)~,~~~c_\frg\in\Omega^0(U_X,\phi^*\frg)~,~~~\lambda\in\Omega^1(U_X,\phi^*\frg)~,
    \end{equation}
    and the higher gauge parameters by 
    \begin{equation}
        \chi\in \Omega^0(U_X,\phi^*\frg)~.
    \end{equation}
    
    Let us again work with the Weil algebra $\sfW_{(\nabla,\zeta)}$ and the coordinates introduced in \ref{ssec:Weil_algebra}. Then the gauge transformations, induced by the action of the differential $\sfd$ in $\frbrst(U_X,\rmT[1]\frg[1])$ on the fields, read as
    \begin{equation}\label{eq:flat_T1g1-connection_gauge_transformations}
        \begin{aligned}
            \delta \phi^a &=\ttr^a_\alpha c_{\frg}^\alpha+c_{M}^a~,
            \\
            \delta A^a_M&=\rmd c_M^a-\ttr^a_\alpha \left(\lambda^\alpha-
            \omega_{a\beta}^\alpha  \left(c_M^a  A_\frg^\beta-
            A_M^a  c_\frg^\beta\right)-\zeta^\alpha_{ab}  c_M^a  A_M^b\right)+\frac{\dpar\ttr^a_\alpha}{\dpar m^b}  \left(c_\frg^\alpha  A_M^b-A_\frg^\alpha  c_M^b\right)~,
            \\
            \delta A^\alpha_\frg &=\rmd c_\frg^\alpha+\ttf^\alpha_{\beta\gamma}A_\frg^\beta c_\frg^\gamma+\lambda^\alpha-
            \omega_{a\beta}^\alpha \left(c_M^a A_\frg^\beta-A_M^a c_\frg^\beta\right)-\zeta^\alpha_{ab} c_M^a A_M^b~,
            \\
            \delta B^\alpha &=\rmd \lambda^\alpha-\left(\ttf^\alpha_{\beta\gamma}+\ttr^a_\gamma \omega^\alpha_{a\beta}\right) \left(c_\frg^\beta B^\gamma-A_\frg^\beta \lambda^\gamma\right)-\left(\omega^\alpha_{a\beta}-\zeta^\alpha_{ab}\ttr^b_\beta\right) \left(c_M^a B^\beta-A_M^a \lambda^\beta\right)
            \\
            &~~~~+\left(\tfrac12\left(\rmd^\nabla\zeta\right)^\alpha_{abc}-\tfrac32 \zeta^\alpha_{ad}\ttr^d_\beta \zeta^\beta_{bc}\right)c_M^a A_M^b A_M^c+\left(R_\nabla^{\rm bas}\right)^\alpha_{\beta\gamma a} \left(c_\frg^\beta A_\frg^\gamma A_M^a+\tfrac12 A_\frg^\beta A_\frg^\gamma c_M^a\right)
            \\
            &~~~~+\left(R_\nabla+\nabla^{\rm bas}\zeta\right)^\alpha_{ab\beta} \left(c_M^a A_M^b A_\frg^\beta+\tfrac12 A_M^a A_M^b c_\frg^\beta\right)~.
        \end{aligned}
    \end{equation}
    
    The higher gauge transformations, induced by the action of the differential $\sfd$ on the ghosts, read as 
    \begin{equation}\label{eq:flat_higher_gauge_trafos}
        \begin{gathered}
            \delta c^a_M=-\ttr^a_\alpha  \chi^\alpha~,~~~       \delta c^\alpha_\frg=\chi^\alpha~,\\ 
            \delta \lambda^\alpha =-\rmd \chi^\alpha-
            \left(\ttf^\alpha_{\beta\gamma}+\ttr^a_\gamma \omega^\alpha_{a\beta}\right) A_\frg^\beta \chi^\gamma-\left(\omega^\alpha_{a\beta}-\zeta^\alpha_{ab}\ttr^b_\beta\right) A_M^a \chi^\beta~.
        \end{gathered}
    \end{equation}
    
    Again, these (higher) gauge transformations can be regarded as (higher) homotopies, and details are found, e.g., in~\cite[Sec.~4.1]{Jurco:2018sby}.
    
    \subsection{Non-flat connections as higher flat connections}\label{sec:Non-flat connections as higher flat connections}
    
    From~\eqref{eq:flat_T1g1-connection} and~\eqref{eq:flat_connections_curvatures}, it becomes clear that we can consider a non-flat $\frg$-valued connection as a flat $\rmT[1]\frg$-valued connection, where the additional potential components $A_M$ and $B$ adopt the role of the curvatures\footnote{We call the covariant derivative $\nabla \phi=\rmd \phi-\rho(A)\phi$ the curvature of $\phi$.} for $\phi$ and $A$. This interpretation, however, requires us to reduce the gauge symmetries to those of a $\frg$-valued connection, because as it stands, gauge transformations relate flat and non-flat connections.\footnote{More generally, one expects a generalized Poincar\'e lemma (cf.~e.g.~\cite{Demessie:2014ewa} for the higher gauge theoretic version), which allows us to trivialize all of these flat $\rmT[1]\frg$-valued connections.}
    
    For higher gauge theory, the appropriate reduction was identified in~\cite[Def.~34]{Sati:2008eg} as partially flat homotopies, i.e.~homotopies for which the auxiliary potentials taking the role of the curvatures vanish in the homotopy directions. This ensures that gauge transformations map flat potentials to flat potentials. In other words, the curvature analogue of gauge transformations is zero or gauge transformations are flat.
    
    In the case of flat $\rmT[1]\frg$-valued connections, the partial flatness condition amounts to demanding
    \begin{equation}\label{eq:truncation_of_gauge_symmetry}
        c_M=0~,~~~\lambda=0~,~~~\chi=0
    \end{equation}
    of the gauge and higher gauge parameters, cf.~also the construction in~\cite[Section 4.2]{Saemann:2019dsl}. For ordinary gauge theories, i.e.~for gauge Lie algebroids $\frg=(\frg\rightarrow M)$ with $M=*$, this truncation can be implemented directly. For general Lie algebroids $\frg$, however, we encounter a consistency condition. In particular, we need to demand that the commutator of infinitesimal gauge transformations yields again an infinitesimal gauge transformation, so that these infinitesimal gauge transformations indeed generate finite gauge transformations that combine together appropriately. This is best studied, again, in the BRST-complex $\frbrst(U_X,\rmT[1]\frg[1])$ introduced above. In particular, we have the differential
    \begin{equation}\label{eq:BRST_on_lambda}
        \begin{aligned}
            \sfd \underline{\lambda}^\alpha &=-\rmd \underline{\chi}^\alpha-
            \left(\ttf^\alpha_{\beta\gamma}+\ttr^a_\gamma \underline{\omega}^\alpha_{a\beta}\right) \left(\underline{A}_\frg^\beta \underline{\chi}^\gamma+\underline{c}^\beta_\frg \underline{\lambda}^\gamma\right)-\left(\underline{\omega}^\alpha_{a\beta}-\zeta^\alpha_{ab}\ttr^b_\beta\right) \left(\underline{A}_M^a \underline{\chi}^\beta+\underline{c}^a_M\underline{\lambda}^\beta\right)\\
            &~~~~-\left(R_{\nabla}^{\rm bas}\right)^\alpha_{\beta\gamma a}\left(\underline{c}^\beta_\frg \underline{A}^\gamma_\frg \underline{c}^a_M -\frac{1}{2}\underline{c}^\beta_\frg \underline{c}^\gamma_\frg \underline{A}^a_M\right)+3\left(\frac{1}{6}\left(\rmd^\nabla \zeta\right)^\alpha_{abc}-\frac{1}{2}\zeta^\alpha_{ad}\ttr^d_\beta\zeta^\beta_{bc}\right)\underline{c}^a_M \underline{c}^b_M \underline{A}_M^c\\
            &~~~~-\left(R_\nabla+\nabla^{\rm bas}\zeta\right)^\alpha_{ab\beta}\left(\underline{c}_M^a \underline{A}_M^b c_\frg^\beta-\frac{1}{2}\underline{c}^a_M \underline{c}^b_M \underline{A}^\beta_\frg\right)~,
        \end{aligned}
    \end{equation}
    where underlined letters denote again coordinate functions. The terms linear in $\uline\chi$ describe the action of higher gauge transformations parameterized by $\chi$ on the gauge parameter $\lambda$. The remaining terms describe the higher algebra of the gauge parameters themselves. 
    
    We now note that consistency of the truncation~\eqref{eq:truncation_of_gauge_symmetry} implies $\uline\lambda^\alpha=0$ and hence $\delta\uline \lambda^\alpha=0$. We see that most of the terms on the right-hand side of~\eqref{eq:BRST_on_lambda} indeed vanish after imposing the truncation $\uline c^a_M=0$, $\uline \lambda^\alpha=0$ and $\uline \chi^\alpha=0$. The only remaining term is
    \begin{equation}
        \tfrac12 \left(R_\nabla^{\rm bas}\right)^\alpha_{\beta\gamma a} \uline c_\frg^\beta \uline c_\frg^\gamma \uline A_M^a~,
    \end{equation}
    which needs to vanish.

    There are now two solutions to this. Either, we impose the condition that $A_M^a=0$ for all the connections we consider. This is the Lie algebroid analogue of the fake curvature condition of higher gauge theory, cf.~\cite{Saemann:2019dsl} and~\cite{Borsten:2024gox} for a discussion. In the case of higher gauge theory, this has already been shown to be too restrictive for many physical applications. Alternatively, we can demand that $\nabla$ is chosen such that $R_\nabla^{\rm bas}=0$. This is the Lie algebroid analogue of the notion of adjustment as defined in~\cite{Saemann:2019dsl}. This perspective proved to be the right one for important physical applications~\cite{Borsten:2024gox}, and it also leads to the correct differential refinement of string structures~\cite{Saemann:2019dsl,Kim:2019owc,Rist:2022hci} and~\cite{Tellez-Dominguez:2023wwr}. In the context of curved gauge theories it was also discovered that the closure of the commutator of infinitesimal gauge transformations requires such a condition, see~\cite{Bojowald:0406445, Mayer:2009wf, Fischer:2021glc, Fischer:2021yoy}. We therefore follow this approach.
    
    While the condition $\tfrac12 (R_\nabla^{\rm bas})^\alpha_{\beta\gamma a} \uline c_\frg^\beta \uline c_\frg^\gamma \uline A_M^a=0$ is required for consistency of a connection, we note that our description becomes more elegant under two further assumptions. First of all, we note in~\eqref{eq:flat_T1g1-connection_gauge_transformations} that a generic curvature $B$ transforms ``non-covariantly'', i.e.~non-linearly under generic truncated gauge transformations $(c^a_M,c^\alpha_\frg,\lambda)=(0,c^\alpha_\frg,0)$, unless
    \begin{equation}\label{eq:condition_1}
        R_\nabla^{\rm bas}=0\eand R_\nabla+\nabla^{\rm bas}\zeta=0~.
    \end{equation}
    Such a covariance condition is usually the key condition imposed in physics discussions, e.g.~in the tensor hierarchy of gauged supergravity, see~\cite{Samtleben:2008pe} and references therein; see the introduction, or \Cref{rem:ReferencesForAdjustment} later, for a history of these conditions in the context of curved gauge theory. This condition is also important in the formulation of action principles.
    
    Moreover, one may want $\rmT[1]\frg[1]$ to possess the structure of a Lie algebroid over $M$ (or, rather, a strict Lie 2-algebroid), for which all cubic terms in $\sfd_\sfW $ have to vanish. This leads to the conditions\footnote{The last relation may be required for writing down Lagrangian field theories, cf.~\ref{rem:strictness}.}
    \begin{equation}\label{eq:condition2}
        R_\nabla^{\rm bas}=0~,~~~R_\nabla+\nabla^{\rm bas}\zeta=0~,\eand \tfrac16\left(\rmd^\nabla\zeta\right)^\alpha_{abc}-\tfrac16 \zeta^\alpha_{[a|d}\ttr^d_\beta \zeta^\beta_{|bc]}=0~,
    \end{equation}
    where we anti-symmetrize over $a$, $b$, and $c$ in the second summand of the third equation, and $\rmd^\nabla$ is the common notation for the exterior covariant derivative w.r.t.~$\nabla$; as a reminder, in this case we have
    \begin{align*}
        \left(\rmd^\nabla\zeta\right)(V_1, V_2, V_3)
        &=
        \nabla_{V_1}\bigl( \zeta(V_2, V_3) \bigr)
        - \nabla_{V_2}\bigl( \zeta(V_1, V_3) \bigr)
        + \nabla_{V_3}\bigl( \zeta(V_1, V_2) \bigr)
        \\&\hspace{1cm}
        - \zeta\bigl([V_1, V_2], V_3\bigr)
        + \zeta\bigl([V_1, V_3], V_2\bigr)
        - \zeta\bigl([V_2, V_3], V_1\bigr)
    \end{align*}
    for all $V_1, V_2, V_3 \in \frX(M)$. A special case of such conditions appeared in Mackenzie's book~\cite[Sec.~7.2]{0521499283}. He also imposed the condition that there is a Lie algebroid structure on the restriction 
    \begin{equation}
        \left.\rmT[1]\frg[1]\right|_{M} \cong \rmT[1]M\oplus \frg[2]~,
    \end{equation}
    where $\rmT[1]\frg[1]$ is now viewed as a bundle over $\frg[1]$ and the restriction to $M$ is w.r.t.~the canonical embedding into $\frg[1]$ via the 0-section.
    This requires that~\eqref{eq:condition2} is satisfied. Moreover, Mackenzie considers Lie algebra bundles in place of $\frg[2]$ and general Lie algebroids in place of $\rmT[1]M$, and Mackenzie shows that the third equation in~\eqref{eq:condition2} can be studied with a cohomology leading to a cohomology class encoding the existence of such $\zeta$, called the Mackenzie obstruction class; this is a natural invariant of the field redefinitions, see~\cite{Fischer:2020lri, Fischer:2021glc} for further details on Mackenzie's constructions and their relations to adjusted connections.
    
    We note that in many cases, there will be coordinate changes on the Weil algebra preserving the condition~\eqref{eq:condition_1} or even the condition~\eqref{eq:condition2}, but simplifying the Weil algebra further, e.g.~removing terms in~\eqref{eq:coordinate_changed_Weil} proportional to $\zeta$ or $\omega$. Such coordinate changes are general forms of the field redefinitions discussed in~\cite{Fischer:2021glc}. 
    
    \begin{remark}[Alternative versions of equations]\label{rem:AlternCompEq}
        Making use of the notion of exterior covariant derivatives of Lie algebroid connections $\nabla^{\rm bas}$, denoted by $\rmd^{\nabla^{\rm bas}}$, one can rewrite the second equation also as 
        \begin{equation}
            R_\nabla+\rmd^{\nabla^{\rm bas}}\zeta=0~.
        \end{equation}
        As pointed out in~\cite{Fischer:2020lri, Fischer:2021glc}, this means that the first two equations are of a cohomological type; integrated in~\cite{Fischer:2022sus} (in the case of Lie group bundles as structure). However, this requires introducing exterior covariant derivatives of Lie algebroid connections and a double-degree version of forms.
        
        The third equation can also be written without coordinates as
        \begin{equation}
            \rmd^{\nabla^\zeta} \zeta = 0~,
        \end{equation}
        where $\nabla^\zeta$ is a vector bundle connection on $\frg$ defined by
        \begin{align*}
            \nabla^\zeta_V \nu
            &\coloneqq
            \nabla_V \nu
            - \zeta \bigl( V, \rho(\nu) \bigr)
        \end{align*}
        for all $V \in \frX(M)$ and $\nu \in \Gamma(\frg)$. A geometric meaning of this may appear in~\cite{future:2024ac}.
    \end{remark}

    \subsection{Adjusted local Lie-algebroid-valued connections}\label{ssec:local_adjusted_connections}
    
    Let us now formalize the results of the above perspective, requiring condition~\eqref{eq:condition_1}. The extension of the notion of adjustment~\cite[Def.~4.2]{Saemann:2019dsl} to Lie algebroids reads as follows.
    \begin{definition}[Adjustment]\label{def:local_adjustment}
        Given a Lie algebroid $\frg=(\frg\rightarrow M)$, we call a connection $\nabla$ on $\frg$ an adjustment (or plain adjustment) if its basic curvature vanishes, 
        \begin{subequations}\label{eq:local_adjustment_conditions}
            \begin{equation}
                R^{\rm bas}_\nabla=0
            \end{equation}
            A covariant adjustment is an adjustment together with a primitive $\zeta$ of $\nabla$, so that 
            \begin{equation}
                R_\nabla+\nabla^{\rm bas}\zeta=0~.
            \end{equation}
            A strict covariant adjustment is a covariant adjustment for which we additionally have
            \begin{equation}
                \left(\rmd^\nabla\zeta\right)^\alpha_{abc}- \zeta^\alpha_{[a|d}\ttr^d_\beta \zeta^\beta_{|bc]}=0~.
            \end{equation}
        \end{subequations}
    \end{definition}
    
    \begin{remark}
        The notions of covariant and strict covariant adjustment are motivated by physical considerations and, in particular, the desire to write down an action principle. Mathematically, they are slightly less natural, as they depend on a coordinate choice in the Weil algebra, while a plain adjustment simply requires the choice of a connection $\nabla$ on $\frg$.
    \end{remark}

    \begin{remark}\label{rem:ReferencesForAdjustment}
        See the introduction for an extended history of these conditions; as major references there are~\cite[general summary of adjusted algebroid theories; but curvature condition phrased differently, not as exactness condition]{Kotov:2015nuz} and~\cite[the exactness condition of the curvature first appeared here]{Fischer:2021glc}.
        
        In the context of these works, gauge invariance of the Lagrangian also requires metric compatibilities regarding the Cartan connection. That is, an adjustment in the context of curved gauge theory has two metric compatibilities of $\nabla^{\mathrm{bas}}$ with metrics on $\frg$ and $\rmT M$. We will gloss over these in this work, until we discuss this in~\cite{future:2024ac}.
    \end{remark}
    
    \begin{remark}[Integrated adjustment]
        From an integrated point of view and as we discussed earlier, the vanishing of the basic curvature gives rise to the notion of Cartan connections $\caH$ on the Lie groupoid $\scG$ whose connection 1-form we denote by $\vartheta^{\mathrm{tot}}$; as pointed out in \cite{Fischer:2022sus} the curvature condition integrates to\footnote{For the sake of readability we simplified the notation regarding the base points.}
        \begin{equation}
            \left.\left(\mathrm{d}^{\sfs^*\nabla} \vartheta^{\mathrm{tot}}
            - \frac{1}{2} \left(\sfs^*t_{\nabla_\rho}\right)\left( \vartheta^{\mathrm{tot}} \stackrel{\wedge}{,} \vartheta^{\mathrm{tot}} \right) \right) \right|_{g}
            =
            %\sAd_{g^{-1}} \circ \mleft.\mleft(\pi_{\mathcal{G}}^! \zeta\mright)\mright|_g
            \mathrm{Ad}_{\caH}\left( g^{-1}, \left.\sft^! \zeta\right|_g\right)
            - \left.\sfs^! \zeta\right|_g
        \end{equation}
        for all $g \in \scG$, where the exclamation marks denote the pullback of forms, $t_{\nabla_\rho}$ is the torsion of $\nabla_\rho$, and
        \begin{equation}
            \left(\left(\sfs^*t_{\nabla_\rho}\right)\left( \vartheta^{\mathrm{tot}} \stackrel{\wedge}{,} \vartheta^{\mathrm{tot}} \right)\right)(V, W)
            \coloneqq
            2 \left(\sfs^*t_{\nabla_\rho}\right)\left( \vartheta^{\mathrm{tot}}(V), \vartheta^{\mathrm{tot}}(W) \right)
        \end{equation}
        for all $V, W \in \mathfrak{X}(\scG)$. It might be surprising to see the torsion in the curvature; this is due to covariantising the theory, see \cite{Kotov:2015nuz, Fischer:2021glc}. In fact, in Eq.~\eqref{eq:flat_connections_curvatures} we also pointed out that the components of the torsion are appearing.
    \end{remark}
    
    \noindent We note that a Lie algebroid with vanishing basic curvature is called a \emph{Cartan algebroid} in~\cite{Blaom:0404313}, and hence Lie algebroid adjustments are Cartan algebroids.
    
    \noindent Using an adjustment, we can then define adjusted local Lie-algebroid-valued connections. We note that all of the following definitions extend in a straightforward manner to covariant and strict covariant adjustments.
    \begin{definition}[Local adjusted $\frg$-valued connection]\label{def:local_adjusted_connections}
        Given a Lie algebroid $\frg=(\frg\rightarrow M)$ with adjustment $\nabla$, a local adjusted $\frg$-valued connection $\caA$ is a local flat $\rmT[1]\frg$-valued connection, cf.~\eqref{eq:flat_local_Tg_connections}. 
        
        The components $(\phi, A) \coloneqq (\phi,A_\frg)$ and $(E,F)\coloneqq (A_M,B)$ of~\eqref{eq:kinematical_data}, defined with respect to the adjustment $\nabla$, are called the connection components and the curvature components, respectively. 
        
        Infinitesimal gauge transformations of $\caA$ are described by the adjusted BRST complex $\frbrst^{\rm adj}(U_X,\frg[1])$ which is the BRST complex $\frbrst(U_X,\rmT[1]\frg[1])$ of~\ref{def:BRST_Tg}, divided by the differential ideal spanned by $\uline c_M$, $\uline \lambda$, and $\uline \chi$ in the coordinates fixed by the adjustment $\nabla$.
    \end{definition}
    
    \begin{remark}\label{def:adjusted_BRST}
        The adjusted BRST complex $\frbrst^{\rm adj}(U_X,\frg[1])$ can also be defined as
        \begin{equation}
            \frbrst^{\rm adj}(U_X,\frg[1])=\ihom^{\rm red}_{\nabla}(\rmT[1]U_X,\rmT[1]\frg[1])~,
        \end{equation}
        where the reduced inner homomorphisms are given by 
        \begin{equation}\label{eq:ihom_red}
            \ihom^{\rm red}_{\nabla}(\caN,\rmT[1]\frg[1])\coloneqq \{\caF\in\ihom(\caN,\rmT[1]\frg[1])~|~\rmdeg(\sfp_{\rm curv}\circ \caF)=0\}
        \end{equation}
        for $\caN\in \CatNQMfd$. Here, $\sfp_{\rm curv}$ is the projection~\eqref{eq:p_curv}, and $\rmdeg(-)$ denotes the degree of the enclosed morphism.
    \end{remark}
    
    \begin{remark} 
        The space $\frbrst^{\rm adj}(U_X,\frg[1])$ defined above exists as a dg-manifold if the basic curvature $R_\nabla^{\rm bas}$ vanishes. This space is generated by (see \ref{sec:Non-flat connections as higher flat connections}) the coordinate functions on field space $(\underline{\phi},\underline{A}_\frg)$ in degree zero, the coordinate functions on gauge parameter space $(\underline{c}_\frg,\underline{c}_M,\underline{\lambda})$ in degree one and the coordinate functions on higher gauge parameter space $\underline{\chi}$ in degree two, quotient by the ideal generated by the equations
        \begin{equation}
            \uline c^a_M=0~,~~\uline \lambda^\alpha=0~,~~\uline \chi^\alpha=0~,~~    \tfrac12 \left(R_\nabla^{\rm bas}\right)^\alpha_{\beta\gamma a} \uline c_\frg^\beta \uline c_\frg^\gamma \uline A_M^a=0~.
        \end{equation}
        The first three equations are linear in generators, and the quotient is straightforward. The last one, however, is not linear and the quotient by the ideal it generates does not lead to a free algebra. In other words, the adjusted BRST complex does not form a dg-manifold\footnote{Here we ignore subtleties related to working with infinite dimensions.}.
    \end{remark}   
    
    Let us spell out this data in detail. Over a patch $U_X$ of some manifold $X$, a $\frg$-valued connection, covariantly adjusted with respect to $(\nabla,\zeta)$, is given by the connection data
    \begin{subequations}\label{eq:local_adjusted_connection}
        \begin{equation}
            \phi \in C^\infty(U_X,M)~,~~~A\in \Omega^1(U_X,\phi^*\frg)~,
        \end{equation}
        and in the coordinates used in the description of the Weil algebra $\sfW_{(\nabla,\zeta)}(\frg)$ in \ref{ssec:Weil_algebra}, the curvature forms $(E,F)\in \Omega^1(U_X,\phi^*TM)\oplus \Omega^2(U_X,\phi^*\frg)$ are given by 
        \begin{equation}
            \begin{aligned}
                E^a &\coloneqq (\rmd \phi)^a-\ttr^a_\alpha A^\alpha~,
                \\
                F^\alpha&\coloneqq \rmd A^\alpha+\tfrac12 \ttf^\alpha_{\beta\gamma}A^\beta A^\gamma+
                \omega_{a\beta}^\alpha E^a A^\beta+\tfrac12 \zeta^\alpha_{ab} E^a E^b~.
            \end{aligned}
        \end{equation}
    \end{subequations}    
    Gauge transformations are parameterized by an element
    \begin{subequations}
        \begin{equation}
            c\in \Omega^0(U_X,\phi^*\frg)
        \end{equation}
        and act according to
        \begin{equation}
            \begin{aligned}
                (\delta \phi)^a &=\ttr^a_\alpha c^\alpha~,~~~&\delta E^a&=-\ttr^a_\alpha 
                \omega_{a\beta}^\alpha
                E^a  c^\beta+\frac{\dpar\ttr^a_\alpha}{\dpar m^b} c^\alpha  E^b~,
                \\
                \delta A^\alpha &=\rmd c^\alpha+\ttf^\alpha_{\beta\gamma}A^\beta c^\gamma+
                \omega_{a\beta}^\alpha E^a c^\beta~,~~~&
                \delta F^\alpha &=-\left(\ttf^\alpha_{\beta\gamma}+\ttr^a_\gamma \omega^\alpha_{a\beta}\right) c^\beta F^\gamma~.
            \end{aligned}
        \end{equation}
    \end{subequations}    
    The Bianchi identities read as 
    \begin{equation}
        \begin{aligned}
            0&=\rmd E^a+\ttr^a_\alpha \left(F^\alpha-
            \omega_{b\beta}^\alpha  E^b  A^\beta-\tfrac12 \zeta^\alpha_{bc}  E^b  E^c\right)-\frac{\dpar\ttr^a_\alpha}{\dpar m^b}  A^\alpha  E^b~,
            \\
            0&=\rmd F^\alpha +\left(\ttf^\alpha_{\beta\gamma}+\ttr^a_\gamma \omega^\alpha_{a\beta}\right) A^\beta F^\gamma+\left(\omega^\alpha_{a\beta}-\zeta^\alpha_{ab}\ttr^b_\beta\right) E^a F^\beta
            \\
            &\hspace{1cm}-\left(\tfrac16\left(\rmd^\nabla\zeta\right)^\alpha_{abc}-\tfrac12 \zeta^\alpha_{ad}\ttr^d_\beta \zeta^\beta_{bc}\right)E^a E^b E^c~.
        \end{aligned}
    \end{equation}
    The expressions for a plain adjustment are obtained by putting $\zeta=0$ in the above formulas.
    
    Let us close the discussion of the local and infinitesimal description by linking our definition of local adjusted $\frg$-valued connections to the literature.
    
    First of all, we note that this definition of local adjusted $\frg$-valued connection is the evident extension of local adjusted connection with values in higher Lie algebras~\cite[Def.~4.2]{Saemann:2019dsl} to Lie algebroids.
    
    If the Lie algebroid $\frg$ admits a trivial covariant adjustment, i.e.~$\nabla$ flat and $\zeta=0$, then our $\frg$-valued connections restrict to the kinematical data of Lie algebroid Yang--Mills--Higgs theory (of ``type I''), as defined in~\cite{Strobl:2004im}. The more general case with arbitrary adjustment is precisely the kinematical data in the discussion of curved Yang--Mills--Higgs gauge theories of~\cite{Kotov:2015nuz, Fischer:2021glc}. The gauge transformations given there are equivalent to ours.
    
    The case $R_\nabla=0$ and $\zeta=0$ was called ``classical'' and the case $R_\nabla=0$ but $\zeta\neq 0$ was called ``pre-classical'' in~\cite{Fischer:2021glc}.
    
    \section{Adjustment groupoids}
    
    In this section, we extend Ševera's perspective on differentiation to define two types of action groupoids: the inner action groupoid $\scA(\scG)$ as well as a generically non-unique adjustment groupoid $\scA^\sfW$. The former will lead to an immediate description of principal groupoid bundles with flat connections, while the latter will extend this picture to principal groupoid bundles with general adjusted connections.

    \subsection{Inner action groupoids of dg-Lie groupoids}
    
    Because the internal homomorphisms in \ref{def:dg-Lie-algebroid} are functors, these naturally extend to an internal groupoid\footnote{The further extension to an internal category will not be useful to us.}.
    \begin{definition}[Inner action groupoid]
        The inner action groupoid $\scA(\scG)$ of a dg-Lie groupoid $\scG$ is the internal homomorphism groupoid\footnote{Recall $\sigma_\Theta(*)$ from \ref{def:tangent_submersion}.} $  \underline{\catdgGrpd}(\check \scC(\sigma_\Theta(*)),\scG)$ with objects $\ihom(\check \scC(\sigma_\Theta(*)),\scG)$ and morphisms the natural isomorphisms between the functors corresponding to $\ihom(\check \scC(\sigma_\Theta(*)),\scG)$ in $\catdgMfd$. 
    \end{definition}
    
    As a useful example, let us compute $\scA(\scG)$ of some (ordinary, i.e.~all differentials are trivial) Lie groupoid $\scG$, continuing the computation of the Lie algebroid of \ref{ssec:Severa_differentiation}. 
    The objects of $\scA(\scG)$ are by definition $\frg[1]$, and a straightforward generalization of the discussion in \ref{ssec:Severa_differentiation} shows that they amount to functors $\Phi$ given by dg-maps
    \begin{subequations}
        \begin{equation}\label{eq:inner_hom_objects_general_form}
            F_1:\caX\times \Theta\times \Theta \rightarrow \sfG
            \eand 
            F_0:\caX\times \Theta \rightarrow M
        \end{equation}
        such that 
        \begin{equation}
            \begin{gathered}
                F_1(x,\theta_1,\theta_3)=F_1(x,\theta_1,\theta_2)F_1(x,\theta_2,\theta_3)~,
                \\
                \sft(F_1(x,\theta_1,\theta_2))=F_0(x,\theta_1)~,\eand \sfs(F_1(x,\theta_1,\theta_2))=F_0(x,\theta_2)~,
            \end{gathered}
        \end{equation}
        and $F_1$ is hence of the form
        \begin{equation}
            F_1(x,\theta_1,\theta_2)=\sfe_{m(x)}+\theta_2 \nu(x)+\theta_1(\cdots)~,
        \end{equation}
    \end{subequations}        
    where $m(x) \coloneqq F_0(x, 0) \in M$ and $\nu(x)\in \rmT[1]_{\sfe_{m(x)}}\sfG$ with $\sfd\sft(\nu(x))=0$. To compute the inner action groupoid, we now need to consider invertible morphisms between functors $\Phi,\tilde \Phi$ of the type~\eqref{eq:inner_hom_objects_general_form}, which are given by natural isomorphisms $\Phi\Rightarrow \tilde \Phi$. These, in turn, are defined by dg-maps $\alpha\colon \caX\times \Theta\rightarrow \sfG$ satisfying\footnote{We note that a similar computation for deloopings of 1-group and 2-groups was performed in~\cite{Jurco:2014mva} in order to derive the form of gauge transformations.}
    \begin{equation}\label{eq:natural_transformation}
        \tilde F_1(x,\theta_1,\theta_2)=\alpha(x,\theta_1)^{-1}F_1(x,\theta_1,\theta_2)\alpha(x,\theta_2)~.
    \end{equation}
    Such maps $\alpha$ can be expanded in suggestive notation as 
    \begin{equation}
        \alpha(x,\theta)=g(x)+\theta \gamma(x)
    \end{equation}
    with $g(x) \coloneqq \alpha(x, 0)$ and $\gamma(x)\in \rmT[1]_{g(x)}\sfG$, and~\eqref{eq:natural_transformation} requires that
    \begin{equation}
        \rmd\sft(g(x),\gamma(x))=\rmd \sfs\left(\sfe_{m(x)},\nu(x)\right)\in \rmT[1]M~.
    \end{equation}
    The differential reads as $\sfd_\caX g(x)=\gamma(x)$ and induces the de~Rham differential on $\sfG$. 
    
    \begin{definition}[{Right action of $\rmT[1]\sfG$ on $\frg[1]$}]\label{def:black_raction}
        Relation~\eqref{eq:natural_transformation} induces a dg-map
        \begin{equation}
            \ractonB: \frg[1]\fibtimes{\rho}{\rmd \sft}\rmT[1]\sfG\rightarrow \frg[1]
        \end{equation}
        with
        \begin{equation}
            (m,\nu)\times (g,\gamma)\mapsto (\tilde m,\tilde \nu) = (m, v) \ractonB (g, \gamma)
        \end{equation}
        for $\rho(m,\nu)=\rmd \sft(g,\gamma)$ with $\tilde m=\sfs(g)$ and $\tilde \nu$ read off $\tilde F_1(x,\theta_1,\theta_2)$ from~\eqref{eq:natural_transformation} according to 
        \begin{equation}
            \tilde{F_1}(x,\theta_1,\theta_2)=\sfe_{\tilde{m}}+\theta_2 \tilde \nu(x)+\theta_1(\cdots)~.
        \end{equation}
    \end{definition}
    \noindent This dg-map is the obvious linearization of the differential of the adjoint groupoid action, and more explicitly, we may write
    \begin{equation}
        (\sfe_{\tilde m},\tilde \nu)=(g,0)^{-1}\circ (\sfe_{m},\nu)\circ (g,\gamma)~,
    \end{equation}
    where $\circ$ here denotes groupoid composition in $\rmT[1]\sfG$.
    
    In summary, the Lie groupoid $\scA(\scG)$ looks as follows.        
    \begin{proposition}\label{prop:inner_action_groupoid}
        Consider a Lie groupoid $\scG=(\sfG\rightrightarrows M)$ with corresponding Lie algebroid given by the dg-manifold $\frg[1]$.
        The inner action groupoid $\scA(\scG)$ is the action dg-Lie groupoid given by
        \begin{subequations}
            \begin{equation}
                \scA(\scG)=\Big(~\frg[1]\fibtimes{\rho}{\rmd \sft} \rmT[1]\sfG \rightrightarrows \frg[1]~\Big)
            \end{equation}
            with structure maps
            \begin{equation}\label{eq:inner_action_groupoid_structure_maps}
                \begin{gathered}
                    \sft(m,\nu;g,\gamma)=(m,\nu)~,~~~\sfs(m,\nu;g,\gamma)=(m,\nu)\ractonB(g,\gamma)~,
                    \\
                    (m,\nu;g_1,\gamma_1)\circ \bigl((m,\nu)\ractonB(g_1,\gamma_1);g_2,\gamma_2\bigr)=\bigl(m,\nu;(g_1,\gamma_1)\circ (g_2,\gamma_2)\bigr)~,
                    \\
                    \sfe(m,\nu)=(m,\nu;\sfe_m,0)~,~~~(m,\nu;g,\gamma)^{-1}=\left((m,\nu)\ractonB(g,\gamma);(g,\gamma)^{-1}\right)
                \end{gathered}
            \end{equation}
            for all $(m,\nu)\in \frg[1]$ and $(g,\gamma)\in \rmT[1]\sfG$ with $\rmd \sft(g,\gamma)=\rho(m,\nu)$. The differential on objects is given by~\eqref{eq:differential}, and the differential on morphisms is given by the sum of~\eqref{eq:differential} with the de~Rham differential on $\sfG$.
        \end{subequations}
    \end{proposition}
    
    We close this section with a number of remarks.
    \begin{remark}
        The inner action groupoid $\scA(\scG)$ can also be seen as a representation of the 2-functor 
        \begin{equation}
            \catdgGrpd(-\times \check \scC(\sigma_\Theta(*)),\scG):\catdgGrpd^{\rm op}\rightarrow \CatCat~.
        \end{equation}
        Furthermore, $\scA(\scG)$ naturally fibers over the dg-Lie groupoid $\rmT[1]\scG=(\rmT[1]\sfG\rightrightarrows \rmT[1]M)$, and we denote the corresponding functor by $\varrho:\scA(\scG)\rightarrow \rmT[1]\scG$.
    \end{remark}
    
    \begin{remark}\label{eq:explicit_form_action}
        In the special case that $M=*$, the action $\ractonB$ reduces to
        \begin{equation}
            (\sfe, \nu\ractonB(g,\gamma))=(g^{-1},0)\circ(\sfe, \nu)\circ (g,0)+(g^{-1},0)\circ(g,\gamma)~,
        \end{equation}
        a formula that for $\gamma=\rmd g$ reproduces the familiar form of gauge transformations of $\frg$-valued connection 1-forms.
        
        In the general case, one can compute the following formula: 
        \begin{equation}\label{eq:action_before_Weilification}
            (m, \nu)\ractonB(g,\gamma)=\sfp_s\Big(\rmAd_{\caH} \left(g^{-1};m,\nu\right)+\vartheta^{\rm tot}(g,\gamma)\Big)~,
        \end{equation}
        where $\rmAd_{\caH}$ is the adjoint action of \ref{def:adjoint_action}, and $\sfp_s$ is the projection $\sfs^*\frg\rightarrow \frg$. Note that both summands depend on a choice of connection, but this dependence drops out in the sum.
    \end{remark}
    
    \begin{remark}
        The action $\ractonB$ that appeared naturally from our construction is the same as the adjoint action discussed in~\cite[Section 3.2]{Laurent-Gengoux:1507.01051}.
    \end{remark}    
    
    \subsection{Adjustment groupoids of Lie groupoids}
    
    Let us motivate the following definition in a bit more detail. As we shall see later, the inner action groupoid $\scA(\scG)$ can be used to define flat connections on principal $\scG$-bundles, and the action of $\rmT[1]\sfG$ on $\frg[1]$ encodes the gauge transformations. To describe non-flat connections, we want to replace $\frg[1]$ by $\rmT[1]\frg[1]$, as explained in \ref{sec:local_description}, which suggests working with $\scA(\rmT[1]\scG)$. This however, also doubles the space of gauge transformations, which we need to restrict, just as in the local case. Altogether, we are looking for a dg-Lie groupoid $\scA^\sfW$ sitting in the commutative diagram
    \begin{equation}
        \begin{tikzcd}
            \scA^\sfW \arrow[r]& \scA(\rmT[1]\scG)
            \\
            \scA(\scG)\arrow[ur,"\Upsilon"']\arrow[u,"\Psi"]
        \end{tikzcd}
    \end{equation}
    where $\Upsilon$ is the natural embedding of $\scA(\scG)$ into $\scA(\rmT[1]\scG)$ as zero sections, which is an extension of the dual to the projection of the Weil algebra of $\frg$ to its Chevalley--Eilenberg algebra~\eqref{eq:Weil_to_CE_projection}. 
    
    More precisely, we give the following definition.
    \begin{definition}[Adjustment groupoid]\label{def:adjustment_of_groupoid}
        Consider a Lie groupoid $\scG$ with Lie algebroid $\frg[1]=\sfLie(\scG)[1]$. Then an adjustment groupoid of $\scG$ is a dg-Lie groupoid $\scA^\sfW=(\scA^\sfW_1\rightrightarrows \rmT[1]\frg[1])$ 
        together with a functor $\Psi:\scA(\scG)\rightarrow \scA^\sfW$ such that we have the following pullback diagram\footnote{Here and in the following, $(\rmid,0)$ denotes the embedding of a (graded) manifold into its (grade-shifted) tangent bundle as zero section.}:
        \begin{equation}\label{diag:preadjustment_groupoid}
            \begin{tikzcd}[column sep=2cm]
                \frg[1]\fibtimes{\rho}{\rmd\sft}\rmT[1]\sfG \arrow[d,"\sft"]\arrow[r,"\Psi"] & \scA^\sfW_1\arrow[d,"\sft"]
                \\
                \underbrace{~~~~~~\frg[1]~~~~~~}_{\scA(\scG)} \arrow[r,"{(\rmid,0)}"']& \underbrace{\rmT[1]\frg[1]}_{\scA^\sfW}
            \end{tikzcd}
        \end{equation}
    \end{definition}
    \noindent Note that in the above definition, we could have equivalently demanded a pullback diagram involving the source maps. In the description of principal $\scG$-connections, the functor $\Psi$ will ensure that flat connections are correctly included in our description of general connections, and the pullback diagram will ensure that we have the correct amount of gauge freedom.

    \noindent We will construct a concrete example of an adjustment groupoid for a general Lie groupoid $\scG$ shortly.

    \subsection{Adjustment groupoids and Cartan connections}\label{ssec:construction_adjustment_groupoid}
    
    In this section, we show that adjustments of Lie groupoids correspond precisely to Cartan connections.
    \begin{theorem}\label{The:Adjustment_groupoid_classsification}
        Adjustment groupoids of a Lie groupoid $\scG$ (\ref{def:adjustment_of_groupoid}), are classified by Cartan connections on $\scG$ (\ref{def:Cartan_connection}). 
    \end{theorem}
    \begin{proof}
        Note that any weak equivalence of adjustment groupoids (via bibundles) is given by strict isomorphisms, i.e. functors. Because of the choice $\scA^\sfW_{0}=\rmT[1]\frg[1]$, for any weak equivalence of adjustment groupoids $\phi\colon \scA^\sfW\cong \scA^{\sfW\prime}$, there is a weakly commutative diagram 
        
        \begin{equation}
            \begin{tikzcd}
                & \scA^\sfW_{0}=\rmT[1]\frg[1]  \arrow[r] \arrow[d,"="]
                & \scA^\sfW_{0}   \arrow[d,"\phi"]\\
                & \scA^{\sfW\prime}_{0}=\rmT[1]\frg[1]  \arrow[r]
                & \scA^{\sfW\prime}~.
            \end{tikzcd}
        \end{equation}
        By \cite[Lem. 2.29]{Behrend:0605694}, $\phi$ can be represented as a functor $\phi$ which is identity on the space of objects $\rmT[1]\frg[1]$. This allows us to split the proof into two arguments presented further below. 
        
        Given a Cartan connection, there is a right action of $\rmT[1]\sfG$ on $\rmT[1]\frg[1]$ in the sense of \ref{def:right_action}, as we will show in \ref{lem:T[1]G_action}, and the corresponding action groupoid $\scA_\caH^\sfW$ in the sense of \ref{def:action_groupoid}, is an adjustment groupoid, as we will show in \ref{def:adjustemnt_groupoid_for_cartan_connection}. 
        
        Conversely, given an adjustment groupoid, we will construct a Cartan connection in \ref{prop:Cartan_connection_from_adjustment}. From the constructions, it is clear that the adjustment defined by the Cartan connection on $\scG$ derived from an adjustment groupoid $\scA^\sfW$ yields again $\scA_\caH^\sfW=\scA^\sfW$.
    \end{proof}
    
    We start with the construction of the action mentioned in the above proof. In the following, let $\scG=(\sfG\rightrightarrows M)$ be a general Lie groupoid and $\frg[1]$ its Lie algebroid, regarded as a grade-shifted vector bundle. Furthermore, let $\caH$ be a Cartan connection on $\scG$. 
    
    The dg-right action of $\rmT[1]\scG$ on $\rmT[1]\frg[1]$ is defined along a dg-map $\psi\colon\rmT[1]\frg[1]\rightarrow \rmT[1]M$, and the relevant map here is the unique map arising from the lift of\footnote{As before, $\sfp_M$ denotes a projection onto the base $M$ of a fiber bundle $E\rightarrow M$.} $\psi_\flat=\sfp_{M}\circ\sfp_{\frg[1]}$ to a dg-map by \ref{rem:graded-map-to-differential-map}. In addition, we have to construct a dg-map
    \begin{equation}
        -\racton-\colon\rmT[1]\frg[1]\fibtimes{\psi}{\rmd\sft}\rmT[1]\sfG\rightarrow \rmT[1]\frg[1]
    \end{equation}
    satisfying the relations in~\eqref{eq:relations_right_action}. Using again \ref{rem:graded-map-to-differential-map}, it is sufficient to construct a map 
    \begin{equation}
        -\racton_\flat-\colon\rmT[1]\frg[1]\fibtimes{\psi}{\rmd\sft}\rmT[1]\sfG\rightarrow \frg[1]~.
    \end{equation}
    We now note that there is a pullback diagram of graded manifolds
    \begin{equation}\label{eq:pullback_diagram}
        \begin{tikzcd}
            \sft^* \frg[1]\oplus \rmT[1]\sfG \dar[" \sfp_\sft\circ\sfp_1"'] \rar["\sfp_2"] & \rmT[1]\sfG \dar["\sfp_M \circ \rmd\sft"] \\
            \frg[1] \rar["\sfp_M \circ\rho"'] & M,
        \end{tikzcd}
    \end{equation}
    where $\sfp_{1,2}$ are the projections on the first and second summand in the Whitney sum bundle $\sft^* \frg[1]\oplus \rmT[1]\sfG$, and $\sfp_\sft \colon \sft^*\frg[1]\rightarrow \frg[1]$ is the evident projection. We can then define a map\footnote{See \ref{app:graded_manifolds} for further details on the notation.}
    \begin{equation}\label{eq:def_Weilified_action_2}
        - \prec - : \sft^* \frg[1]\oplus \rmT[1]\sfG \rightarrow \frg[1]
    \end{equation}
    as 
    \begin{equation}
        \begin{aligned}
            |\prec|&:&\sfG&\rightarrow M~,
            \\
            &&g &\mapsto \sfs(g)~,
            \\
            \prec^\sharp&:&\Gamma(\frg[1]^*)&\rightarrow C^\infty(\sft^* \frg[1]\oplus \rmT[1]\sfG)~,
            \\
            && \xi&\mapsto \prec^\sharp(\xi)
        \end{aligned}
    \end{equation}
    with
    \begin{equation}\label{eq:def_Weilified_action}
        \left(\prec^\sharp(\xi)\right)(\hat\nu;g,\gamma)\coloneqq (\xi\circ \sfp_\sfs)\left(\rmAd_{\caH}\left(g^{-1},\hat\nu\right)+\vartheta^{\rm tot}(g,\gamma)\right)
    \end{equation}
    for $\hat\nu\in\sft^*\frg[1]$ and $(g,\gamma)\in \rmT[1]\sfG$, where $\sfp_\sfs:\sfs^*\frg[1]\rightarrow \frg[1]$ is the evident projection.
    
    The formula for the action~\eqref{eq:def_Weilified_action} looks identical to the formula~\eqref{eq:action_before_Weilification}. Note, however, that the domain in~\eqref{eq:def_Weilified_action} is larger than in~\eqref{eq:action_before_Weilification}. In the former case, we only have the condition $(\sfp_M\circ \rho \circ \sfp_\sft)(\hat\nu)=(\sfp_M\circ\rmd \sft)(g,\gamma)$, while in the latter case, we have $(\rho \circ \sfp_\sft)(\hat\nu)=\rmd \sft(g,\gamma)$. Thus,~\eqref{eq:def_Weilified_action} is a generalization of~\eqref{eq:action_before_Weilification}, as one would expect.  
    
    The universality of the diagram~\eqref{eq:pullback_diagram} for the evident projections
    \begin{equation}
        \frg[1]~~ \leftarrow~~ \rmT[1]\frg[1]\fibtimes{\psi}{\rmd\sft}\rmT[1]\sfG~~\rightarrow~~ \rmT[1]\sfG
    \end{equation}
    ensures the existence of a unique map $\wp$, which extends to the map $\racton_\flat$ as follows:
    \begin{equation}\label{eq:def_racton_0}
        -\racton_\flat-\colon\rmT[1]\frg[1]\fibtimes{\psi}{\rmd\sft}\rmT[1]\sfG\xrightarrow{~\wp~} \sft^* \frg[1]\oplus \rmT[1]\sfG \xrightarrow{~\prec~} \frg[1]~.
    \end{equation}
    In particular, the map $\wp$ sends $(x, \bar x, \alpha, \bar \alpha;g, \gamma)$ to $(x, \alpha;g, \gamma)$, where $(x,\bar x)$ belongs to $\rmT[1] M$,  $(x, \alpha)$ belongs to $\frg[1]$, and $(x,\bar \alpha)$ belongs to $\frg[2]$. This gives us what we need:
    \begin{lemma}\label{lem:T[1]G_action}
        Consider the map 
        \begin{equation}
            -\racton-:\rmT[1]\frg[1]\fibtimes{\psi}{\rmd\sft}\rmT[1]\sfG\rightarrow \rmT[1]\frg[1]
        \end{equation}
        defined as the unique lift of $\racton_\flat$ from~\eqref{eq:def_racton_0} by \ref{rem:graded-map-to-differential-map}. This map defines a right action of $\rmT[1]\scG$ on $\rmT[1]\frg[1]$.
    \end{lemma}
    \begin{proof}
        The map $\racton$ is a dg-map by definition. Because of~\ref{rem:graded-map-to-differential-map}, the first relation in~\eqref{eq:right_action_explicit_relations} amounts to verifying that the maps 
        \begin{equation}
            (\psi\circ \racton)_\flat: \rmT[1]\frg[1]\fibtimes{\psi}{\rmd\sft}\rmT[1]\sfG\rightarrow M\eand (\rmd \sfs)_\flat:\rmT[1]\sfG\rightarrow M
        \end{equation}
        agree. Note that both $(\psi\circ \racton)_\flat$ and $(\rmd \sfs)_\flat$ are fully defined by $|(\psi\circ \racton)_\flat|=\sfs:\sfG\rightarrow M$ and $|(\rmd \sfs)_\flat|=\sfs:\sfG\rightarrow M$, and hence they agree.
        
        Unitality, i.e.~the second relation in~\eqref{eq:right_action_explicit_relations}, follows immediately from the unitality of the Cartan connection, cf.~\ref{lem:connection_one_form_properties}, which implies that $\prec^\sharp$ is trivial for identities, which, because of universality, then continuous through $\racton_\flat$ to $\racton$.
        
        The third relation in~\eqref{eq:right_action_explicit_relations}, i.e.~consecutive actions amount to groupoid composition, is implied by the fact that this relation holds for the map $\prec^\sharp$. This is seen from
        \begin{equation}
            \begin{aligned}
                &\rmAd_{\caH}\left(g^{-1}_2,\rmAd_{\caH}\left(g^{-1}_1,\hat\nu\right)+\vartheta^{\rm tot}(g_1,\gamma_1)\right)+\vartheta^{\rm tot}(g_2,\gamma_2)
                \\&\hspace{4cm}
                =\rmAd_{\caH}\left((g_1\circ g_2)^{-1},\hat\nu\right)+\vartheta^{\rm tot}\bigl((g_1,\gamma_1)\circ (g_2,\gamma_2)\bigr)
            \end{aligned}
        \end{equation}
        for all $(g_1,\gamma_1)\in \rmT[1]\sfG$ and $(g_2,\gamma_2)\in \rmT[1]\sfG$ with $(g_1,\gamma_1)\circ (g_2,\gamma_2)$ defined and $\hat \nu\in \frg_{\sft(g_1)}$, which is true for Cartan connections $\caH$. That is, $\vartheta^{\rm tot}$ is a multiplicative 1-form with respect to $\rm Ad_\caH$, whose horizontal distribution is complementary to $\sfs^*\frg$, which implies
        \begin{equation}\label{eq:multiplicativity}
            \rmAd_{\caH}\left(g^{-1}_2,\vartheta^{\rm tot}(g_1,\gamma_1)\right)+\vartheta^{\rm tot}(g_2,\gamma_2)
            =
            \vartheta^{\rm tot}\bigl((g_1,\gamma_1)\circ (g_2,\gamma_2)\bigr)~,
        \end{equation}
        see~\cite[discussions around Def.\ 2.1, Lemma 3.6 \& 3.7, and Def.\ 3.9]{Crainic:1210.2277} for specific details and proofs.
    \end{proof}
    
    We can now define an adjustment groupoid as the action groupoid of $\racton$, cf.~\ref{def:action_groupoid}.
    \begin{lemma}\label{def:adjustemnt_groupoid_for_cartan_connection}
        Consider the action dg-groupoid $\scA^\sfW_{\caH}$ given by 
        \begin{equation}
            \scA^\sfW_{\caH}\coloneqq \Big(~\rmT[1]\frg[1]\fibtimes{\psi}{\rmd\sft}\rmT[1]\sfG~\rightrightarrows~\rmT[1]\frg[1]~\Big)
        \end{equation}
        with structure maps
        \begin{equation}
            \begin{gathered}
                \sft(V,\Gamma_1)=V~,~~~\sfs(V,\Gamma_1)=V\racton \Gamma_1~,~~~(V,\Gamma_1)\circ (V\racton \Gamma_1,\Gamma_2)=(V,\Gamma_1\circ \Gamma_2)~,
                \\
                \sfe_V=\left(V,\sfe_{\psi(V)}\right)~,~~~(V,\Gamma_1)^{-1}=\left(V\racton \Gamma_1,\Gamma_1^{-1}\right)
            \end{gathered}
        \end{equation}
        for all $V\in\rmT[1]\frg[1]$ and $\Gamma_{1,2}\in \rmT[1]\sfG$ with $\psi(V)=\rmd \sft(\Gamma_1)$ and such that all compositions and actions are defined. Moreover, the dg-map $\psi\colon\rmT[1]\frg[1]\rightarrow \rmT[1]M$ is the unique map arising from the lift of $\psi_\flat=\sfp_{M}\circ\sfp_{\frg[1]}:\rmT[1]\frg[1]\rightarrow M$ to a dg-map by \ref{rem:graded-map-to-differential-map}
        
        This dg-Lie groupoid is an adjustment groupoid in the sense of~\ref{def:adjustment_of_groupoid}.
    \end{lemma}
    \begin{proof}
        First, we define the functor $\Psi:\scA(\scG)\rightarrow \scA^\sfW_{\caH}$ as 
        \begin{equation}
            \begin{aligned}
                \Psi_1&\colon&\frg[1]\fibtimes{\rho}{\rmd \sft} \rmT[1]\sfG&\rightarrow \rmT[1]\frg[1]\fibtimes{\psi}{\rmd\sft}\rmT[1]\sfG~,~~~&\Psi_1&=(\rmid,0)\times \rmid~,
                \\
                \Psi_0&\colon&\frg[1]&\rightarrow \rmT[1]\frg[1]~,~~~&\Psi_0&=(\rmid,0)~,
            \end{aligned}
        \end{equation}
        where both $(\rmid,0):\frg[1]\rightarrow \rmT[1]\frg[1]$ and $\rmid:\rmT[1]\sfG\rightarrow \rmT[1]\sfG$ are evidently dg-maps. For consistency, we also need 
        \begin{equation}\label{eq:property_rho}
            \rho=\psi\circ(\rmid,0)~,
        \end{equation}
        which is clear from making use of~\ref{rem:graded-map-to-differential-map}: both maps are maps from $\frg[1]$ to $\rmT[1]M$, which are fully defined by dg-maps $\frg[1]\rightarrow M$, which, for degree reasons, are given by the identity map $M\xrightarrow{~\rmid~} M$ on the underlying bodies.
        
        Functoriality of $\Psi$ amounts to showing that $\Psi_0$ is equivariant with respect to the $\rmT[1]\sfG$ actions. Using again \ref{rem:graded-map-to-differential-map}, we have to show that\footnote{Recall $\ractonB$ from \ref{def:black_raction}.}
        \begin{equation}
            \bigl(\racton_\flat\circ (\Psi_0\times \rmid)\bigr)(\nu;g,\gamma)=\ractonB (\nu;g,\gamma)
        \end{equation}
        for all $(\nu; g, \gamma) \in \frg[1]\fibtimes{\rho}{\rmd \sft} \rmT[1]\sfG$,
        and both sides of the equation evaluate to 
        \begin{equation}
            \sfp_\sfs\left(\rmAd_{\caH}\left(g^{-1},\hat \nu\right)+\vartheta^{\rm tot}(g,\gamma)\right)~,
        \end{equation}
        cf.~\eqref{eq:action_before_Weilification}.
        
        Next, we have to show that the morphisms of $\scA(\scG)$ are obtained as a pullback of $\scA^\sfW_{\caH}$ along $\Psi_0=(\rmid,0)$, and we consider the cone
        \begin{equation}
            \begin{tikzcd}[column sep=2cm]
                \caC\arrow[drr,"\beta_1",bend left=20] \arrow[ddr,"\beta_0"',bend right=20] \arrow[dr,dashed,"\beta_2"]
                \\
                & \frg[1]\fibtimes{\rho}{\rmd\sft}\rmT[1]\sfG \arrow[d,"\sfp_1"]\arrow[r,"\Psi_1"] & \rmT[1]\frg[1]\fibtimes{\psi}{\rmd\sft}\rmT[1]\sfG\arrow[d,"\sfp_1"]
                \\
                & \frg[1] \arrow[r,"{(\rmid,0)}"']& \rmT[1]\frg[1]
            \end{tikzcd}
        \end{equation}
        where $\caC$ is a dg-manifold, and $\beta_0$ and $\beta_1$ are dg-maps such that
        \begin{equation}\label{eq:commutativity_1}
            (\rmid,0)\circ \beta_0=\sfp_1\circ \beta_1~.
        \end{equation}
        The map $\beta_2$ is uniquely defined as $\beta_2\coloneqq\beta_0\times (\sfp_2\circ\beta_1)$. Applying $\psi$ to both sides of the above equation, we have
        \begin{equation}
            \rho\circ\beta_0=\rmd\sft\circ\sfp_2\circ\beta_1~,
        \end{equation}
        cf.~also~\eqref{eq:property_rho}.
    \end{proof}

    In the special case of a strict covariant adjustment $(\nabla,\zeta)$ based on the form of the Weil algebra $\sfW_{(\nabla,\zeta)}$, cf.~\eqref{eq:coordinate_changed_Weil}, we can give more explicit formulas that use the isomorphism
    \begin{equation}\label{eq:isomorphismOfTGonTg}
        (\rmT[1]M\oplus \frg[1]\oplus \frg[2])\fibtimes{\psi}{\rmd \sft} \rmT[1]\sfG~\cong~\sft^*\frg[1]\oplus \sft^*\frg[2]\oplus \rmT[1]\sfG~.
    \end{equation}
    
    If $\caH$ integrates such a covariant adjustment $(\nabla, \zeta)$, then we will write $\scA_{(\caH,\zeta)}^\sfW$ for the corresponding adjustment groupoid.
    
    \begin{proposition}
        Consider an adjustment groupoid $\scA_{(\caH,\zeta)}^\sfW$ of a Lie groupoid $\scG=(\sfG\rightrightarrows M)$ with Lie algebroid $\frg=(\frg\rightarrow M)$ and a strict covariant  adjustment $(\nabla,\zeta)$ as in~\ref{def:local_adjustment}. Then the right action $\racton$ is given by 
        \begin{equation}
            \racton: \sft^*\frg[1]\oplus \sft^*\frg[2]\oplus \rmT[1]\sfG\rightarrow \rmT[1]M\oplus \frg[1]\oplus \frg[2]
        \end{equation}
        with
        \begin{equation}
            |\racton|\colon\sfG\rightarrow M~,~~~|\racton|=\sfs
        \end{equation}
        and
        \begin{equation}
            \begin{aligned}
                \racton^\sharp\colon&& \ourodot^\bullet\Gamma\Big(\rmT[1]^*M\oplus\frg[1]^*\oplus\frg[2]^*\Big)&\rightarrow \ourodot^\bullet\Gamma\Big(\sft^*\frg[1]^*\oplus \sft^*\frg[2]^*\oplus \rmT[1]^*\sfG\Big)~,
                \\
                &&\racton^\sharp(\bar m)(\hat{\nu},\bar{\hat{\nu}};g,\gamma)&\coloneqq\sfs^*\bar m(\gamma)-\racton^\sharp(\rho^*\bar m)(\hat{\nu},\bar{\hat{\nu}};g,\gamma)~,
                \\
                &&\racton^\sharp(\xi)(\hat{\nu},\bar{\hat{\nu}};g,\gamma)&\coloneqq\xi\big(\sfp_\sfs(\rmAd_{\caH}(g^{-1},\hat{\nu})+\vartheta^{\rm tot}(g,\gamma)))~,
                \\
                &&\racton^\sharp(\tilde{\bar{\xi}})(\hat{\nu},\bar{\hat{\nu}};g,\gamma)&\coloneqq\Big(\sfd_\scA (\racton^\sharp([1]\tilde{\bar{\xi}}))-\racton^\sharp(\sfd_{\rm c} [1]\tilde{\bar{\xi}})\Big)(\hat{\nu},\bar{\hat{\nu}};g,\gamma)~,
            \end{aligned}
        \end{equation}
        where $\bar m\in \Gamma(\rmT[1]^*M)$, $\xi\in \Gamma(\frg[1]^*)$, and $\tilde{\bar{\xi}}\in \Gamma(\frg[2]^*)$. Moreover, $\sfd_\scA$ is the differential on the morphisms of $\scA^{\sfW}_{(\caH,\zeta)}$, $\sfd_{\rm c}=\sfd_{\rmT [1]M\oplus \frg[1]}$ is the differential on $\rmT [1]M\oplus \frg[1]$ (for whose existence strictness is assumed), and $[1]$ is the shift isomorphism $\Gamma(\frg[2]^*)\rightarrow \Gamma(\frg[1]^*)$.
    \end{proposition}
    \begin{proof}
        By \ref{rem:graded-map-to-differential-map}, the first and third equations are uniquely determined by $\racton$ being a dg-map. Consider now an element $h\in C^\infty(M)$. Then
        \begin{equation}
            \sfd_\scA \racton^\sharp(h)=\sfd_\scA(h\circ \sfs)=\sfs^* \rmd h= \racton^\sharp(\rmd_{\sfW_{(\nabla,\zeta)}}h)= \racton^\sharp(\rmd h+\rmd h\circ \rho)=\racton^\sharp(\rmd h)+\racton^\sharp(\rho^*\rmd h)~,
        \end{equation}
        and as a result, $\racton^\sharp(\rmd h)=\sfs^* \rmd h- \racton^\sharp(\rho^*\rmd h)$. Similarly,
        \begin{equation}
            \sfd_\scA \racton^\sharp(\xi)=\racton^\sharp(\rmd_c\xi+ \tilde{\bar{\xi}}))~,
        \end{equation}
        which leads to the equation $\racton^\sharp(\tilde{\bar{\xi}})=\Big(\sfd_\scA (\racton^\sharp([1]\tilde{\bar{\xi}}))-\racton^\sharp(\sfd_{\rm c} [1]\tilde{\bar{\xi}})\Big)~$.
    \end{proof}

    \begin{remark} 
        Note that Cartan connections on Lie groupoids can be linearized to Cartan connections on Lie algebroids. Also, the primitive $\zeta$ of $\nabla$ for covariant adjustments can be seen as originating from a primitive on the groupoid. This follows from the perspective of connections being closed and their curvatures being exact forms in a particular cohomology (a cohomology as developed in~\cite[Sec.~3.4]{Crainic:1210.2277}), see also~\cite{Fischer:2021glc} and~\cite[Rem.~6.88]{Fischer:2022sus}. In the latter paper, this fact is discussed in more detail for the special case in which $\scG$ is a Lie group bundle, see also~\ref{ssec:Lie_group_bundles}. This shows how the global theory leads to the local theory.
    \end{remark}
    
    \begin{proposition}\label{prop:Cartan_connection_from_adjustment}
        An adjustment groupoid $\scA^\sfW$ of a Lie groupoid $\scG$ induces a corresponding Cartan connection on $\scG$.
    \end{proposition}
    \begin{proof}
        Consider the dg-Lie groupoid $\scA^\sfW$ satisfying the axioms of \ref{def:adjustment_of_groupoid}. Using the fact that $\sft\colon\scA^\sfW_1 \to \rmT[1]\frg[1]$ is a surjective submersion\footnote{This implies that it has local sections.} and the diagram \eqref{diag:preadjustment_groupoid}, one can show that $\scA^\sfW_1$ is locally isomorphic to $\rmT[1]\frg[1]\fibtimes{\psi}{\rmd\sft}\rmT[1]\sfG$. On the other hand, 
        \begin{equation}
            \big|\scA^\sfW_1\big|=\big|\rmT[1]\frg[1]\fibtimes{\psi}{\rmd\sft}\rmT[1]\sfG\big|=\sfG~,
        \end{equation}
        and altogether $\scA^\sfW_1$ and $\rmT[1]\frg[1]\fibtimes{\psi}{\rmd\sft}\rmT[1]\sfG$ are isomorphic. 
        
        Consider the source map $\sfs\colon\rmT[1]\frg[1]\fibtimes{\psi}{\rmd\sft}\rmT[1]\sfG \to \rmT[1]\frg[1]$. As this map is a dg-map, it is given by a graded map $\sfs_\flat\colon\rmT[1]\frg[1]\fibtimes{\psi}{\rmd\sft}\rmT[1]\sfG \to \frg[1]$. Furthermore, the body $|\sfs|$ of the map $\sfs$ must coincide with the source function $\sfs\colon\sfG\to M$. 
        
        Recall that the sheaf of functions $\caO_{\frg[1]}$ is generated by the space $\Gamma(\frg[1]^*)$. Therefore, the map of algebras $\sfs^\sharp\colon\caO_{\frg[1]}\to \caO_{\rmT[1]\frg[1]\fibtimes{\psi}{\rmd\sft}\rmT[1]\sfG}$ is given by the map of modules $\sfs^\sharp_1\colon\Gamma(\frg[1]^*)\to \left(\caO_{\rmT[1]\frg[1]\fibtimes{\psi}{\rmd\sft}\rmT[1]\sfG}\right)_1$. Moreover, one can show that
        \begin{equation}
            \left(\caO_{\rmT[1]\frg[1]\fibtimes{\psi}{\rmd\sft}\rmT[1]\sfG}\right)_1
            \cong
            \left(\caO_{(\rmT[1]M\oplus\frg[1])\times_{\rmT[1]M} \rmT[1]\sfG}\right)_1 
            \cong
            \left(\caO_{\sft^*\frg[1] \oplus \rmT[1]\sfG}\right)_1~,
        \end{equation}
        where $(-)_1$ denotes the truncation to the degree~1 elements. We conclude that the map $\sfs^\sharp_1$ is given by a map of vector bundles $\left(\sfs^\sharp_1\right)^\star\colon\sft^* \frg\oplus \rmT\sfG\to \frg$. 
        
        The functor $\Psi$ in \ref{def:adjustment_of_groupoid} shows that the diagram 
        \begin{equation}
            \begin{tikzcd}
                &\frg\fibtimes{\rho}{\rmd\sft}\rmT \sfG \arrow[r, hook] \arrow[rr,"\ractonB", bend right=30]& \sft^*\frg\oplus \rmT\sfG \arrow[r,"\left(\sfs^\sharp_1\right)^\star"] 
                &\frg 
            \end{tikzcd}
        \end{equation}
        is commutative.
        
        Define the vector bundle map $\omega\colon\rmT \sfG \to \sft^*\frg\oplus \rmT \sfG\to \frg$, where the first map is the obvious inclusion and the second map is $\left(\sfs^\sharp_1\right)^\star$. The diagram above shows that $\omega^*\colon\rmT \sfG\to \sfs^* \frg$ defines a splitting of the first exact sequence in \eqref{eq:ShortSequenceForHorLifts}. In other words, $\omega^*$ is a connection 1-form and the projection $\scL\circ \omega^*$ to $\ker(\rmd \sft)$ defines a distribution $\scH$.  Note that the map $\omega^*$ as splitting of the exact sequence is equivalent to the horizontal lift $\sigma_{\caH}\colon\sft^*\rmT M \to \rmT\sfG$, and we have $\omega^*\circ\sigma_{\caH}=0$. 
        
        Define the map $\rmAd^{-1}\colon \sft^*\frg \to \sft^*\frg\oplus \rmT\sfG \to \frg$, where the first map is the obvious inclusion and the second one is  $\left(\sfs^\sharp_1\right)^\star$. By linearity of the map $\left(\sfs^\sharp_1\right)^\star$ and the diagram above, we have
        \begin{equation}
            \begin{split}
                \ractonB (\sft(g),\nu; \sigma_{\caH}(g;\sft(g),\rho(\nu)))
                &= \rmAd^{-1}(\sft(g),\nu;g)+\omega^*\bigl( \sigma_{\caH}(g;\sft(g),\rho(\nu))\bigr)
                \\
                &= 
                \rmAd^{-1}(\sft(g),\nu;g)
                \\
                &=
                \rmd L_{g^{-1}}(\sfe_{\sft(g)},\nu)\circ \sigma_{\caH}(g;\sft(g),\rho(\nu))
                \\
                &= 
                \rmAd_{\caH}\left(g^{-1};\sft(g), \nu\right)~,
            \end{split}
        \end{equation}
        where we have used \ref{eq:adjoint_simplified} and the identity $\omega^*\circ\sigma_{\caH}=0$, and $\rmAd_{\caH}$ is the adjoint action defined in \ref{def:adjoint_action}.
        
        By linearity of $\left(\sfs^\sharp_1\right)^\star$, it follows that this map coincides with the one in \ref{eq:def_Weilified_action_2}. The connection $\caH$ has to be Cartan as $\scA^\sfW$ forms a groupoid.        
    \end{proof}
    
    \section{A new classifying space for principal bundles with connections}
    
    We now come to our new perspective on the cocycles underlying an ordinary principal bundle with connection, which allows for an elegant lift to the case of principal groupoid bundles with connection.
    
    \subsection{Atiyah algebroids from inner action groupoids}
    
    Consider a principal $\sfG$-bundle $P$ over some manifold $X$, where $\sfG$ is a Lie group with Lie algebra $\frg$. Recall that a connection on $P$ is equivalently a splitting of the short exact sequence of vector bundles~\cite{Atiyah:1957}
    \begin{equation}
        0\xrightarrow{~~~}P\times_\sfG\frg\xrightarrow{~~~}\frat(P)\xrightarrow{~\rho~} \rmT X\xrightarrow{~~~} 0~,
    \end{equation}
    where $P\times_\sfG\frg$ is the adjoint bundle of $P$, and $\frat(P)=\rmT P/\sfG$ together with the Lie bracket induced from $\rmT P$ and the anchor map $\rho$ is called the \uline{Atiyah algebroid}.
    
    We will generalize this picture in two ways. First of all we note that for the discussion of groupoid bundles with connections\footnote{as well as for a discussion of higher principal bundles with connections, cf.~\cite{future:2024aa}} a splitting as vector bundles unfortunately does not reproduce the gauge transformations. Instead, one has to regard the vector bundles in the sequence as Lie algebroids (e.g.~$\rmT X$ is here the tangent algebroid) and split in this category. This, in turn, only leads to flat connections, and we will have to extend this picture. Just as in the local description of~\ref{sec:local_description}, we will have to construct general connections as restricted higher flat connections. This construction will make use of adjustment Lie groupoids.
    
    Second, for performing concrete computations, e.g.~in a (physical) field theory having groupoid connections as kinematical data, it is much more convenient to switch to a local perspective in terms of cocycles and local connection forms. In this picture, we replace the manifold $X$ by the Morita-equivalent Čech groupoid $\check \scC(\sigma)$ of a surjective submersion $\sigma:Y\rightarrow X$. 
    
    Consider now a principal $\sfG$-bundle $P$ and a submersion $\sigma$ fine enough to represent $P$ by a cocycle (i.e.\ $P$ is subordinate to $\sigma$). A Čech cocycle, or transition functions, is then a morphism $g:\check \scC(\sigma)\rightarrow \sfB\sfG$, where $\sfB\sfG=(\sfG\rightrightarrows *)$ is the delooping of $\sfG$, i.e.~the one-object groupoid with $\sfG$ as morphisms. The functor $g$ is non-trivial only on morphisms $Y^{[2]}$, and we write $g_{y_1y_2}=g(y_1,y_2)$.
    
    In this situation, we can now replace $P$ by the Lie groupoid
    \begin{equation}
        \scR_\sigma(P)=\left(Y^{[2]}\times \sfG\rightrightarrows Y\times \sfG\right)
    \end{equation}
    with structure maps
    \begin{equation}
        \begin{gathered}
            \sft(y_1,y_2,h)=(y_1,h)~,~~~\sfs(y_1,y_2,h)=\left(y_2,g^{-1}_{y_1y_2}h\right)~,
            \\
            (y_1,y_2,h)\circ (y_2,y_3,\tilde h)=(y_1,y_3,h)~,
            \\
            \sfe(y,h)=(y,y,h)~,~~~(y_1,y_2,h)^{-1}=\left(y_2,y_1,g^{-1}_{y_1y_2}h\right)
        \end{gathered}
    \end{equation}
    for all $y\in Y$, $(y_1,y_2),(y_2,y_3)\in Y^{[2]}$ and $h,\tilde h\in \sfG$ with $\tilde h = g^{-1}_{y_1y_2} h$. Note that the groupoid $\scR_\sigma(P)$ is evidently Morita-equivalent to $P$, trivially regarded as a Lie groupoid.
    
    In this picture, the Atiyah algebroid is replaced by dg-Lie groupoid
    \begin{equation}
        \scR_\sigma(\frat(P))=\left(\rmT[1]Y^{[2]}\times \frg[1]\rightrightarrows \rmT[1]Y\times \frg[1]\right)
    \end{equation}
    with structure maps 
    \begin{equation}
        \begin{gathered}
            \sft(\hat y_1,\hat y_2,\nu)=(\hat y_1,\nu)~,~~~
            \sfs(\hat y_1,\hat y_2,\nu)=\left(\hat y_2,g^{-1}_{y_1y_2}\nu g_{y_1y_2}+g^{-1}_{y_1y_2}\rmd g_{y_1y_2}\right)~,
            \\
            (\hat y_1,\hat y_2,\nu)\circ (\hat y_2,\hat y_3,\tilde \nu)=(\hat y_1,\hat y_3,\nu)~,
            \\
            \sfe(\hat y,\nu)=(\hat y,\hat y,\nu)~,~~~
            (\hat y_1,\hat y_2,\nu)^{-1}=
            (\hat y_2,\hat y_1,g^{-1}_{y_1y_2}\nu g_{y_1y_2}+g^{-1}_{y_1y_2}\rmd g_{y_1y_2})
        \end{gathered}
    \end{equation}
    for all $\hat y\in \rmT[1]Y$, $(\hat y_1,\hat y_2),(\hat y_2,\hat y_3)\in \rmT[1]Y^{[2]}$, $y_i$ the projection onto the base of $\hat y_i$, and $\nu\in \frg[1]$. The differentials on morphisms and objects are given by 
    \begin{equation}
        \begin{aligned}
            \sfd_{\rmT[1]Y^{[2]}\times \frg[1]}&=\rmd_{Y^{[2]}}\otimes \rmid +\rmid\otimes \sfd_{\sfCE(\frg)}~,
            \\
            \sfd_{\rmT[1]Y\times \frg[1]}&=\rmd_{Y}\otimes \rmid +\rmid\otimes \sfd_{\sfCE(\frg)}~,
        \end{aligned}
    \end{equation}
    where $\rmd_X$ is the de~Rham differential on $X$ and $\sfd_{\sfCE(\frg)}$ is the Chevalley--Eilenberg differential of $\frg$, and they act on elements of $\Omega^\bullet(Y^{[2]})\otimes C^\infty(\frg[1])$ and $\Omega^\bullet(Y)\otimes C^\infty(\frg[1])$, respectively. We call $\scR_\sigma(\frat(P))$ the \uline{resolved Atiyah algebroid}, and $\scR_\sigma(\frat(P))$ is Morita-equivalent to the dg-manifold $\frat(P)[1]$, trivially regarded as a dg-Lie groupoid.
    
    This resolved Atiyah algebroid can now be constructed as a pullback of the inner action groupoid $\scA(\sfB\sfG)$. Let us first make the formulas from \ref{prop:inner_action_groupoid} explicit.
    \begin{lemma}\label{lem:inner_action_for_BG}
        The inner action groupoid $\scA(\sfB\sfG)$ for $\sfG$ a Lie group is the dg-Lie groupoid
        \begin{equation}
            \scA(\sfB\sfG)=\Big(~\frg[1]\times \rmT[1]\sfG~\rightrightarrows~\frg[1]~\Big)
        \end{equation}
        with structure maps
        \begin{equation}
            \begin{gathered}
                \sft(\nu;g,\gamma)=\nu~,~~~\sfs(\nu;g,\gamma)=g^{-1}\nu g+g^{-1}\gamma~,
                \\
                (\nu;g_1,\gamma_1)\circ\left(g_1^{-1}\nu g_1+g_1^{-1}\gamma_1; g_2, \gamma_2\right)=(\nu;g_1g_2,g_1\gamma_2+\gamma_1g_2)
                \\
                \sfe(\nu)=(\nu;\unit,0)~,~~~(\nu;g,\gamma)^{-1}=\left(g^{-1}\nu g+g^{-1}\gamma;g^{-1},-g^{-1}\gamma g^{-1}\right)~,
            \end{gathered}
        \end{equation}
        for $\nu\in \frg[1]$, $g\in \sfG$ and $\gamma\in \rmT[1]_g\sfG$, where $\unit$ is the unit element of $\sfG$, and the differential is given by the Chevalley--Eilenberg differential on $\frg[1]$ and the de~Rham differential on $\sfG$.
    \end{lemma}
    \begin{proof}
        This follows by specializing \ref{prop:inner_action_groupoid}; see also~\cite[Section 4.2]{Jurco:2014mva}. 
    \end{proof}
    
    \begin{theorem}\label{thm:resolved_Atiyah_from_pullback}
        Let $P$ be a principal $\sfG$-bundle over $X$ for $\sfG$ a Lie group with transition functions $g:Y^{[2]}\rightarrow \sfG$ subordinate to a surjective submersion $\sigma:Y\rightarrow X$. The resolved Atiyah algebroid $\scR_\sigma(\frat(P))$ is the pullback of the inner action groupoid $\scA(\sfB\sfG)$ along the functor $\rmd g:\rmT[1]\check \scC(\sigma)\rightarrow \rmT[1]\sfB\sfG$: 
        \begin{equation}\label{eq:pullback_resolved_Atiyah}
            \begin{tikzcd}
                \scR_\sigma(\frat(P)) \dar["\rho"] \rar["\hat g"] & \scA(\sfB\sfG) \dar["\varrho"] \\
                \rmT[1]\check \scC(\sigma) \rar["\rmd g"] & \rmT[1]\sfB\sfG
            \end{tikzcd}
        \end{equation}
    \end{theorem}
    \begin{proof}
        Consider a general cone $\caC=(\caC_1\rightrightarrows \caC_0)$, and fit it into the diagram 
        \begin{equation}
            \begin{tikzcd}
                &\caC  \arrow[ddr,bend right=30, "\Psi",swap] \arrow [drr,bend left=30, "\Phi"] \arrow[dr, dotted, "\exists !"]
                &
                &\\
                &
                &  \scR_\sigma(\frat(P)) \dar["\rho"] \rar["\hat g"] & \scA(\sfB\sfG) \dar["\varrho"] \\
                &
                & \rmT[1]\check \scC(\sigma)  \rar["\rmd g"] & \rmT[1]\sfB\sfG
            \end{tikzcd}
        \end{equation}
        Denote by $\Psi_0,\Psi_1$ and $\Phi_0,\Phi_1$ the maps in the functors $\Psi$ and $\Phi$ on objects and morphisms, respectively. The unique functor $\Upsilon:\caC\rightarrow\scR_\sigma(\frat(P))$ is then given by 
        \begin{equation}
            \begin{aligned}
                \Upsilon_0&=\Psi_0 \times \Phi_0:&\caC_0&\rightarrow \rmT[1]Y \times \frg[1]~,
                \\
                \Upsilon_1&=\Psi_1\times(\sfpr_\frg\circ \Phi_{1}):& \caC_1&\rightarrow \rmT[1] Y^{[2]}\times \frg[1]~,
            \end{aligned}
        \end{equation}
        where we used the evident projection $\sfpr_\frg:\frg[1]\times \rmT[1]\sfG\rightarrow \frg[1]$. Note that compatibility with the differential is evident for these maps.
    \end{proof}
    \begin{remark}
        In the diagram \eqref{eq:pullback_resolved_Atiyah}, the weak pullback coincides with the strict pullback which can be proved abstractly. However, we omit this not to complicate the discussion.
    \end{remark}
    \noindent We note that the construction as pullback produces the dg-Lie groupoid morphism $\rho$, which plays the role of the anchor functor. 
    
    \subsection{Principal bundles with connections as dg-Lie groupoid morphisms}
    
    We then have the following, expected result. 
    \begin{proposition}\label{prop:flat_connections_Atiyah}
        Consider again the situation of \ref{thm:resolved_Atiyah_from_pullback}. A section of the dg-Lie groupoid morphism $\rho:\scR_\sigma(\frat(P))\rightarrow \rmT[1]\check \scC(\sigma)$ describes a flat connection on $P$.
    \end{proposition}
    \begin{proof}
        On the objects of the dg-Lie groupoid, such a section $\caA$ is a morphism of dg-graded manifolds \(\caA:\rmT[1]Y\to\rmT[1]Y\times\frg[1]\), which restricts to the identity on $\rmT[1]Y$. The remaining freedom is a map $\rmT[1]Y\rightarrow \frg[1]$, which amounts to the map $\caA^\sharp:\frg[1]^*\rightarrow \Omega^1(Y)$. That is, $\caA$ is of the form
        \begin{equation}
            (y,\bar y)\mapsto \bigl(y,\bar y,A_y(\bar y)\bigr)
        \end{equation}
        for all $y\in Y$ and $\bar y\in T_yX$, where $A\in \Omega^1(Y)\otimes \frg[1]$. On morphisms, the compatibility forces $\caA$ to be of the form
        \begin{equation}
            (y_1,y_2,\bar y_1,\bar y_2)\mapsto \bigl(y_1,y_2,\bar y_1,\bar y_2,A_{y_1}(\bar y_1)\bigr)
        \end{equation}
        with 
        \begin{equation}
            A_{y_2}=g_{y_1y_2}^{-1}A_{y_1}g_{y_1y_2}+g_{y_1y_2}^{-1}\rmd g_{y_1y_2}
        \end{equation}
        for all $(y_1,y_2)\in Y^{[2]}$ and $\bar y_{1,2}\in \rmT[1]_{y_{1,2}}Y$. Because $\caA$ is a dg-functor, we require
        \begin{equation}
            \sfd_{\rmT[1]\check \scC(\sigma)}\circ\caA^*-\caA^*\circ\sfd_{\scR_\sigma(\frat(P))}=0~,
        \end{equation}
        which amounts to
        \begin{equation}
            F_y\coloneqq \rmd A_y+\tfrac12[A_y,A_y]=0
        \end{equation}
        for all $y\in Y$. Altogether, $\caA$ encodes the differential refinement of $g$ by a local flat connection one-form.
    \end{proof}
    
    Diagram~\eqref{eq:pullback_resolved_Atiyah} makes it evident that the resolved Atiyah algebroid in the above construction is somewhat superfluous, as any section of $\rho$ can be postcomposed by $\hat g$ to a dg-Lie groupoid morphism from $T[1]\check \scC(\sigma)$ to $\scA(\sfB\sfG)$, and $\scA(\sfB\sfG)$ can be seen as the classifying space for principal $\sfG$-bundles with flat connections.
    \begin{theorem}\label{thm:cocycle_bundles_flat_connection}
        Consider a manifold $X$ together with a surjective submersion $\sigma\colon Y\rightarrow X$. The differential cocycle of a principal $\sfG$-bundle with flat connection subordinate to $\sigma$ is given by a dg-Lie groupoid functor from $\rmT[1]\check\scC(\sigma)$ to $\scA(\sfB\sfG)$. Isomorphisms between two such principal $\sfG$-bundles with flat connections, also known as gauge transformations, amount to natural isomorphisms between the corresponding functors.
    \end{theorem}
    \begin{proof}
        A dg-functor $\Psi\colon\rmT[1]\check\scC(\sigma)\rightarrow \scA(\sfB\sfG)$ contains a dg-functor $\psi\colon\rmT[1]\check\scC(\sigma)\rightarrow \rmT[1]\sfB\sfG$. By~\ref{rem:graded-map-to-differential-map}, this functor is defined by a functor $\rmT[1]\check\scC(\sigma)\rightarrow \sfB\sfG$, which, for degree reasons, is given by a functor $g\colon\check\scC(\sigma)\rightarrow \sfB\sfG$, with the functor $\psi$ the differential $\psi=\rmd g$. As discussed above, $g$ encodes the Čech cocycle of a principal $\sfG$ bundle subordinate to $\sigma$.
        
        As in~\Cref{prop:flat_connections_Atiyah}, the dg-map on objects $\Psi_{0}:\rmT[1]Y\rightarrow \frg[1]$ encodes a local connection 1-form, which is flat. Compatibility with the structure maps in both groupoids then glues together the local connection one forms as in~\Cref{prop:flat_connections_Atiyah}.
        
        A natural isomorphism between two such dg-Lie functors $\Psi$ and $\tilde \Psi$ is a dg-morphism $\alpha:\rmT[1]Y\rightarrow \frg[1]\times \rmT[1]\sfG$, which is necessarily of the form $0\times \rmd p$ for some map $p:Y\rightarrow \sfG$. The fact that this is a natural isomorphism then translates to 
        \begin{equation}
            \tilde g_{y_1y_2}=p_{y_1}^{-1} g_{y_1y_2}p_{y_2}\eand 
            \tilde A_y=p_y^{-1}A_y p_y+p_y^{-1}\rmd p_y~,
        \end{equation}
        where we used a similar notation as in the proof of \Cref{prop:flat_connections_Atiyah}.
        Altogether, we recover the usual bundle isomorphisms in terms of differential cocycles.
    \end{proof}
    
    As done in the local description~\ref{sec:local_description}, we now obtain general connections on principal bundles as higher connections with restricted gauge transformations. The relevant classifying space here is precisely the adjustment groupoid. We start by specializing the adjustment groupoid of \ref{ssec:construction_adjustment_groupoid} to the case $\scG=\sfB\sfG$. Note that the connection on $\sfB\sfG$ is necessarily trivial.
    \begin{lemma}
        An adjustment groupoid of $\sfB\sfG$ for $\sfG$ a Lie group is the dg-Lie groupoid
        \begin{equation}
            \scA^\sfW_\sfG=\Big(~\rmT[1]\frg[1]\times \rmT[1]\sfG\rightrightarrows \rmT[1]\frg[1]~\Big)
        \end{equation}
        with structure maps
        \begin{equation}
            \begin{gathered}
                \sft(\nu,\bar \nu;g,\gamma)=(\nu,\bar \nu)~,~~~
                \sfs(\nu,\bar \nu;g,\gamma)=\left(g^{-1}\nu g+g^{-1}\gamma,
                g^{-1}\bar \nu  g\right)~,
                \\
                (\nu ,\bar \nu ;g_1,\gamma_1)\circ \left(g_1^{-1}\nu g_1+g_1^{-1}\gamma_1,
                g^{-1}_1\bar \nu  g_1;g_2,\gamma_2\right)=(\nu ,\bar \nu ;g_1g_2,g_1\gamma_2+\gamma_1g_2)~,
                \\
                \sfe(\nu ,\bar \nu )=(\nu ,\bar \nu ;\unit,0)~,~~~(\nu ,\bar \nu ;g,\gamma)^{-1}=\left(g^{-1}\nu g+g^{-1}\gamma,
                g^{-1}\bar \nu  g;g^{-1},-g^{-1}\gamma g^{-1}\right)
            \end{gathered}
        \end{equation}
        for $\nu \in \frg[1]$, $\bar \nu \in \rmT[1]_\nu \frg[1]$, $g,g_{1,2}\in \sfG$, $\gamma,\gamma_{1,2}\in \rmT[1]_{g,g_{1,2}}\sfG$ and differentials on morphisms and objects
        \begin{equation}
            \begin{aligned}
                \sfd_{\rmT[1]\frg[1]\times \rmT[1]\sfG}&=\sfd_{\sfW(\frg)}\otimes \rmid +\rmid\otimes \rmd_\sfG~,
                \\
                \sfd_{\rmT[1]\frg[1]}&=\sfd_{\sfW(\frg)}~,
            \end{aligned}
        \end{equation}
        where $\sfd_{\sfW(\frg)}$ is the differential in the Weil algebra of $\frg$ and $\rmd_\sfG$ is the de~Rham differential on $\sfG$, and they act on elements of $\sfW(\frg)\otimes\Omega^\bullet(\sfG)$. This form is unique up to isomorphisms preserving the embedding of $\frg[1]\rightarrow \rmT[1]\frg[1]\times \rmT[1]\sfG$.
    \end{lemma}
    
    \begin{proof}
        This is the specialization of \Cref{The:Adjustment_groupoid_classsification}. Once it is noted that there is a unique Cartan connection (which is trivial) on a Lie group, the proof is complete.
        
    \end{proof}
    
    The groupoid $\scA^\sfW_\sfG$ can now indeed be regarded as the classifying space for principal $\sfG$-bundles with connection.
    \begin{theorem}\label{thm:cocycle_bundles_general_connection}
        The differential cocycles of a principal $\sfG$-bundle with connection over a manifold $X$ subordinate to a cover $\sigma\colon Y\rightarrow X$ are given by morphisms of dg-Lie groupoids
        \begin{equation}
            \Phi\colon\rmT[1]\check \scC(\sigma)\rightarrow \scA^\sfW_\sfG~.
        \end{equation}
        Isomorphisms between two such principal $\sfG$-bundles with connections, commonly known as gauge transformations, amount to natural isomorphisms between the corresponding functors.
    \end{theorem}
    \begin{proof}
        Just as for flat connections, such a morphism $\Phi$ contains the differential of a Čech cocycle $g:\check \scC(\sigma)\rightarrow \sfB\sfG$, cf.~the proof of \ref{thm:cocycle_bundles_flat_connection}. The remaining information in $\Phi$ is a map from $\rmT[1]Y$ to $\rmT[1]\frg$ or, dually, a morphism of commutative differential algebras $\caA:\sfW(\frg)\rightarrow \Omega^\bullet(Y)$. Such a map is defined by its image on generators, and we introduce $A\in \Omega^1(Y)\otimes \frg$ and $F\in \Omega^2(Y)\otimes \frg$ by
        \begin{equation}
            A(\xi[1])\coloneqq\caA(\xi)\in \Omega^1(Y)\eand F(\bar \xi[2])\coloneqq \caA(\bar \xi)\in \Omega^2(Y)
        \end{equation}
        with $\xi\in \frg[1]^*$ and $\bar \xi\in \frg[2]^*$. Compatibility with the differential leads to the relation
        \begin{equation}
            F_y=\rmd A_y+\tfrac12[A_y,A_y]~
        \end{equation}
        for all $y \in Y$.
        Functoriality of $\Phi$ then implies that this data is correctly glued together, that is
        \begin{equation}
            A_{y_2}=g_{y_1y_2}^{-1}A_{y_1}g_{y_1y_2}+g_{y_1y_2}^{-1}\rmd g_{y_1y_2}
            \eand
            F_{y_2}=g_{y_1y_2}^{-1}F_{y_1}g_{y_1y_2}~
        \end{equation}
        for all $(y_1, y_2) \in \check \scC(\sigma)$,
        and there are no further conditions.
        
        The fact that bundle isomorphisms amount to natural transformations follows then fully analogously to \ref{thm:cocycle_bundles_flat_connection}.
    \end{proof}

    \section{Principal groupoid bundles with connections}\label{sec:principal_groupoid_bundles}
    
    We now have everything at hand to generalize the discussion of the previous section to arbitrary principal groupoid bundles.
    
    \subsection{Principal groupoid bundles}
    
    Let us briefly review principal groupoid bundles, both from the total space and the Čech cocycle perspective. Given a Lie groupoid $\scG=(\sfG\rightrightarrows M)$, a \uline{principal $\scG$-bundle} over a manifold $X$ is a manifold $P$ carrying a principal fiberwise action of $\scG$ along the \uline{lifted Higgs field}, which is a map $\hat \phi:P\rightarrow M$, cf.~\cite[Section 5.7]{Moerdijk:2003bb}. 
    
    For our purposes, the description in terms of differential cocycles is sufficient:
    \begin{definition}[Čech cocycles and coboundaries of a principal groupoid bundle]\label{def:principal_groupoid_bundles_cocycles}
        Given a Lie groupoid $\scG=(\sfG\rightrightarrows M)$, the Čech cocycles of a principal $\scG$-bundle over a manifold $X$ subordinate to a surjective submersion $\sigma:Y\rightarrow X$ are functors $\Phi:\check \scC(\sigma)\rightarrow \scG$. The Čech coboundaries between two such Čech cocycles $\Phi$ and $\tilde \Phi$ are natural isomorphisms $\Phi \Rightarrow \tilde\Phi$. We call the groupoid $\scG$ the structure groupoid of the principal $\scG$-bundle.
    \end{definition}
    
    Explicitly, such a functor corresponds to maps
    \begin{subequations}\label{eq:cech_cocycles_groupoidbundle}
        \begin{equation}
            g:Y^{[2]}\rightarrow \sfG\eand m: Y\rightarrow M
        \end{equation}
        such that
        \begin{equation}
            g_{y_1y_3}=g_{y_1y_2}g_{y_2y_3}~,~~ m_{y_1}=\sft(g_{y_1y_2})~,~~m_{y_2}=\sfs(g_{y_1y_2})~.
        \end{equation}
    \end{subequations}    
    We will call the function $m$ the \uline{Higgs field}. 
    
    As for principal bundles, one can go back and forth between the total space description and the cocycle description by introducing a local trivialization of the groupoid bundle. 
    
    Note that a principal groupoid bundle for the groupoid $\sfB\sfG=(\sfG\rightrightarrows *)$ is the same as a principal bundle with structure group $\sfG$. Also, we can define the trivial principal groupoid bundle for a global Higgs field $\phi\colon X\rightarrow M$ by the cocycle  $m=\phi\circ \sigma$ and $g_{y_1y_2}=\sfe_{m_{y_2}}$ for all $(y_1, y_2) \in \check \scC(\sigma)$; $P$ as bundle is then defined as the pullback $\phi^*\sfG$ of $\sft\colon \sfG \to M$ along $\phi$, equipped with the evident structure, in particular, $\hat \phi$ is defined as the evident projection $\phi^* \sfG \to \sfG$ followed by $\sfs$.
    
    \begin{remark}
        For ordinary principal bundles, we usually use the terms ``gauge group'' and ``structure group'' interchangeably. For principal groupoid bundles, however, we exclusively reserve the term ``structure groupoid'' for the groupoid $\scG$ because the term ``gauge groupoid'' is already in use for the transport or Atiyah groupoid of a principal bundle $P$, which differentiates to the Atiyah algebroid of $P$. 
    \end{remark}
    
    \subsection{Flat and adjusted connections on principal groupoid bundles}
    
    We now come to the key result of this paper, the global definition of flat and in particular adjusted connections for principal groupoid bundles. 
    
    We start with flat connections. Here, we can turn the obvious generalization of \ref{thm:cocycle_bundles_flat_connection} into a definition:
    \begin{definition}[Differential cocycles for flat principal groupoid bundles]
        Consider a manifold $X$ together with a surjective submersion $\sigma:Y\rightarrow X$ as well as a groupoid $\scG=(\sfG\rightrightarrows M)$ with corresponding inner action groupoid $\scA(\scG)$. A differential cocycle of a flat principal $\scG$-bundle subordinate to $\sigma$ is a dg-Lie groupoid functor from $\rmT[1]\check\scC(\sigma)$ to $\scA(\scG)$. Isomorphisms between two such principal $\scG$-bundles with flat connections, also known as gauge transformations, are natural isomorphisms between the corresponding functors.
    \end{definition}
    
    Similarly, we can define non-flat, adjusted connections, turning the obvious generalization of \ref{thm:cocycle_bundles_general_connection} into a definition.
    \begin{definition}[Differential cocycles for principal groupoid bundles]\label{def:cocycles_groupoid_bundles_adjusted_connection}
        Consider a manifold $X$ together with a surjective submersion $\sigma:Y\rightarrow X$ as well as a groupoid $\scG=(\sfG\rightrightarrows M)$. Let $\scA^\sfW$ be an adjustment groupoid of $\scG$. A differential cocycle of a principal $\scG$-bundle with $\scA^\sfW$-adjusted connection subordinate to $\sigma$ is a dg-Lie groupoid functor from $\rmT[1]\check\scC(\sigma)$ to $\scA^\sfW$. Isomorphisms between two such principal $\scG$-bundles with adjusted connections, also known as gauge transformations, are natural isomorphisms between the corresponding functors.
    \end{definition}
    
    The fact that the above definitions generalize \ref{thm:cocycle_bundles_flat_connection} and \ref{thm:cocycle_bundles_general_connection} immediately implies that they correctly reproduce ordinary principal bundles with flat and general connections for $\scG=\sfB\sfG$ with $\sfG$ some Lie group.
    
    Let us verify that we indeed reproduce local connections and their gauge transformations. It is sufficient to consider the adjusted case, as the flat case is then automatically induced.
    \begin{proposition}
        Consider the situation of \ref{def:cocycles_groupoid_bundles_adjusted_connection}. Over a patch $Y$ and in terms of the coordinates introduced in \ref{ssec:Weil_algebra}, an $\scA^\sfW$-adjusted connection is described by the data 
        \begin{equation}\label{eq:kin_data}
            \phi \in C^\infty(Y,M)~,~~~A\in \Omega^1(Y,\phi^*\frg)~.
        \end{equation}
        Infinitesimal gauge transformations act as
        \begin{equation}
            \delta \phi^a =\ttr^a_\alpha c_{\frg}^\alpha~,~~~\delta A^\alpha =\rmd c_\frg^\alpha+\ttf^\alpha_{\beta\gamma}A^\beta c_\frg^\gamma+
            \omega_{a\beta}^\alpha E^a c_\frg^\beta
        \end{equation}
        with $c_\frg\in \Omega^0(Y,\phi^*\frg)$ and $E$ the covariant derivative of $\phi$, cf.~\eqref{eq:local_adjusted_connection}.
    \end{proposition}
    \begin{proof}
        The kinematical data is evident, because on objects, the dg-Lie groupoid functor $\Psi:\rmT[1]\check\scC(\sigma)\rightarrow \scA^\sfW$ reduces to a map $\caA:\rmT[1]Y\rightarrow \rmT[1]\frg[1]$, which by~\ref{rem:graded-map-to-differential-map} is fully determined by a map of graded manifolds $\caA_0:\rmT[1]Y\rightarrow \frg[1]$. The data specifying this map is precisely~\eqref{eq:kin_data}.
        
        The fact that the gauge transformations are of this form follows directly from \ref{eq:adjustment_of_groupoid} and \Cref{def:adjusted_BRST}, together with the derivation of adjusted gauge transformations in \ref{ssec:local_adjusted_connections}.
    \end{proof}
    
    \begin{lemma}[Adjustment groupoid]\label{eq:adjustment_of_groupoid}
        For any Cartan Connection $\caH$ on a Lie groupoid $\scG$, there is an isomorphism of presheaves consisting of dg-maps
        \begin{equation}\label{eq:adjustment_condition}
            \caN\mapsto \ihom^{\rm red}_{\nabla_{\caH}}(\caN,\rmT[1]\frg[1])\eand \caN\mapsto \sfLie\left(\catdgGrpd\left(\caN\rightrightarrows \caN,\scA^\sfW_{\caH}\right)\right)~,
        \end{equation}
        where $\nabla_{\caH}$ is the Cartan connection on the Lie algebroid $\frg$ induced by $\caH$. 
    \end{lemma}
    
    \begin{proof}
        This requires a short characterization of the first presheaf and a rather lengthy computation for the second presheaf.
        
        \textit{Step 1.}	
        For $\caX$ an N$Q$-manifold, each element $f$ of $\sfHom(\caX,\ihom^{\rm red}_{\nabla}(\caN,\rmT[1]\frg[1]))$ is a dg-map $f:\caX\times \caN\to \rmT[1]\frg[1]$ such that $(\sfp_{\rm curv}\circ f)^\sharp$ takes values in $C^\infty(\caX)_0\otimes C^\infty(\caN)$, where a subscript restricts the ring of functions to the subset with elements of a particular degree. By \ref{rem:graded-map-to-differential-map}, the dg-map $f$ is given by the lift of a graded map $f_{\flat}:\caX\times \caN\to \frg[1]$. As $C^\infty(\frg[1])$ is generated by $C^\infty(M)$ and $\Gamma(\frg[1]^*)$, such a map is characterized by
        \begin{itemize}
            \item a topological map
            \begin{equation}
                |f_{\flat}|:|\caX|\times |\caN|\rightarrow M
            \end{equation}
            \item and two maps of modules over $ C^\infty(M)$,
            \begin{equation}
                \begin{aligned}
                    f^\sharp_{\flat 0}:&& \Gamma(\frg[1]^*)&\rightarrow |f|_*\left(C^\infty(\caX)_0\otimes C^\infty(\caN)_1\right)~,
                    \\
                    f^\sharp_{\flat 1}:&& \Gamma(\frg[1]^*)&\rightarrow |f|_*\left(C^\infty(\caX)_1\otimes C^\infty(\caN)_0\right)~,
                \end{aligned}
            \end{equation}
            subject to conditions explained below and listed in~\eqref{eq:brstproof1}.
        \end{itemize}
        
        The additional conditions ensure that $f^\sharp \circ \sfp_{\rm curv}^\sharp$ only takes values in  $C^\infty(\caX)_0\otimes C^\infty(\caN)$. The lift $f$ can be explicitly calculated by imposing that $f$ is a dg-map. In the local coordinates $(m^a,\bar m^a,\xi^\alpha,\bar \xi^\alpha)$ and the structure functions introduced in~\eqref{eq:coordinate_changed_Weil}, $f$ is given by $|f|=|f_\flat|$ and
        \begin{equation}\label{eq:brstproof0}
            \begin{split}
                f^\sharp(m^a)&=f^\sharp_{\flat }(m^a)
                \\
                &=   |f_{\flat}|^*(m^a)~,
                \\
                f^\sharp(\xi^\alpha)&=   f^\sharp_{\flat }(\xi^\alpha)
                \\
                &=   f^\sharp_{\flat 0}(\xi^\alpha) +  f^\sharp_{\flat 1}(\xi^\alpha)~,
                \\
                f^\sharp(\bar m^a)&=(\rmd_{\caX}+\rmd_{\caN})    f^\sharp_{\flat}(m^a)-     f^\sharp_{\flat}(\rmd_{\frg[1]}(m^a))
                \\
                &=
                (\rmd_{\caX}+\rmd_{\caN})   f^\sharp_{\flat}(m^a)- |f_{\flat}|^*\ttr^a_\alpha \left(   \tilde f^\sharp_{\flat 0}(\xi^\alpha) +   f^\sharp_{\flat 1}(\xi^\alpha)\right)~,
                \\
                f^\sharp(\bar \xi^\alpha)&=(\rmd_{\caX}+\rmd_{\caN})   f^\sharp_{\flat}(\xi^\alpha)-     f^\sharp_{\flat}(\rmd_{\frg[1]}(\xi^\alpha))~.
            \end{split}
        \end{equation}
        
        For $f^\sharp \circ \sfp_{\rm curv}^\sharp$ to only take values in $C^\infty(\caX)_0\otimes C^\infty(\caN)$, we need to impose the following relations:
        \begin{equation}\label{eq:brstproof1}
            \begin{gathered}
                \sfd_{\caX} |f|^*h= \left( f^\sharp_1\circ \rho^*\right)(\rmd h)~,
                \\
                \sfd_{\caX}f_0^\sharp\left(\xi^\alpha\right) + |f|^*\left(\ttf_{\beta\gamma}^\alpha\right) f^\sharp_1\left(\xi^\beta\right) f^\sharp_0\left(\xi^\gamma\right)+|f|^*\left(\omega_{a\beta}^\alpha\right) f^\sharp\left(\bar m^a\right)\tilde f^\sharp_1\left(\xi^\beta\right)+ \sfd_{\caN} f^\sharp_1\left(\xi^\alpha\right)=0~,
                \\
                \sfd_{\caX} f^\sharp_1\left(\xi^\alpha\right)=-\tfrac{1}{2}|f|^*\left(\sff_{\beta\gamma}^\alpha\right) f^\sharp_1\left(\xi^\beta\right) f^\sharp_1\left(\xi^\gamma\right)
            \end{gathered}
        \end{equation}
        for all $h\in C^\infty(M)$.

        \textit{Step 2.} On the other hand, $\sfHom(\caX,\sfLie(\catdgGrpd(\caN\rightrightarrows \caN,\scA^\sfW_{\caH})))$ is given by\footnote{A computation similar to the following one is found in~\cite[Sec.~8.4.3.2]{Li:2014}.}
        \begin{equation}
            \sfHom\left(\caX,\ihom_{\catdgGrpd}\left(\scP\mathsf{air}(\Theta),\catdgGrpd\left(\caN\rightrightarrows \caN,\scA^\sfW_{\caH}\right)\right)\right)~.
        \end{equation}
        To compute this presheaf, we have to perform a computation very similar to the one below~\ref{prop:Severa_diff}, but lifted to differential graded manifolds, and we will be more precise in the following.
        
        Note that we cannot quite use inverse currying, since the space $\catdgGrpd(\caN\rightrightarrows \caN,\scA^\sfW_{\caH})$ is only the degree~0 part of the inner hom. Therefore, the above inner hom is given by dg-functors 
        \begin{equation}\label{eq:fun_F}
            F:\Big(\caX\times \Theta \times \Theta \times \caN \rightrightarrows \caX\times \Theta \times \caN\Big)~~\rightarrow~~\scA^\sfW_{\caH}
        \end{equation}
        that factor through another groupoid:
        \begin{equation}\label{diag:degree_0_condition_groupoid}
            \begin{tikzcd}[column sep=3cm]
                \caX\times \Theta \times \Theta \times \caN \rightrightarrows \caX\times \Theta \times \caN \arrow[r, "F"] \arrow[d, "\sfp_{0}\times \rmid"']
                &\scA^\sfW_{\caH}      
                \\
                (\caX\times \Theta \times \Theta )_0 \times \caN \rightrightarrows (\caX\times \Theta)_0 \times \caN  \arrow[ru, "\check F"']
                &
            \end{tikzcd}
        \end{equation}
        where the topological part $|\sfp_0|$ of $\sfp_0$ is identity, $\sfp_0^\sharp$ is the evident inclusion of degree~$0$ elements in the local ring of functions, and $\check F$ is some morphism of graded manifolds. In the following, we will show that such functors are given by 
        \begin{itemize}
            \item a topological map
            \begin{equation}
                |F|:|\caX|\times |\caN| \to M
            \end{equation}
            \item and two maps of modules over $C^\infty(M)$ 
            \begin{equation}
                \begin{aligned}
                    (\bar F_0)^\sharp:&&\Gamma(\frg[1])^* &\to |F|_*\left(C^\infty(\caX)_0\otimes C^\infty(\caN)_1\right)~,
                    \\
                    (\bar F_1)^\sharp:&&\Gamma(\frg[1])^* &\to |F|_*\left(C^\infty(\caX)_1\otimes C^\infty(\caN)_0\right)~,
                \end{aligned}
            \end{equation}
            subject to further constraints which are the same as~\eqref{eq:brstproof1} after the identifications $|F|=|f|$, $(\bar F_1)^\sharp=f_1^\sharp$ and $(\bar F_0)^\sharp=f_0^\sharp$.
        \end{itemize}
        
        A general functor $F$ as in~\eqref{eq:fun_F} is given by a pair of dg-maps between morphisms and objects,
        \begin{equation}
            \begin{aligned}
                F_1:&&\caX\times \Theta \times \Theta \times \caN &\to \rmT[1]\frg[1]\fibtimes{\psi}{\rmd\sft}\rmT[1]\sfG~,
                \\
                F_0:&&\caX\times \Theta \times \caN &\to \rmT[1]\frg[1]~,
            \end{aligned}
        \end{equation}
        satisfying functoriality conditions. The dg-map $F_0:\caX\times \Theta\times \caN\rightarrow \rmT[1]\frg[1]$ is determined by uniquely lifting a graded map $F_{\flat 0}:\caX\times \Theta\times \caN\rightarrow \frg[1]$ which is given by a topological map $|F_{\flat 0}|=|F|$ as well as maps
        \begin{equation}
            \begin{aligned}
                \left(F_{\flat 0}^\sharp\right)_0:&&C^\infty(M) &\to |F|_*\left(C^\infty(\caX)_0\otimes C^\infty(\caN)_0~\oplus~ \theta C^\infty(\caX)_1\otimes C^\infty(\caN)_0\right)~,
                \\
                \left(F_{\flat 0}^\sharp\right)_1:&&\Gamma(\frg[1])^* &\to |F|_*\left(C^\infty(\caX)_0\otimes C^\infty(\caN)_1 ~\oplus~\theta C^\infty(\caX)_1\otimes C^\infty(\caN)_1\right)~,
            \end{aligned}
        \end{equation}
        where $(F_{\flat 0}^\sharp)_0$ is a ring homomorphism, $(F_{\flat 0}^\sharp)_1$ is a morphism of modules over $C^\infty(M)$ and $\theta$ is the generator of $C^\infty(\Theta)$ with $\rmd_{\Theta}\theta=1$. The map $(F_{\flat 0}^\sharp)_0$ is calculated similarly to~\eqref{eq:brstproof0}. Because the images of $F_{\flat 0}^\sharp$ on the ``curvature generators'' $\bar m^a$ and $\bar \xi^\alpha$ do not belong to 
        \begin{equation}
            (C^\infty(\caX)_0\otimes C^\infty(\caN)) \oplus \theta( C^\infty(\caX)_1\otimes C^\infty(\caN))=\big(C^\infty(\Theta \times \caX)\big)_0\otimes C^\infty(\caN)~,
        \end{equation}
        condition~\eqref{diag:degree_0_condition_groupoid} translates to $(\rmd_{\caX}+\rmd_{\Theta}) (F_{\flat 0}^\sharp)_0=0$ and  $(\rmd_{\caX}+\rmd_{\Theta}) (F_{\flat 0}^\sharp)_1=0$. Therefore, we have
        \begin{equation}\label{eq:brstproof2}
            \begin{aligned}
                \left(F_{\flat 0}^\sharp\right)_0(h)&= |F|^*h- \theta \sfd_{\caX}|F|^*h~,
                \\
                \left(F_{\flat 0}^\sharp\right)_1(s)&= (\bar F_0)^\sharp (|F|^* s)- \theta \sfd_{\caX}(\bar F_0)^\sharp( |F|^*s)
            \end{aligned}
        \end{equation}
        for all $h\in C^\infty(M)$ and $s\in \Gamma(\frg[1]^*)$.
        
        Next, we discuss the dg-map $F_{1}$. For $F_1$ to respect the target maps, $F_1$ must be of the form $F_1=(F_0\circ\sft)\times F_1^\prime$ for some dg-map $F_1^\prime:\caX\times \Theta \times \Theta \times \caN \to \rmT[1]\sfG$ with
        \begin{equation}\label{eq:target_compatibility}
            \rmd \sft \circ F_1^\prime= \psi\circ F_0\circ \sft~.
        \end{equation}
        Compatibility with the source maps then reads as
        \begin{equation}\label{eq:source_compatibility}
            \sfs\circ\left((F_0\circ\sft)\times F_1^\prime\right)=F_0\circ \sfs~,
        \end{equation}
        and $F_1^\prime$ must also respect composition of morphisms, which we will address later.
        
        By our usual argument from \ref{rem:graded-map-to-differential-map}, the map $F_1^\prime$ is given by a lift of a graded map $F_{1\flat}^\prime:\caX\times \Theta \times \Theta \times \caN \to \sfG$ and~\eqref{eq:target_compatibility} translates to $\sft \circ F_{1\flat}^\prime=\sfp_M\circ \psi\circ F_0\circ \sft$, where $\sfp_M$ is again the projection to the underlying manifold $M$ of $\rmT[1]M$. Furthermore, as $|\scP\mathsf{air}(\Theta)|$ is a trivial Lie groupoid and $|\scA^\sfW_{\caH}|=(\sfG\rightrightarrows M)$, the condition $\sfe\circ F_0=F_1 \circ \sfe$ implies
        \begin{equation}\label{eq:brstproof3}
            |F_1^\prime|=|F_1|=\sfe\circ |F_0|=\sfe\circ |F|~.
        \end{equation}
        
        Because the dg-Lie groupoid $\scP\mathsf{air}(\Theta)$ has contractible objects and morphisms, we can characterize the map $F_1$ by simpler maps. Consider the graded maps $\eps_{1,2}: \Theta \to \Theta \times \Theta $ given by $\eps_{1,2}^\sharp(r_0+r_1\theta_1+r_2 \theta_2+ r_{12}\theta_1\theta_2)= r_0+r_{1,2}\theta$, for real numbers $r_s$, and $\eps_0:\Theta\to \Theta$ given by $\eps_0^\sharp(r_0+r_1\theta)=r_0$. These maps fit into the commutative diagram
        \begin{equation}
            \begin{tikzcd}[column sep=2cm]
                & \Theta \times \Theta \arrow[r, "\eps_2\times \eps_1"] \arrow[rr, bend left=25,"\sfid"]
                &( \Theta \times \Theta)\times_{\Theta} (\Theta \times \Theta  )    \arrow[r, "\circ"]
                & \Theta \times \Theta
            \end{tikzcd}
        \end{equation}
        and satisfy the identities 
        \begin{equation}
            \begin{gathered}
                \eps_1=\sfinv\circ\eps_2~,~~\sfs\circ\eps_1=\sfid_{\Theta}~,\\
                \sft\circ\eps_1=\eps_0~,~~\eps_1\circ\eps_0=\sfe\circ\eps_0~.
            \end{gathered}
        \end{equation}

        Let us denote $F^\prime_{1\flat }\circ(\sfid_{\caX}\times\eps_1\times\sfid_{\caN})$ by $F_{1\flat }^{\prime\prime}$.  Then $F_{1\flat }^{\prime\prime}$ must satisfy 
        \begin{equation}\label{eq:functoriality_conditions_expanded}
            \begin{gathered}
                \sft \circ F_{1\flat }^{\prime\prime}=\sfp_M\circ\psi \circ F_0\circ (\sfid_{\caX}\times\eps_0\times\sfid_{\caN})
                \eand
                |F_{1\flat }^{\prime\prime}|= \sfe \circ|F|~.
            \end{gathered}
        \end{equation}
        The first condition is due to~\eqref{eq:target_compatibility} while the second equation is a consequence of~\eqref{eq:brstproof3}. Conversely,
        the map $F_{1\flat }^\prime$ is given in terms of $F_{1\flat }^{\prime\prime}$ by
        \begin{equation}\label{diag:morphism_map_definition}
            \begin{tikzcd}[column sep=1cm]
                &\caX \times \Theta \times \Theta \times \caN \arrow[r, "\sft \times \sfs"] \arrow[d,"F_1^\prime"]
                & (\caX \times \Theta  \times \caN)\times_{ (\caX \times \caN)} (\caX \times \Theta  \times \caN) \arrow[d,"(\sfinv\circ F_{1\flat}^{\prime\prime})\times F_{1\flat}^{\prime\prime}"]
                \\
                &  \sfG
                & \sfG\times_{M} \sfG    \arrow[l, "\circ"]~,
            \end{tikzcd}
        \end{equation}
        which can be shown by using compatibility of $F_{1\flat}^{\prime}$ with composition of morphisms. Hence, we conclude that $F_{1\flat}^{\prime\prime}$ fully fixes $F_{1}^{\prime}$.
        
        The map $F_{1\flat}^{\prime\prime}$ is now defined by $|F^{\prime\prime}_{1\flat}|=|F|$ and a derivation 
        \begin{equation}
            H\colon C^\infty(G)\to \left(\sfe\circ|F|\right)_*\left(C^\infty(\caX)_1\otimes C^\infty(\caN)_0\right)
        \end{equation}
        according to 
        \begin{equation}
            \begin{gathered}
                (F^{\prime\prime}_{1\flat})^\sharp(g)= \left(\sfe\circ|F|\right)^*g-\theta  H(g)~\in~C^\infty(|\caX|\times |\caN|)~\oplus~\theta C^\infty(\caX)_1\otimes C^\infty(\caN)_0
            \end{gathered}
        \end{equation}
        for $g \in   C^\infty(\sfG)$. As argued in~\cite[Prop. 8.21]{Li:2014}, any such derivation is given by the de~Rham differential followed by a morphism of modules 
        \begin{equation}
            H^\prime\colon\Omega^1(\sfG) \to \left(\sfe\circ|F|\right)_*\left( C^\infty(\caX)_1\otimes C^\infty(\caN)_0\right)~,
        \end{equation}
        which we now characterize further. The first equation in~\eqref{eq:functoriality_conditions_expanded} translates to 
        \begin{equation}
            H^\prime\left(\sft^* \varpi\right)=0
        \end{equation}
        for $\varpi\in\Omega^1(M)$, which implies that $H^\prime$ has to be of the form $H^{\prime\prime} \circ \sfi^*$, where 
        \begin{equation}
            H^{\prime\prime}\colon \Gamma(\ker(\rmd \sft))^* \to  \left(\sfe\circ|F|\right)_*\left(C^\infty(\caX)_1\otimes C^\infty(\caN)_0\right)
        \end{equation}
        and $\sfi^*$ is induced by $\sfi\colon\mathsf{Ker}(\rmd \sft)\to \rmT \sfG$. Finally, $H^{\prime\prime}$ is a map of sheaves, and therefore it factors uniquely through pullbacks if the topological part does. $H^{\prime\prime}$ is uniquely determined by 
        \begin{equation}
            H^{\prime\prime\prime}\colon  \sfe^*\Gamma(\ker(\rmd \sft))^*= \Gamma(\frg[1])^* \to  |F|_\star\left(C^\infty(\caX)_1\otimes C^\infty(\caN)_0\right)~.
        \end{equation}
        Therefore, we can identify $H^{\prime\prime\prime}$ with $(\bar F_1)^\sharp$. 
        
        Using the Baker--Campbell--Hausdorff formula for Lie groupoids~\cite{Ramazan:0001149}, see also the computation in~\cite[Sec.~8.4.3.1]{Li:2014}, one finds that $F_{1\flat}^\prime$ is given by 
        \begin{equation}
            \begin{gathered}
                |F^{\prime}_{1\flat}|: |\caX|\times |\caN| \to \sfG
                \ewith 
                |F^{\prime}_{1\flat}|=\sfe\circ |F|~,
                \\
                (F^{\prime}_{1\flat})^\sharp(g)=|F|^*(g)-\theta_2(\bar F_1)^\sharp (\rho^*(\sfe^*\rmd g))+(\theta_2-\theta_1) (\bar F_1)^\sharp(\sfe^*\sfi^* \rmd g)
                \\
                -\theta_2\theta_1\frac{1}{2} [ (\bar F_1)^\sharp(\sfe^*\sfi^* \rmd g), (\bar F_1)^\sharp(\sfe^*\sfi^* \rmd g)]~,
            \end{gathered}
        \end{equation}
        and the lift of $F_{1\flat}^\prime$ to $F_{1}^\prime$ is simply        
        \begin{equation}
            \begin{gathered}
                |F_{1\flat}^\prime|=|F_{1}^\prime|~,~~(F_{1}^\prime)^\sharp(g)= F_{1\flat}^\prime (g)~,\\
                (F_{1}^\prime)^\sharp(\rmd g)= \rmd_{(\caX\times \Theta\times \caN)}F_{1\flat}^\prime (g)= ( \rmd_{\caX}+\rmd_{ \Theta}+\rmd_{ \caN})F_{1\flat}^\prime (g)~.
            \end{gathered}
        \end{equation}
        Condition~\eqref{diag:degree_0_condition_groupoid} implies $( \rmd_{\caX}+\rmd_{ \Theta})F_{1\flat}^\prime (g)=0$, and therefore $(\bar F_1)^\sharp$ must satisfy the first and third equations in~\eqref{eq:brstproof1} with the identifications $|F|=|f|$ and $(\bar F_1)^\sharp=f_1^\sharp$. As a result $(F_{1}^\prime)^\sharp(\rmd g)=\rmd_{\caN}(F_{1}^\prime)^\sharp(\rmd g)$.
        We still need to implement condition~\eqref{eq:source_compatibility}. Precomposing this equation with $\sfid_{\caX}\times\sfe_1\times\sfid_{\caN}$ results in        
        \begin{equation}\label{eq:brstproof5}
            \sfs\circ((F_0\circ(\sfid_{\caX}\times\sfe_0\times\sfid_{\caN}))\times F_1^{\prime}\circ (\sfid_{\caX}\times\sfe_1\times\sfid_{\caN}))=F_0~.
        \end{equation}
        
        Let us again choose local coordinates $(m_1^a,\bar m_1^a,  \xi^\alpha,\bar  \xi^\alpha)$ for $\rmT[1]\frg[1]$ in a neighborhood $U\subset M$. The embedding $\sfe:M\rightarrow \sfG$ then allows us to identify $\sfe(U)\subset \sfG$ with $U\times V$, on which we use local coordinates $(m_2^a, \bar m_2^a, y^\alpha, \bar y^\alpha)$. In the fibered product $\rmT[1]\frg[1]\fibtimes{\psi}{\rmd \sft} \rmT[1]\sfG$ we need to identify $\bar m_1^a+ \ttr^a_\alpha \xi^\alpha$ with $\bar m_2^a$ and $m_1$ with $m_2$.  We will show later that
        \begin{equation}\label{eq:Infinitesimal_action}
            \racton^\sharp(\xi^\alpha)=\xi^\alpha+ \bar y^\alpha+ \omega^\alpha_{a\beta}\bar m_1^a y^\beta -\ttf_{\beta\gamma}^\alpha y^\beta \xi^\gamma+ O_2(y)~,
        \end{equation}
        where $O_2$ represents functions of order at least 2 in its argument.

        Let us assume this for now, then a straightforward calculation shows that the left hand side of~\eqref{eq:brstproof5} becomes
        \begin{equation}
            \begin{split}
                (\bar F_0)^\sharp(\xi^\alpha)+\theta \rmd_{\caN}(\bar F^1)^\sharp(\xi^\alpha)+\theta|F|^*( \omega^\alpha_{a\beta})F_0^\sharp(\bar m_1^a) (\bar F_1)^\sharp(\xi^\alpha) +\theta|F|^*(\ttf_{\beta\gamma}^\alpha)(\bar F_1)^\sharp( \xi^\beta)(\bar F_0)^\sharp( \xi^\gamma)~,
            \end{split}
        \end{equation}
        where we used the identity $\sfe^*\sfi^*\rmd y^\alpha=\xi^\alpha$, a consequence of the definition $\Gamma(\frg)^*=\Gamma(\ker(\rmd \sft))^*$. Comparing this with $F_0$, we find the second equation of~\eqref{eq:brstproof1} with the identifications $|F|=|f|$, $(\bar F_1)^\sharp=f_1^\sharp$ and $(\bar F_0)^\sharp=f_0^\sharp$. Hence, $\frbrst^{\rm adj}(\caN,\frg[1]) \cong  \sfLie(\catdgGrpd(\caN\rightrightarrows \caN,\scA^\sfW))$, if~\eqref{eq:Infinitesimal_action} holds. 
        
        It remains to prove~\eqref{eq:Infinitesimal_action}. 
        Recall that the Lie algebroid connection corresponding to the Lie groupoid connection is given by~\cite[Thm.\ 3.11]{Crainic:1210.2277} (see also \Cref{prop:associated_Lie_algebroid_connection})
        \begin{equation}
            \begin{split}
                (\nabla_ V \gamma)(x)
                &=
                [\mathsf{hor}_\caH(V), \gamma_l](\sfe_x)
                \\
                &=\left.-\frac{\rmd}{\rmd \epsilon} \rmd R_{\mathsf{exp}(-\epsilon \gamma )} \left(\sigma_\caH\bigl(\mathsf{exp}(\epsilon \gamma_x), V_x\bigr)\right)\right|_{\epsilon=0}
            \end{split}
        \end{equation} 
        for all $\gamma\in \Gamma(\frg)$ and $ V\in \frX(M)$, 
        where $ \gamma_l$ is the corresponding left invariant vector field, and where we expressed the Lie bracket of vector fields as a Lie derivative along $\gamma_l$ whose flow is given as right-translation with the exponential of Lie groupoids, that is, the flow of $\gamma_l$ through $q \in \sfG$ is given by $R_{\mathsf{exp}(\epsilon \gamma)}(q) \coloneqq q ~ \mathsf{exp}\left(\epsilon \gamma_{\sfs(q)}\right)$.\footnote{With some caveats, the exponential works similarly to that for Lie groups for our purpose, see~\cite[Prop.\ 3.6.1]{0521499283} for more details; recall that we define $\frg$ via the vertical structure of $\sft$, in particular, $x \mapsto \mathsf{exp}(-\epsilon \gamma_x)$ is a (local) section of $\sft$.}
        
        Firstly, we show that 
        \begin{equation}
            \begin{aligned}
                \sigma_\caH \colon &&\sft^* \rmT[1]M &\to \rmT[1]\sfG~,
                \\
                &&(g,\bar m)&\mapsto(m,\bar m,y,- \omega^\alpha_{a \beta}\bar m^a y^\beta \bar y_\alpha+O_2(y))~,
            \end{aligned}
        \end{equation}
        where $g=(m^a,y^\beta)$ and $\bar m=\bar{m}^a \bar{m}_a \in \rmT_{\sft(g)}M$ and $\bar{m}_a$ is a set of generators, which, in particular, dualizes to $(\bar m_2^a, \bar y^\alpha) \mapsto (\bar m^a- \omega^\alpha_{b \beta}\bar m^a y^\beta +O(y^2))$. To see this, suppose 
        \begin{equation}
            \sigma_\caH(m,\bar m,y)= (m, \bar m,y, U(m, y)_a^\alpha\bar m^a \partial_\alpha)~.
        \end{equation}
        Let us define $ \bar U(m, y)_{a\beta}^\alpha \coloneqq \left.\frac{\rmd}{\rmd y^\beta} U(m, y)_a^\alpha\right|_{y=0}$, then, by Hadamard's lemma, we have 
        \begin{equation}
            \sigma_\caH(m, \bar m,y)= (m ,\bar m, y,\bar U(m, y)_{a\beta}^\alpha\bar m^a y^\beta \partial_\alpha+ O_2(y))~.
        \end{equation}
        Using the Baker--Campbell--Hausdorff formula, one can see that
        \begin{equation}
            \sigma_\caH(m,\bar m,y) \circ \left(\sfs(m, y), \ttr^i_\alpha  U(m, y)_a^\alpha\bar m^a \bar m_i,-y, 0\right)
            = 
            \left(m,\bar m,0, \bar U(m, y)_{a\beta}^\alpha\bar m^a y^\beta \partial_\alpha+ O_2(y)\right)~.
        \end{equation}
        Note that the left hand side is nothing but $ \rmd R_{\mathsf{exp}(-\epsilon \gamma )} \sigma_\caH\bigl(\mathsf{exp}(\epsilon \gamma_x), V_x\bigr)$ with $\epsilon\gamma=y^\alpha \bar y _\alpha$ and $V=\bar m$, which makes it clear that $\bar U(m, y)_{a\beta}^\alpha= -\omega _{a\beta}^\alpha$. 
        
        Recall that in \ref{eq:adjoint_simplified} we showed that $\rmAd(g^{-1},\gamma)=(g^{-1},0)\circ(\sfe_{\sft(g)}, \gamma)\circ \sigma_\caH(g, \rho(\gamma))$. Using the Baker--Campbell--Hausdorff formula for Lie groupoids, one sees the dual of the map $ \sfp_{\frg[1]}\circ \rmAd\circ\sfinv^\star :\sft^*\frg[1]\to \frg[1]$ takes the form 
        \begin{equation}\label{eq:Infinitesimal_action_1}
            \xi^\alpha\mapsto \xi^\alpha+\ttf^\alpha_{\gamma\beta}\xi^\gamma y^\beta- \omega^\alpha_{a \beta}\ttr^a_\gamma \xi^\gamma y^\beta +O_2(y)~.
        \end{equation}
        By the identity $ \sfp^{\rm vert} + \sigma_\caH \circ \rmd \sft = \sfid_{\rmT G}$, and the definition $\vartheta^{\rm tot}(g,\gamma)=\left(\scL^{-1}\circ\sfp^{\rm vert}\right)(g,\gamma)$, in the local coordinates, we have
        \begin{align}
            \vartheta^{\rm tot}(m,\bar m, y, \bar y)
            &= 
            (m,0, y, 0)^{-1}\circ \bigl((m,\bar m, y, \bar y)- \sigma_\caH (m,y; \rmd \sft (m,\bar m, y, \bar y))\bigr)
            \\
            &=
            \left(m, 0, -y + O_2(y), 0\right)\circ \bigl((m,\bar m, y, \bar y)-\sigma_\caH(m, \bar m, y,0)\bigr)~.
        \end{align}
        With this, it is easy to see that the above map dualizes to
        \begin{equation}\label{eq:Infinitesimal_action_2}
            \xi^\alpha\mapsto  \bar y^\alpha + \omega^\alpha_{a \beta}\bar m_2^a  y^\beta +O_2(y)~.
        \end{equation}
        Summing~\eqref{eq:Infinitesimal_action_1} and~\eqref{eq:Infinitesimal_action_2}, together with the identification $\bar m_1^a+ \ttr^a_\alpha \xi^\alpha$ with $\bar m_2^a$, results in~\eqref{eq:Infinitesimal_action}. This completes the proof.
    \end{proof}
    
    \subsection{General statements on existence of adjustments and connections}
    
    Before coming to concrete examples, let us collect a number of statements that are readily derived from the results of previous literature. 
    
    We start with a statement that is evident because ordinary principal bundles are special cases of principal groupoid bundles.
    \begin{proposition}
        Not every principal groupoid bundle admits a flat connection.
    \end{proposition}
    
    In the case of higher gauge theories, the existence of adjustments is not guaranteed, and there are higher Lie groups that do not admit an adjustment, see e.g.~the case of the 2-group $\sfT\sfB_n$ discussed in~\cite{Kim:2023hqx}. Below are the analogous statements for Lie algebroids and Lie groupoids.
    \begin{proposition}
        Not every Lie algebroid admits an adjustment.
    \end{proposition}
    \begin{proof}
        Consider the class of totally intransitive Lie algebroids, i.e.~Lie algebroids where the anchor map vanishes, which are also called bundles of Lie algebras (BLA). Only a proper subclass of these are what one calls Lie algebra bundles (LAB) as defined in~\cite[Def.~3.3.8~ff.]{0521499283}, that is, in general BLA $\neq$ LAB. However, a totally intransitive Lie algebroid over a connected base manifold admits a Cartan connection if and only if it is a Lie algebra bundle~\cite[Thm.\ 6.4.5]{0521499283}, see also~\cite[Prop.\ 2.13]{Abad:0901.0319}.
    \end{proof}
    
    This directly implies that not every Lie groupoid admits an adjustment. Note that a brief discussion of literature about the existence of Cartan connections on non-transitive examples can be found in the introduction of~\cite{Fernandes:2204.08507}.
    
    \begin{proposition}
        Not every Lie algebroid admitting a Cartan connection also admits a covariant adjustment.
    \end{proposition}
    \begin{proof}
        As an example, consider again the case of a Lie algebra bundle $\frg$. The existence of a primitive $\zeta$ requires that the Cartan connection is a Lie derivation law covering a Lie algebroid morphism $\Xi\colon \rmT M \to \operatorname{Out}\left(\scD_{\operatorname{Der}}(\frg)\right)$, see the discussion in~\cite[\S 7.2]{0521499283} after Def.~7.2.2. Any vector bundle viewed as an abelian Lie algebra bundle with a curved connection is an example for $\frg$ with a Cartan connection which is not a covariant adjustment.
        
        Alternatively, a Cartan connection is just a certain type of closed connection (\cite{Crainic:1210.2277, Abad:0911.2859}, see also~\cite{Fernandes:2204.08507}). The curvature of a Cartan connection is in fact also closed, see\footnote{This result was also independently found by A.\ Kotov (unpublished).}~\cite[Thm.\ 4.8.1]{Fischer:2021glc} as well as~\cite[Prop.\ 2.22]{Fernandes:2204.08507} (for multiplicative Ehresmann connections). As pointed out in~\cite{Fischer:2020lri, Fischer:2021glc, Fischer:2022sus},\footnote{A.\ Kotov and T.\ Strobl found an alternative proof and motivation of this statement (unpublished), independently from and after S.-R.\ Fischer's proof.} an adjustment is such a connection whose curvature is actually exact, i.e.~admitting a primitive. Any Cartan connection whose curvature is not exact hence does not admit an adjustment.
    \end{proof}
    
    \begin{proposition}
        An adjustment of a Lie groupoid is generally not unique. 
    \end{proposition}
    \begin{proof}
        This is easily seen from the case of Lie algebra bundles, cf.~\cite[Cor.\ 5.1.39 and Thm.\ 5.1.38]{Fischer:2021glc}, for which the Cartan connection comes with a family of different primitives.
    \end{proof}
    \noindent We note that a precise characterization of uniqueness involves the definition of isomorphisms of groupoids with adjustment. This can be done by spans of adjusted groupoids, cf.~\cite{Rist:2022hci} for the analogous case of Lie 2-groups. We plan to address this in~\cite{future:2024ac}. 
    
    We can say a few more things about adjustments originating from flat Cartan connections.
    \begin{proposition}[{\cite[Cor.~3.12]{Crainic:1210.2277}}]\label{prop:localIsomToActionAlg}
        If $\scG$ is a $\sft$-simply connected Lie groupoid over a compact and simply connected manifold $M$, then $\scG$ admits a flat Cartan connection $\caH$ if and only if $\scG$ is isomorphic to an action Lie groupoid associated to a Lie group $\sfH$ acting on $M$.
    \end{proposition}
    
    Infinitesimally, we have the following, analogous statement:
    \begin{proposition}[{\cite[Thm.~A]{Blaom:0404313},~\cite[Prop.\ 2.12]{Abad:0901.0319}}]\label{prop:FlatCartanAndActionAlg}
        Let $\frg$ be a Lie algebroid over a simply connected manifold $M$. Then $\frg$ admits a flat Cartan connection if and only if $\frg$ is isomorphic to an action algebroid associated to a Lie algebra action on $M$.
    \end{proposition}
    \noindent For more details on this point, see also the explicit discussion in~\cite[near~eqn.~(9)]{Kotov:2015nuz}; see also~\cite{Blaom:1304.7838} for further discussions.
    
    \begin{remark}
        The flat Cartan connection (in both propositions) can be identified with the canonical flat connection under this isomorphism.
    \end{remark}
    
    There are a couple of further remaining evident questions, such as e.g.~the general existence of an adjusted connection on a principal $\scG$-bundle for $\scG$ a groupoid with connection, but we prefer to answer them in future work on the total space description of principal groupoid bundles like~\cite{future:2024ac}. Note that in the case of $\frg$ being a Lie algebra bundle, a total space description of adjusted connections is already provided in~\cite{Fischer:2022sus}. A discussion of Ehresmann connections on principal bundles equipped with Lie group bundle actions also appears in~\cite{Castrillon:2201.07088}, however, without discussing adjustments, merely assuming a Cartan connection, and without discussing gauge invariance; in particular, the condition on $R_\nabla$ as in \Cref{def:local_adjustment} is not observed.
    
    \section{Examples}\label{sec:examples}
    
    In the following, we will go through a useful list of classes of Lie groupoids and comment in each case on the properties of the adjustments and the arising connections. The cases that appear in physical field theories are reproduced as special cases of action groupoids, see \ref{ssec:action_groupoids}.
    
    We immediately note that Čech groupoids $\check \scC(\sigma)$ for surjective submersions $\sigma:Y\rightarrow M$, cf.~\ref{ex:Cech_groupoid}, are uninteresting: if $Y=\sqcup_a U_a$ for $\sqcup_a U_a$ an atlas of a manifold, the Lie algebroid of $\check \scC(\sigma)$ is simply $Y$ itself, because the kernel of the map $\rmd \sigma:\rmT Y\rightarrow \rmT M$ is trivial in the fibers. Moreover, recall that $\check \scC(\sigma)$ is Morita-equivalent to the groupoid $(M\rightrightarrows M)$, which produces a trivial special case of an action groupoid, see~\ref{ssec:action_groupoids}.\footnote{We note, however, that Morita-equivalence of the gauge Lie groupoids does not have to imply physical equivalence, see e.g.~\cite[Sec.~4.4]{Borsten:2024gox} for a counterexample in the context of higher gauge theory.}
    
    Similarly, we do not consider the Lie groupoids underlying categorified or Lie 2-groups separately, as any 2-group can be brought into the form $\sfG\times \sfH\rightrightarrows \sfG$, and the underlying Lie groupoid is an action groupoid, describing an action of $\sfH$ onto $\sfG$.
    
    \begin{remark}[Disclaimer]\label{rem:Disclaimer}
        As proven in~\cite{Fischer:2021glc}, the following types of examples always admit a flat Cartan connection locally, i.e.~over a patch of the base space. Moreover, in the context of curved gauge theory, there is an equivalence relation, preserving dynamics and kinematics, so that every such example is locally equivalent to a gauge theory with flat connection on the structure. Recall \Cref{prop:localIsomToActionAlg}: if the primitive is also zero, then the associated gauge theory is locally equivalent to an ordinary gauge theory. This equivalence relation is called \uline{field redefinition}.
        
        However, as also argued in~\cite{Fischer:2020lri, Fischer:2021glc}, the primitive may not be equivalent to a vanishing one, even locally. 
        
        Field redefinitions are outside of the scope of this paper, and we plan to include a full discussion in~\cite{future:2024ac}; nevertheless, we will comment after each example on the possibility of field redefinitions.
    \end{remark}

    \subsection{Fundamental groupoid}\label{sec:PairGroupoids}
    
    Recall that the pair groupoid $\scP\mathsf{air}(M)=(M\times M\rightrightarrows M)$ of \ref{ex:pair_groupoid} differentiates to the tangent algebroid $\rmT M$. As for Lie groups, however, a Lie groupoid integrating an algebroid is, if it exists, not unique. The unique Lie groupoid with connected and simply connected $\sft$-fibers integrating $\rmT M$ is the \emph{fundamental} or \emph{homotopy groupoid} $\Pi(M)$. This groupoid consists of homotopy classes $[\gamma]$ of curves $\gamma \colon [0, 1] \to M$ with 
    \begin{equation}
        \begin{gathered}
            \sft ([\gamma])= \gamma(1)~,~~~\sfs([\gamma])=\gamma(0)~,~~~[\gamma] \circ [\gamma']=[\gamma * \gamma']~,
            \\
            \sfe_m=[m]~,~~~[\gamma]^{-1}=\left[ \gamma^- \right]
        \end{gathered}
    \end{equation}
    for all curves $\gamma, \gamma' \colon [0, 1] \to M$ such that $\gamma(0) = \gamma'(1)$, where $[m]$ is the homotopy class of the constant curve at $m$, $\gamma * \gamma'$ is the concatenation\footnote{The concatenation may not be smooth, however this technical issue can be avoided with the concept of sitting instants which will not be important here; see e.g.~\cite{Mackaay:2000ac,Schreiber:0705.0452} for more details.} of $\gamma$ and $\gamma'$ (first following $\gamma'$, then $\gamma$), and $\gamma^-$ denotes $\gamma$ traversed backwards. The fundamental groupoid $\Pi(M)$ is a transitive groupoid. In fact, it is the Atiyah groupoid of the universal cover of $M$ with isotropy group bundle given by $\pi_1(M)$, the (first) fundamental group bundle of $M$. Thus, we have the following short exact sequence of Lie groupoids
    \begin{center}
        \begin{tikzcd}
            \pi_1(M) \arrow[hook]{r}
            &
            \Pi(M) \arrow[two heads]{r}
            &
            \scP\mathsf{air}(M)~.
        \end{tikzcd}
    \end{center}
    
    \begin{proposition}[Cartan connections, {\cite[consequence of Def.\ 2.11 \& Prop.\ 2.15]{Abad:0911.2859}}]\label{prop:CartanOnPair}
        Cartan connections on $\Pi(M)$ are in one-to-one correspondence to $\Pi(M)$-representations on $\rmT M \to M$.
    \end{proposition}
    
    \begin{remark}\label{rem:CartanConnOnTM}
        This recovers the well-known infinitesimal situation, see for example~\cite{Blaom:0404313, Abad:0901.0319, Fischer:2021glc}. Infinitesimally, a Cartan connection $\nabla$ on $\rmT M$ is in one-to-one correspondence to flat vector bundle connections $\bar\nabla$ on $\rmT M$ via
        \begin{equation}
            \bar{\nabla}_V W
            =
            \nabla_W V
            + [V, W]
        \end{equation}
        for all $V, W \in \mathfrak{X}(M)$. Observe that in this case we identify $\bar \nabla$ with $\nabla^{\mathrm{bas}}$, and in fact, by construction, the $\Pi(M)$-representation in \Cref{prop:CartanOnPair} will be the adjoint action of $\Pi(M)$ on $\rmT M$ given as the parallel transport of $\bar \nabla$.
    \end{remark}
    
    \begin{proof}[Proof of \Cref{prop:CartanOnPair}]
        Due to the fact that $\sft$-fibers of $\Pi(M)$ are simply connected, there is a one-to-one correspondence of Cartan connections on $\rmT M$ and $\Pi(M)$, so that the existence of a Cartan connection is equivalent to the existence of a flat connection $\bar \nabla$ on $\rmT M$ (recall \Cref{rem:CartanConnOnTM}). Flat connections are in one-to-one correspondence to $\Pi(M)$-representations on $\rmT M \to M$. This finishes the proof.
    \end{proof}
    
    Cartan connections on $\Pi(M)$ always allow for a primitive, which is simply given by the torsion tensor of the connection; compare also with~\cite[Cor.~3.6.6 and Thm.~4.8.4]{Fischer:2021glc}.
    
    \begin{proposition}[Covariant adjustments always exist, {\cite[special case of first Prop.\ in \S 4.6]{Blaom:0404313}}]\label{prop:PrimitivesForFundGroupi}
        Every Cartan connection on $\Pi(M)$ has a primitive of the form 
        \begin{equation}
            \zeta
            =
            t_\nabla
            + H~,
        \end{equation}
        where $t_\nabla$ is the torsion of $\nabla$, and $H \in \Omega^2(M; \rmT M)$ with $\nabla^{\mathrm{bas}}H = 0$, that is,
        \begin{equation*}
            R_\nabla
            =
            - \nabla^{\mathrm{bas}} \zeta~.
        \end{equation*}
    \end{proposition}
    
    Thence, a covariant adjustment exists if and only if there is a $\Pi(M)$-representation on $\rmT M \to M$, if and only if there is a flat connection on $\rmT M$. If the Cartan connection is also flat, then we achieve a refinement of \Cref{prop:FlatCartanAndActionAlg}:
    
    \begin{proposition}[Lie group structure on $M$, {\cite[\S 3.1 and references therein]{Blaom:0404313}} and {\cite[Comment after Proposition 2.12]{Abad:0901.0319}}]\label{prop:FlatCartanAndLieGroups}
        Let $M$ be a smooth, compact, and simply connected manifold, and assume we have a flat Cartan connection on $\mathrm{T}M$. Then $M$ is diffeomorphic to a Lie group.
    \end{proposition}
    
    Henceforth, we have the canonical examples:
    
    \begin{example}[Pair groupoids with curved Cartan connections]\label{ex:PairGroupoidsEx}
        Every parallelizable, compact, and simply connected smooth manifold $M$ which does not admit a Lie group structure can be equipped with an adjustment for which the Cartan connection is not flat. In particular, there is the example $M=S^7$, the seven-dimensional sphere regarded as the unit octonions, as provided in~\cite[\S 5.2.3]{Fischer:2021glc}.
    \end{example}
    
    \begin{remark}[Field redefinitions]
        \Cref{ex:PairGroupoidsEx} is stable under the field redefinitions \Cref{rem:Disclaimer}. That is, there is no field redefinition towards a structure with flat Cartan connection, otherwise we would immediately have a contradiction via \Cref{prop:FlatCartanAndLieGroups}.
        
        Locally, there is always a field redefinition towards a flat adjustment. 
    \end{remark}
    
    \begin{remark}[Strictness]\label{rem:StrictnessPiM}
        If a covariant adjustment exists and one chooses $\zeta = t_\nabla$, then the adjustment is strict. This is also stable under the field redefinitions: by \Cref{prop:PrimitivesForFundGroupi}, every primitive differs from $t_\nabla$ by tensors constant w.r.t.\ $\nabla^{\rm bas} = \bar \nabla$, where $\bar \nabla$ is the flat connection of \Cref{rem:CartanConnOnTM}.
        
        A straightforward calculation making use of the first Bianchi identity of $\bar \nabla= \nabla^{\rm bas}$ and the fact that $\nabla^{\rm bas}$ is flat shows that $\zeta = t_\nabla = - t_{\nabla^{\rm bas}}$ is a strict adjustment. Observe that strictness requires that (\Cref{rem:AlternCompEq})
        \begin{equation}
            \rmd^{\nabla^\zeta} \zeta = 0~,
        \end{equation}
        where $\nabla^\zeta$ is a connection on $\rmT M$ defined by $\nabla^\zeta_V W \coloneqq \nabla_V W - \zeta(V, W)$ for all $V, W \in \frX(M)$. Thus, choosing $\zeta = t_\nabla = -t_{\nabla^{\rm bas}}$ leads to the following: by assumption we have
        \begin{equation}
            \nabla^{\zeta}_{V} W
            =
            \nabla_{V} W
            - t_{\nabla}(V, W)
            =
            \nabla_{W} V
            + [V, W]_\frg
            =
            \nabla^{\mathrm{bas}}_V W~.
        \end{equation}
        Using this and $\zeta = - t_{\nabla^{\rm bas}}$ one has
        \begin{equation}
            \mathrm{d}^{\nabla^{\zeta}} \zeta
            =
            - \mathrm{d}^{\nabla^{\mathrm{bas}}} t_{\nabla^{\mathrm{bas}}}~.
        \end{equation}
        Recall the first Bianchi identity of $\nabla^{\rm bas}$, that is,
        \begin{equation}
            \left(\mathrm{d}^{\nabla^{\mathrm{bas}}} t_{\nabla^{\mathrm{bas}}}\right)(V_1, V_2, V_3)
            =
            R_{\nabla^{\mathrm{bas}}}(V_1, V_2) V_3 + R_{\nabla^{\mathrm{bas}}}(V_2, V_3) V_1 + R_{\nabla^{\mathrm{bas}}}(V_3, V_1) V_2 
        \end{equation}
        for all $V_1, V_2, V_3 \in \frX(M)$. We know by assumption that $\nabla^{\rm bas}$ is flat, thus, combining everything
        \begin{equation}
            \mathrm{d}^{\nabla^{\zeta}} \zeta
            =
            0~.
        \end{equation}
    \end{remark}
    
    \subsection{Lie group bundles}\label{ssec:Lie_group_bundles}
    
    Another case which is essentially understood is the case of Lie group bundles~\cite{Fischer:2022sus,Fischer:2401.05966}. We therefore keep our discussion brief, and restrict ourselves to covariant adjustments.
    
    Recall that Lie group and Lie algebra bundles are fibrations with the typical fibers being Lie groups and Lie algebras so that the local trivializations are Lie group and Lie algebra isomorphisms when projected to the typical fiber. Furthermore, a Lie group bundle $\pi_\scH\colon\scH\rightarrow M$ can be regarded as a Lie groupoid $\scH$ with source and target maps equal $\sfs=\sft=\pi_\scH$ and the units given by the unit element over the base points.
    
    Consider now a covariant adjustment $\caH$ on a Lie group bundle $\pi_\scH\colon\scH\rightarrow M$ with typical fiber the Lie group $\sfH$, and assume for simplicity that $M$ is connected. In this case an adjusted connection is also called \emph{(multiplicative) Yang--Mills connection},~\cite{Fischer:2022sus, Fischer:2401.05966}. The equations~\eqref{eq:local_adjustment_conditions} defining an adjustment now simplify. For the connection 1-form $\vartheta^{\rm tot}$ on $\pi_\scH \colon \scH \to M$ we have (in simplified notation):
    \begin{equation}
        \begin{aligned}
            \vartheta^{\rm tot}_{gq}\left( V \cdot W  \right)
            &=
            \mathrm{Ad}_{q^{-1}}\Bigl( {\vartheta_{g}^{\rm tot}}(V) \Bigr)
            + \vartheta^{\rm tot}_{q}(W)~,
            \\
            \left.\left(\mathrm{d}^{\pi_\scH^*\nabla} \vartheta^{\rm tot}
            + \frac{1}{2} \left[ \vartheta^{\rm tot} \stackrel{\wedge}{,} \vartheta^{\rm tot}\right]_{\pi_\scH^*\frh} \right)\right|_g
            &=
            \mathrm{Ad}_{g^{-1}} \circ \left.\pi_\scH^!\zeta\right|_g
            - \left.\pi_\scH^!\zeta\right|_g
        \end{aligned}
    \end{equation}
    for all $g, q \in \scH_x$ ($x \in M$), $V \in \rmT_g \scH$, and $W \in \rmT_q \scH$. Here, $(V, W) \mapsto V \cdot W$ denotes the canonical group structure on $\rmT \scH$, $\mathrm{Ad}$ is now just the usual fiber-wise defined adjoint representation of $\scH$ on its Lie algebra bundle $\frh$, and $\zeta$ some element of $\Omega^2(M;\frh)$, the \emph{primitive} of $\vartheta^{\rm tot}$; $\pi_\scH^!\zeta$ denotes its pullback as form. The first equation implies that $\vartheta^{\rm tot}$ is multiplicative, in particular its parallel transport is an isomorphism of the group structure. The second equation implies that the curvature of $\vartheta^{\rm tot}$ is essentially encoded by $\zeta$. Furthermore, $\nabla$ is the induced connection on the corresponding Lie algebroid $\frh$ (a Lie algebra bundle), which now satisfies
    \begin{equation}
        \begin{aligned}
            \nabla \left( \left[ \cdot, \cdot \right]_{\frh} \right)
            &=
            0~,
            \\
            R_{\nabla}
            &=
            \mathrm{ad} \circ \zeta~.
        \end{aligned}
    \end{equation}
    
    There is a natural class of adjusted connections with non-trivial covariant adjustment. In order to understand this, we follow and use results of~\cite[in particular the appendix]{Fischer:2401.05966}, but we simplify along the way because we will mainly focus on the situation when the Lie group is simple: Assume an ordinary principal $\sfH$-bundle $P$, where $\sfH$ is a Lie group. Also assume that we have a manifold $\sfT$ on which $\sfH$ acts (from the left).\footnote{For \Cref{ssec:Lie_group_bundles}, $\sfT$ and its related constructions and assumptions are only needed to provide an extra geometric interpretation as in \Cref{tab:ComparingTheGeometryGaugeTheoryWithSingularFoliations}; from \Cref{ex:Hopf} to \Cref{ex:OrdinaryLGBSetting} one does not need this extra data.} Then we have two associated bundles:
    \begin{equation}
        \begin{aligned}
            c_\sfH(P) &\coloneqq (P \times \sfH) \Big/ \sfH~,
            \\
            \scT &\coloneqq (P \times \sfT) \Big/ \sfH~,
        \end{aligned}
    \end{equation}
    where $\sfH$ acts on itself by conjugation so that the former bundle is the well-known inner group bundle, the kernel of the anchor on the Atiyah groupoid, whose Lie algebra bundle is given by the adjoint bundle $\mathrm{ad}(P) \coloneqq (P \times \frh) / \sfH$. We have a canonical left-action, $c_\sfH(P) \curvearrowright \scT$. An Ehresmann connection on $P$ (in the usual sense) induces associated connections on $c_\sfH(P)$ and $\scT$. The one on $c_\sfH(P)$ is a covariant adjustment on $c_\sfH(P)$ whose primitive is given by the corresponding curvature on $P$, and the one on $\scT$ is a connection that has been called \emph{compatible Yang--Mills connection}~\cite{Fischer:2401.05966}. That is, this is an Ehresmann connection on $\scT$ w.r.t.~the $c_\sfH(P)$-action and the adjustment on $c_\sfH(P)$, which satisfies a certain curvature equation.
    
    The following remark provides further motivation for considering Lie group bundles like $c_\sfH(P)$ beyond the existence of natural covariant adjustments.
    
    \begin{remark}\label{rem:JustifyingExample}
        As argued in the appendix of \cite{Fischer:2401.05966} every Lie group bundle $\scH$ with a covariant adjustment and whose structural Lie group $\sfH$ has a trivial center is associated to some principal bundle $P$ (whose structural Lie group $\sfK$ is the extension of a Galois group by the identity component of $\sfH$); see also \cite{meinrenken2021integration} for similar arguments. That is, $\scH \cong (P \times \sfH) \Big/ K$, and the covariant adjustment is given as an associated connection. Here we consider for simplicity the situation with $K = H$.
    \end{remark}
    
    In particular \cite{Fischer:2401.05966} provides the following (a modification of \cite{blazquez2022group}):    
    \begin{lemma}[Simple Lie group bundles]\label{lem:SemiSimpleLGBAdjust}
        If $\sfH$ is simple and centerless, then all covariant adjustments on $\scH$ are associated to an Ehresmann connection of a unique principal $\sfH$-bundle $P$, and their primitive is the corresponding curvature on $P$.
    \end{lemma}
    
    \begin{remark}[Existence, strictness, and field redefinitions]\label{rem:LGBStrictness}
        As mentioned earlier and as argued in~\cite[\S 5.1]{Fischer:2021glc} and~\cite{Fischer:2020lri}, there is a strong relation to the discussion in~\cite[\S 7.2 ff.]{0521499283}: a covariant adjustment on a Lie algebra bundle is a \uline{Lie derivation law covering a coupling} in~\cite{0521499283}. This relation implies statements about the existence of covariant adjustments and interpretations of field redefinitions (\ref{rem:Disclaimer}); in particular such covariant adjustments are always strict because \uline{Mackenzie's obstruction class} is by construction trivial for our chosen example, and this is stable under the field redefinitions. The interested reader can find more details in the mentioned references.
    \end{remark}
        
    Simplicity also implies that the correspondence in \Cref{lem:SemiSimpleLGBAdjust} is one-to-one. Slightly generalized, assuming that $\sfH$ has a trivial center implies that the associated connection on $c_\sfH(P)$ is in one-to-one correspondence with the Ehresmann connection on $P$. Thus, let us now assume that $\sfH$ and hence $c_\sfH(P)$ have a trivial center; furthermore assume that $\sfT = \mathbb{R}^d$, and that $\sfH$ acts faithfully on $\sfT$ with 0 as fixed point. Then there is a canonical singular foliation $\scF$ on $\scT$ with $M$ as a leaf generated by vector fields of the form
    \begin{equation}\label{anchor}
        \chi_\scT(V) + \overline{\nu}
    \end{equation}
    for all $V \in \mathfrak{X}(M)$ and $\nu \in \Gamma(\mathrm{ad}(P))$, where $\chi_\scT$ is the projectable horizontal lift of the associated connection on $\scT$, and $\nu \mapsto \overline{\nu}$ denotes the inherited Lie algebra bundle action. The connection on $\scT$ is then also what one calls an $\scF$-connection.
    
    Because the action is faithful, there is the following one-to-one correspondence between:
    \begin{itemize}
        \item the covariant adjustment on $c_\sfH(P)$~;
        \item the Ehresmann connection on $P$~;
        \item the $\scF$-connection on $\scT$~.
    \end{itemize}
    The major result of~\cite{Fischer:2401.05966} implies that (in a formal setting, and in a neighborhood around a leaf) every foliation can be modeled like this; however $P$ is then in general infinite-dimensional.\footnote{In our context we simply assume a finite-dimensional $P$ so that the arguments of~\cite{Fischer:2401.05966} carry over to a smooth setting.} Altogether, we can study connections with non-trivial adjustment on such $c_\sfH(P)$ as structure groupoid by looking at $\scF$ and $P$. In particular, $c_\sfH(P)$ admits a flat adjustment if and only if $\scF$ admits a flat $\scF$-connection, which is the case if and only if $P$ is flat. Furthermore,~\cite{Fischer:2401.05966} provides a clear geometric interpretation of the field redefinitions, extending the discussion of~\cite{Fischer:2020lri} and~\cite[\S 5.1]{Fischer:2021glc} (recall \Cref{rem:Disclaimer}).
    
    \begin{table*}[htbp]
        \centering
        \begin{tabular}{l|l|l}
            Lie group bundle $c_\sfH(P)$ & Singular foliations $\scF$ & principal $\sfH$-bundle $P$\\ \hline
            covariant adjustment $\vartheta^{\rm tot}$ & $\scF$-connection & Ehresmann connection \\
            $c_\sfH(P)$ admits flat $\vartheta^{\rm tot}$ & $\scF$ is flat & $P$ is flat \\
            field redefinition of $\vartheta^{\rm tot}$ & change of $\scF$-connection & change of Ehresmann connection
        \end{tabular}
        \caption{Comparing the geometry of non-trivially covariantly adjusted connections with the geometry of singular foliations. Here, $\scF$-connections only serve to highlight a geometric interpretation, but we actually provided a notion only defined as associated to $P$; however, the general notion of $\scF$-connections allows for more than such connections, in general one has to consider $\sfK$ as in \Cref{rem:JustifyingExample}; this will not matter here, see~\cite{Fischer:2401.05966} for the general classification of $\scF$-connections and their relation to adjustments.}
        \label{tab:ComparingTheGeometryGaugeTheoryWithSingularFoliations}
    \end{table*}
    
    Thus, in order to find non-trivial covariant adjustments which cannot be flattened by field redefinitions one only needs to apply the previous construction on principal bundles $P$ which do not admit flat connections; alternatively search for foliations $\scF$ which do not admit flat $\scF$-connections. There is plenty of literature on both cases.
    
    \begin{example}\label{ex:Hopf}
        Using a non-trivial principal bundle $P$ over a simply connected base manifold, such as the Hopf fibration $S^7 \to S^4$, leads to curved covariant adjustments on $c_\sfH(P)$ in the above construction.
    \end{example}
    
    \begin{remark}[Field redefinitions] 
        By construction, the examples constructed as above are stable under field redefinitions.
        
        Locally, however, (and this holds for all Lie group bundles $\scH$):
        \begin{itemize}
            \item There is always a field redefinition towards a flat adjustment.
            \item However, there are examples which are locally not equivalent to an adjustment with vanishing primitive $\zeta$.
        \end{itemize}
    \end{remark}
    
    \begin{example}[{\cite[Thm.\ 5.17]{Fischer:2020lri}, \cite[Thm.\ 5.1.34]{Fischer:2021glc}}]\label{rem:strictness}
        Consider a Lie group bundle $\scH$ with a covariant adjustment which is not strict. Then there is no field redefinition so that the transformed adjustment is flat. However, as pointed out by Alexei Kotov (private communication), the metric compatibility conditions arising in Lagrangian gauge theories probably imply that a covariant adjustment compatible with the metric, if one exists, is always strict.
    \end{example}
    As a final example, we can answer the conjecture in~\cite[Conjecture 7.8]{Fischer:2022sus} in the simple situation:
    
    \begin{example}[The ordinary setting]\label{ex:OrdinaryLGBSetting}
        By ``ordinary setting,'' we mean a trivial Lie group bundle equipped with its canonical flat connection as adjustment, in particular we also require $\zeta = 0$ as the primitive; while we just speak of an ``adjusted setting'' when speaking of a Lie group bundle with a (non-trivial) adjustment.
        
        Assume that the structural fiber $\sfH$ of a Lie group bundle is simple and centerless. By \Cref{rem:JustifyingExample} and \Cref{lem:SemiSimpleLGBAdjust}, there is a unique principal $\sfH$-bundle $P$ so that the Lie group bundle is of the form $c_\sfH(P)$, and its adjustments are the associated connections. In particular, $c_\sfH(P)$ is trivial if and only if $P$ is trivial. In this case, $P$ is flat, and its canonical flat connection induces a canonical flat connection on $c_\sfH(P)$ so that a field redefinition flattening a given adjustment always exists. Because of the trivial center, the only primitive is then $\zeta \equiv 0$. Thus, centerless simple Lie group bundles are an ordinary setting (up to field redefinitions) if and only if the group bundle is trivial, if and only if $P$ is trivial.
    \end{example}
    
    \subsection{Action groupoids}\label{ssec:action_groupoids}
    
    \paragraph{Generalities.} A particularly important class of examples is formed by action groupoids. This situation is less well understood, but it is the source of many examples that are relevant in physical applications.    
    
    In the following we consider ordinary action groupoids, i.e.~action groupoids in the sense of \ref{def:action_groupoid}, but with trivial differential and objects and morphisms concentrated in degrees $0$.
    
    We have the following statement, which allows us to construct natural connections on action groupoids.
    \begin{proposition}[Induced connection on action groupoid]\label{prop:InducedCartConnOnActionGr}
        Consider a Lie groupoid $\scG$ with an action $(\racton,\psi)$ on $N$. A Cartan connection $\caH$ on $\scG$ induces a Cartan connection on the corresponding action groupoid.
    \end{proposition}
    \begin{proof}
        Recall that $\caH$ equivalently defines a projection $\sfp_{\caH}\colon \rmT\sfG\rightarrow\ker(\rmd \sft)$, and we have the following cone over a pullback diagram:
        \begin{equation}
            \begin{tikzcd}
                (\rmd \psi)^*\rmT\sfG \arrow[rd,dashed,"\sfp_{\hat\caH}"] \arrow[rrr] \arrow[ddr,bend right, swap,"\rmd \sft_\scK"] & & & \rmT \sfG \arrow[dl, swap,"\sfp_{\caH}"] \arrow[ddl, bend left, "\rmd\sft"]
                \\
                & (\rmd \psi)^*(\ker(\rmd\sft)) \arrow[r] \arrow[d] & \ker(\rmd\sft) \arrow[d] &
                \\
                & \rmT N\arrow[r,"\rmd\psi"] & \rmT M &
            \end{tikzcd}
        \end{equation}
        with $\rmT \left(N\fibtimes{\psi}{\sft}\sfG\right)=(\rmd \psi)^*\rmT\sfG$. The morphism $\sfp_{\hat\caH}$ is uniquely given as the form-pullback of $\sfp_{\caH}$ along the canonical projection $N\fibtimes{\psi}{\sft}\sfG \to \sfG$, and it defines the connection $\hat \caH$ on the action groupoid $\scK$ as a map from $\rmT (N\fibtimes{\psi}{\sft}\sfG)$ to $\ker(\rmd \sft_\scK)=(\rmd \psi)^*(\ker(\rmd\sft))$. 
        
        Because composition in $\scK$ (and hence in $\rmT \scK$) is defined in terms of composition in $\scG$, and $\caH\rightrightarrows M$ is a subgroupoid of $\rmT \scG$, multiplicativity then follows immediately by inspection.
    \end{proof}
    
    We also need the corresponding statement for the adjustment datum.
    \begin{proposition}[Canonical adjusted connection on action groupoids and algebroids]\label{prop:CanonicalCompConnOnActionGroupoid}
        Consider a groupoid $\scG$ with an action $(\racton,\psi)$ on a manifold $N$ along a submersion $\psi$ and denote the corresponding action groupoid by $\scK$. Let $\caH$ be a Cartan connection on $\scG$ and let $(\nabla,\zeta)$ be a covariant adjustment of the Lie algebroid $\frg$ of $\scG$ such that $\nabla$ is the vector bundle connection induced by $\caH$ on $\frg$. Then $\caH$ induces a Cartan connection on $\scK$ with Lie algebroid connection $\nabla_\frk$ on the Lie algebroid $\frk$ of $\scK$, and $(\nabla_\frk,\zeta_\frk)$ forms a covariant adjustment of $\frk$, where $\zeta_\frk$ is the pullback of $\zeta$ (regarded as a 2-form) along $\psi$. 
    \end{proposition}
    \begin{proof}
        We note that $\nabla_\frk$ is the pullback connection of $\nabla$ along $\psi$. By construction, it is the infinitesimal version of the Cartan connection on $\scK$ as constructed in \Cref{prop:InducedCartConnOnActionGr}, and thus it has a vanishing basic curvature.\footnote{With the following arguments this could also be shown directly, even if integrability would not be given.} Due to the pullback structure in $\frk$ it is useful to show compatibility conditions w.r.t.\ the canonical generators given by pullbacks of sections of $\frg$ and \emph{projectable vector fields} $W \in \mathfrak{X}(N)$, that is, $\rmd \psi (W) = \psi^* V$ for some vector field $V \in \mathfrak{X}(M)$; we then just write $W = V'$ for simplicity, even though $V'$ is not unique for a given $V$. First, recall the following property of algebroid actions:
        \begin{equation}\label{eq:psiAnchorMorph}
            \rmd\psi\bigl(\rho_{\frk}\left(\psi^*\nu\right) \bigr)
            =
            \psi^*\bigl( \rho_\frg (\nu) \bigr)
        \end{equation}
        for all $\nu \in \Gamma(\frg)$, i.e.\ $\rho_{\frk}\left(\psi^*\nu\right) =\left( \rho_\frg (\nu) \right)' $, and therefore we have
        \begin{equation}
            \begin{aligned}
                \bigl(\nabla_\frk\bigr)^{\mathrm{bas}}_{\psi^*\nu}\left( \psi^*\mu \right)
                &=
                \psi^*\left(
                \left[ \nu, \mu \right]_\frg
                + \nabla_{\rho_\frg(\mu)} \nu
                \right)
                =
                \psi^*\left(
                \nabla^{\mathrm{bas}}_\nu \mu
                \right)~,
                \\
                \bigl(\nabla_\frk\bigr)^{\mathrm{bas}}_{\psi^*\nu}\left( V' \right)
                &=
                \left[  \left( \rho_\frg (\nu) \right)', V' \right]
                + \bigl( \rho_\frg\left( \nabla_V \nu \right) \bigr)'
                =
                \left( \nabla^{\mathrm{bas}}_\nu V \right)'
            \end{aligned}
        \end{equation}
        for all $\mu, \nu \in \Gamma(\frg)$ and $V \in \mathfrak{X}(M)$. Thus, it is straightforward to check that we have
        \begin{equation}
            \begin{aligned}
                R_{\nabla_\frk}(V', W')\psi^*\nu
                &=
                \psi^*\left( \left(\nabla^{\mathrm{bas}}_\nu \zeta\right)(V, W) \right)
                =
                \left(\bigl(\nabla_\frk\bigr)^{\mathrm{bas}}_{\psi^*\nu}\left(\psi^!\zeta\right)\right)(V', W')
            \end{aligned}
        \end{equation}
        for all $\mu, \nu \in \Gamma(\frg)$ and $V, W \in \mathfrak{X}(M)$, where $\psi^!\zeta$ denotes the pullback of $\zeta$ as a form, not as a section. In particular,
        \begin{equation*}
            \left( \psi^!\zeta \right)(V', W')
            =
            \psi^*\bigl( \zeta(V, W) \bigr)~.
        \end{equation*}
        This finishes the proof.
    \end{proof}
    
    \begin{remark}\label{rem:ConditionAboutProjectab}
        We assume that $\psi$ is a submersion so that $\frX(N)$ is generated by projectable vector fields. The proof thus canonically extends to all $\psi$ with this property on $\frX(N)$, for example, the proof also works for constant $\psi$. Thus, assuming the existence of projectable generators can certainly be relaxed to a more general assumption: observe that the existence of an action already implies that $\psi$ is a submersion by Eq.~\eqref{eq:psiAnchorMorph}, if $\scG$ is a transitive groupoid. Even more general, the proof works for $\psi$ admitting generators $\caV$ of $\frX(N)$ such that 
        \begin{equation*}
            \bigl(\nabla_\frk\bigr)^{\mathrm{bas}}_{\psi^*\nu}\left( \caV \right)
            =
            \left( \left(\psi^*\nabla^{\rm{bas}}\right)_{\psi^*\nu} \bigl(\rmd\psi(\caV)\bigr) \right)'
        \end{equation*}
        for all $\nu \in \Gamma(\frg)$, where one views $\rmd\psi(\caV)$ as a section of $\psi^*\rmT M$, and $\psi^*\nabla^{\rm bas}$ is the induced $\frk=\psi^*\frg$-connection on $\psi^*\rmT M$, given as the pullback of $\nabla^{\rm bas}$ as a $\frg$-connection on $\rmT M$ along the evident projection $\psi^*\frg \to \frg$; such a pullback exists here due to the fact that $\psi^*\frg \to \frg$ is a morphism of anchored vector bundles by Eq.~\eqref{eq:psiAnchorMorph}.\footnote{If the reader is unfamiliar with such pullbacks then see~\cite{Fischer:2021glc, Fischer:2021yoy}.} To do so, one only has to check whether $\left[ \rho_{\frk}\left(\psi^*\nu\right), \caV \right]$ projects ``nicely'' because the other summand in $\bigl(\nabla_\frk\bigr)^{\mathrm{bas}}_{\psi^*\nu}\left( \caV \right)$ already behaves properly due to the fact that we define $\nabla_\frk$ as $\psi^*\nabla$.
        
        The proof is left to the reader; in our corresponding examples $\psi$ will be a (surjective) submersion, so that we decided to keep the proof simpler.
    \end{remark}
    
    We can combine the previous results to the following corollary (assuming that $\psi$ is a submersion or alternatively satisfies what we discussed in \Cref{rem:ConditionAboutProjectab}).
    
    \begin{corollary}[Canonical covariant adjustments on action groupoids]\label{cor:induced_adjustments}
        (Covariant) adjustments of groupoids acting on manifolds induce (covariant) adjustments on the corresponding action groupoids.
    \end{corollary}
    
    Observe that the previous results are pure pullback arguments. In particular:
    
    \begin{corollary}[Canonical strict covariant adjustments on action groupoids]\label{cor:induced_strictadjustments}
        Strict covariant adjustments of groupoids acting on manifolds induce strict adjustments on the corresponding action groupoids.
    \end{corollary}
    
    \paragraph{Recovering physically known examples.} Let us briefly go through a number of situations that commonly appear in field theory within physics. 
    
    The first case is that of the kinematical data for a non-linear sigma-model, which can be regarded as a principal $\scG$-bundle for the extreme case of $\scG$ the trivial groupoid $\scG=(*\rightrightarrows *)$ and $N$ the target space of scalar (or Higgs) fields. The resulting action groupoid is discrete: $\scK=(N\rightrightarrows N)$, and there is a unique, trivial groupoid connection and hence adjustment.
    
    The other extreme case is that of the kinematical data for a gauge theory, i.e.\ a connection on a principal fiber bundle with structure Lie group $\sfG$. This can be regarded as a principal $\scG$-bundle with connection for the extreme case of a single object groupoid $\scG=(\sfG\rightrightarrows *)$ acting on the trivial manifold $N=*$. Again, the groupoid connections and adjustments are evidently trivial and known as \emph{Maurer--Cartan form}.
    
    The evident combination of both cases is that of a gauged sigma model, where the kinematical data is given by a manifold $M$, in which scalar fields (sometimes called ``Higgs fields'') take values and on which a symmetry group $\sfG$ acts. This data is described by a principal $\scG$-bundle with connection, where $\scG=(\sfG\rightrightarrows *)$ acts non-trivially on a target space manifold $N$, leading to the action groupoid $\scK=(N\rtimes \sfG\rightrightarrows N)$. The adjustments on $\scK$ induced according to \ref{cor:induced_adjustments} are again trivial.
    
    A specialization of the previous situation is when $N$ is a vector space. A principal $\scG$-bundle is then a principal $\sfG$-bundle, together with an associated vector bundle for the representation given by the action of $\sfG$ on $N$. This describes the kinematical data for gauge--matter theories.\footnote{The usual nomenclature is that connections on principal fiber bundles are regarded as gauge fields, while sections of associated vector bundles, i.e.~objects which involve the choice of a representation of the structure group, are regarded as matter fields.} 
    
    \paragraph{More general action groupoids.}
    We now extend the constructions of \ref{ssec:Lie_group_bundles} to action groupoids. Canonically, the Atiyah groupoid $\scG \coloneqq (P \times P) / \sfH$ of $P$ acts on $\scT = \left(P \times \mathbb{R}^d\right) / \sfH$ via 
    \begin{equation}
        [p, q] \cdot [q, y]
        \coloneqq
        [p, y]
    \end{equation}
    for all $p, q \in P$ and $y \in \mathbb{R}^d$, where $[\cdot, \cdot]$ denotes the corresponding equivalence classes, cf.~\cite[Ex.\ 1.6.4]{0521499283}. The orbits of the anchor of the corresponding action\footnote{In our case w.r.t.\ the canonically induced right-action.} groupoid $\scK$ are given by $\scF$. Its Lie algebroid is canonically the action algebroid $\frk$ induced by the corresponding action of the Atiyah algebroid $\frg$ of $P$.  
    Then observe that $M$ being a leaf of $\scF$ implies that
    \begin{equation}\label{eq:AlgebroidRestrictionOfActionAlg}
        \left.\scK\right|_{M}
        \cong
        \scG
        \eand
        \left.\frk\right|_{M}
        \cong
        \frg
    \end{equation}
    as Lie groupoids and Lie algebroids, where the restriction is the usual restriction of Lie groupoids and Lie algebroids along an orbit.
    
    \paragraph{Action groupoids generating transverse foliation.}
    Every foliation $\scF$ of $\scT$ has a canonical \emph{transverse foliation $\tau$}, which is here just the vertical part of $\scF$ in $\scT$. By construction, $\tau$ is given as the image of the $c_\sfH(P)$-action on $\scT$, i.e.\ generated by vector fields of the form $\overline{\nu}$ for $\nu \in \Gamma(\mathrm{ad}(P))$. In other words, we have a subgroupoid of $\scK$ given by the action algebroid induced by the $c_\sfH(P)$-action, which we denote by $\scK^{\rm t}$. Its action Lie algebroid $\frk^{\rm t}$ is a subalgebroid of $\frk$, induced by the $\mathrm{ad}(P)$-action on $\scT$, and the image of its anchor is $\tau$. By construction, equations~\eqref{eq:AlgebroidRestrictionOfActionAlg} restrict to
    \begin{equation}\label{eq:AlgebroidRRestrictionOfActionAlg}
        \left.\scK^{\rm t}\right|_{M}
        \cong
        c_\sfH(P)
        \eand
        \left.\frk^{\rm t}\right|_{M}
        \cong
        \mathrm{ad}(P)~.
    \end{equation}
    
    \begin{corollary}[Covariant adjustment on $\scK^t$]\label{cor:InducedConnViaPullback}
        Every (covariant) adjustment on $\scK^t$ induces a (covariant) adjustment on $c_\sfH(P)$ by restriction; vice versa, every (covariant) adjustment on $c_\sfH(P)$ induces a (covariant) adjustment on $\scK^{\rm t}$ via pullback along the canonical projection $\scK^{\rm t} \to c_\sfH(P)$. 
    \end{corollary}
    
    \begin{proof}[Proof of \ref{cor:InducedConnViaPullback}]
        The first part is trivial (via equation~\eqref{eq:AlgebroidRRestrictionOfActionAlg}), while the second one follows by \ref{prop:CanonicalCompConnOnActionGroupoid}.
    \end{proof}
    
    \begin{remark}[Strictness]\label{rem:StricTransverse}
        By \Cref{cor:induced_strictadjustments} and \Cref{rem:LGBStrictness}, an associated connection on $c_\sfH(P)$ induces a strict covariant adjustment on $\scK^{\rm t}$.
    \end{remark}
    
    We now want to extend the arguments of \ref{ssec:Lie_group_bundles} to $\scK^{\rm t}$, claiming that $\scK^{\rm t}$ admits a flat adjustment if and only if $\scF$ admits a flat $\scF$-connection. This is the case if and only if $P$ is flat. However, we can only conclude this by assuming \emph{simplicity}. 
    
    \begin{theorem}[Curved action groupoids]\label{thm:CurvedGroupoidEx} Let $\sfH$ be a simple and centerless Lie group. Then $\scK^{\rm t}$ admits a flat adjustment if and only if $P$ admits a flat Ehresmann connection (in the ordinary sense).
    \end{theorem}
    
    \begin{remark}[Field redefinitions]
        In particular, there is no field redefinition flattening the adjustment on $\scK^{\rm t}$, if $P$ is not flat. However, this is not a simple consequence of \ref{ssec:Lie_group_bundles} (recall \ref{tab:ComparingTheGeometryGaugeTheoryWithSingularFoliations}) because field redefinitions in $\scK^{\rm t}$ are richer than the ones in $c_\sfH(P)$; this is due to possibly non-trivial transverse directions of $\scF$ encoded in $\tau$. The assumptions are rather crucial to avoid the problem of "transverse field redefinitions", and in essence avoiding introducing general field redefinitions as a whole.
    \end{remark}
    
    \begin{proof}[Proof of \ref{thm:CurvedGroupoidEx}]
        It is clear that a flat ordinary Ehresmann connection on $P$ induces a flat adjustment on $c_\sfH(P)$ following \ref{ssec:Lie_group_bundles}. By \ref{cor:InducedConnViaPullback}, its pullback gives a flat adjustment on the action groupoid $\scK^{\rm t}$.
        
        It remains to show the other direction. With the assumptions, we know from \ref{lem:SemiSimpleLGBAdjust} that all adjustments on $c_\sfH(P)$ are associated to connections on $P$. Furthermore, the simplicity implies a one-to-one correspondence between associated connections on $c_\sfH(P)$ and ordinary Ehresmann connections on $P$, as already implied by the discussion around \ref{tab:ComparingTheGeometryGaugeTheoryWithSingularFoliations}.
        
        Now assume that we have a flat adjustment on $\scK^{\rm t}$. By \ref{cor:InducedConnViaPullback} it induces a flat adjustment on $c_\sfH(P)$. Following the previous paragraph, such a connection is in a one-to-one correspondence to an ordinary flat Ehresmann connection on $P$. This concludes the proof.
    \end{proof}
    
    We have a natural enhancement of \ref{ex:Hopf}, elaborated on the Hopf fibration as canonical example:
    
    \begin{example}\label{ex:Hopf2}
        Let $P$ be the Hopf fibration $S^7 \to S^4$, whose structural Lie group is given by $S^3 \cong \mathrm{SU}(2)$, and define $\scT \coloneqq \left( P \times \mathbb{C}^2\right) / S^3$. The canonical action algebroid $\scK^{\rm t}$ inherited by the $c_\sfH(P)$-action on $\scT$ admits an adjustment by \ref{cor:InducedConnViaPullback} (given by a pullback of an adjustment on $c_\sfH(P)$), but there is no flat adjustment since $P$ would be otherwise trivial due to the fact that $S^4$ is simply connected. In particular, such adjustments are strict by \Cref{rem:StricTransverse}.
    \end{example}
    
    \paragraph{Action groupoids generating whole foliations.}
    
    Naturally, we would like to argue similarly for $\scK$ and $\frk$, whose anchor's orbits give rise to the whole foliation $\scF$, not just the transverse one. In the same fashion we would like to use \ref{prop:CanonicalCompConnOnActionGroupoid}. However, even though the Atiyah algebroid $\frg$ is transitive, the existence of a Cartan connection is not necessarily given. Also in our context, \ref{tab:ComparingTheGeometryGaugeTheoryWithSingularFoliations} only gives a relation to the curvature along the transverse groupoid, as we previously argued.
    
    There are several ways to construct a Cartan connection over $\scF$. The orbits of the $\frg$-action on $\scT$ generate $\scF$ in our construction. That is, $\scF$ is induced by the action of the Lie algebra $\Gamma(\frg)$, an infinite-dimensional Lie algebra.\footnote{For general singular foliations there is always locally a free Lie algebra generating it, also infinite-dimensional; in the context of~\cite{Fischer:2401.05966} $\frg$ itself is already infinite-dimensional in general.} Thus, we have an action algebroid over $\scF$ corresponding to this Lie algebra action, implying a canonical (flat) covariant adjustment.
    
    Ideally, we would like to derive a covariant adjustment on the Atiyah groupoid $\scG$ itself. In the following, however, we work only locally because we do not yet know how to easily characterize the situation in which adjustments on $\scG$ exist. This time we want to use \ref{sec:PairGroupoids}. To do so, 
    let us recall some basic facts of splittings of the Atiyah groupoid sequence
    \begin{equation}\label{eq:AtiyahSequenceToFundGr}
        \begin{tikzcd}
            c_\sfH(P) \arrow[hook]{r} &
            \scG \arrow[two heads]{r} &
            \scP\mathsf{air}(M)~.
        \end{tikzcd}
    \end{equation}
    A right section groupoid morphism $\chi \colon \scP\mathsf{air}(M) \to \scG$ of~\eqref{eq:AtiyahSequenceToFundGr}, also called a \uline{splitting}, is in one-to-one correspondence to flat connections with trivial holonomy, inducing trivializations on $P$; $\chi$ is also called a (special type of) multiplicative Ehresmann connection on $\scG \to \scP\mathsf{air}(M)$.\footnote{See~\cite{Fernandes:2204.08507} for a good introduction.} If it exists, $\chi$ induces a $\scP\mathsf{air}(M)$-action on $c_\sfH(P)$ simply given by the associated connection, a strict adjustment as discussed before, and now also the canonical flat connection (Maurer--Cartan form) w.r.t.\ the inherited trivialization of $c_\sfH(P)$.
    The splitting $\chi$ also induces a trivialization of $\scG$ by
    \begin{equation}
        \begin{aligned}
            \sfH \times \scP\mathsf{air}(M)
            &\to
            \scG~,
            \\
            (h, m_2, m_1)
            &\mapsto
            \left[s_{m_2}, h\right] \circ \chi\left(m_2, m_1\right)~,
        \end{aligned}
    \end{equation}
    where $s$ is some global section of $P$ corresponding to a trivialization by $\chi$. We are now able to define a strict covariant adjustment on the action groupoid $\scK$, corresponding to the action of the Atiyah groupoid $\scG$ on the associated bundle $\scT$.
    
    \begin{proposition}[A strict covariant adjustment on $\scK$]\label{prop:AdjustmentOnActionAlgOverF}
        Assume that the principal bundle $P$ over $M$ is trivial, and that $\scP\mathsf{air}(M)$ admits an adjustment. Then the action groupoid $\scK$ admits a strict covariant adjustment restricting to a canonical adjustment on a given splitting $\scK|_M \cong \scG \cong \sfH \times \scP\mathsf{air}(M)$, which is flat if and only if the adjustment on $\scP\mathsf{air}(M)$ is flat.
    \end{proposition}
    
    \begin{proof}
        Assume a given splitting $\scG \cong \sfH \times \scP\mathsf{air}(M)$ inherited by a fixed trivialization of $P$. By assumption, we have an adjustment on $\scP\mathsf{air}(M)$ which is covariant by \Cref{prop:PrimitivesForFundGroupi}, and whose multiplicative and horizontal distribution we denote by $\caH^M$. As a canonical candidate for a covariant adjustment, we define a horizontal distribution $\caH$ on $\scG$ now by 
        \begin{equation}
            \caH
            \coloneqq
            \pi^*\caH^M~,
        \end{equation}
        where $\pi \colon \scG \to \scP\mathsf{air}(M)$ is the canonical projection induced by the fixed trivialization of $P$. This distribution is by construction a Cartan connection and admits a primitive $\zeta \coloneqq 0 \oplus \zeta^M$, where $0$ is the zero element of the Lie algebra of $\sfH$. Furthermore, there is always a strict covariant adjustment due to \Cref{rem:StrictnessPiM}. By a further pullback, making use of \Cref{cor:induced_strictadjustments}, we have a strict covariant adjustment on the action groupoid $\scK$.
    \end{proof}
    
    \begin{remark}
        Following the proof, every action groupoid $\scK$ induced by an action of a transitive groupoid over $M$ admits locally a (strict) covariant adjustment: transitive groupoids are locally always trivial (cf.~\cite[\S 1.3]{0521499283}) and $\scP\mathsf{air}(M)$ always has a (strict) covariant adjustment locally, as argued in \ref{sec:PairGroupoids}.
    \end{remark}
    
    \begin{example}
        As canonically curved examples one could define $M$ to be given by the unit octonions $S^7$ (recall \Cref{ex:PairGroupoidsEx}), while taking any arbitrary trivial principal bundle $P$ over $S^7$. Following the construction above one finds curved covariant adjustments on $\scK$.
    \end{example}
    
    \begin{remark}[Field redefinitions]
        Here, we could not answer whether there is a field redefinition flattening such covariant adjustments because a field redefinition may not preserve a given trivialization $\scG \cong \sfH \times \scP\mathsf{air}(M)$. Thus, possibly arising mixed terms complicate the discussion, and a field redefinition on $\scK$ may even allow to change the structural group $\sfH$, see~\cite[Rem.\ 5.3.10]{Fischer:2021glc}. However,~\cite[Lemma 5.3.7]{Fischer:2021glc} implies that a field redefinition of a covariant adjustment preserving a splitting still restricts to ${\rm ad}(P)$, so that there is hope for curved covariant adjustments which cannot be flattened. This will be investigated in future works.
    \end{remark}
    
    \appendices
    
    \section{Basics on graded manifolds}\label{app:graded_manifolds}
    
    Below, we summarize some facts on (differential) graded manifolds. For further details, see e.g.~\cite{Kostant:1975qe,Cattaneo:2010re,Fairon:1512.02810,Vysoky:2021wox}. While $\IZ_2$-gradings are very popular in the physics literature, in particular in the context of supersymmetry, we work exclusively with $\IZ$-gradings.
    
    \paragraph{Graded vector spaces.} A \uline{graded vector space} $V$ is a vector space together with a decomposition $V=\bigoplus_{k\in \IZ} V_k$ that respects the linear structure on $V$. For an element $v\in V_k$, we say that $v$ has \uline{degree} $k\in \IZ$, and we write $|v|=k$. 
    
    Given a graded vector space $V=\bigoplus_{k\in \IZ} V_k$,   
    the \uline{grade-shifted graded vector space} $V[i]$ is defined as the graded vector space
    \begin{equation}
        V[i]\coloneqq \oplus_{k\in \IZ}V[i]_k~~~\mbox{with}~~~V[i]_k\coloneqq V_{i+k}~.
    \end{equation}
    We say that a graded vector space is \uline{concentrated} in degrees $I=\{i_1,\ldots, i_k\}$ if $V_j=*$ unless $j\in I$.
    
    The \uline{algebra of polynomial functions} $P(V)$ on a graded vector space $V$ is the graded symmetric algebra
    \begin{equation}
        \ourodot^\bullet V^*=~V^0~~\oplus~~V^*~~\oplus~~V^*\odot V^*~~\oplus~~\ldots~,
    \end{equation}
    where $V^0$ denotes the ground field of the vector space and $\odot$ is the symmetric tensor product. The \uline{algebra of smooth functions} $C^\infty(V)$ is obtained by the tensor product with $C^\infty(V_0)$ over $P(V_0)$.
    
    \paragraph{Graded manifolds.} A \uline{graded manifold} is an object that is locally modeled on a graded vector space. That is, a graded manifold is a pair $\caX=(X,\caO_X)$, where $|\caX|\coloneq X$, the \uline{body} of $\caX$, is a topological manifold, and $\caO_X$ is a sheaf of $\IZ$-graded rings on it, such that the stalk $\caO_x=\caO_{X,x}$ is a local ring, and $\caX$ is locally isomorphic to
    \begin{equation}
        \left(U,C^\infty(|\caX|)\otimes \ourodot^\bullet V^*\right)
    \end{equation}
    where $V$ is a graded vector space over a patch $U\subset |\caX|$. If $V$ is \uline{concentrated} in degrees $I=\{i_1,\ldots,i_k\}$, then we say that $X$ is concentrated in degrees $I$ if $X=*$ and $X$ is concentrated in degrees $I\cup\{0\}$ otherwise.\footnote{For some purposes, it is more convenient to introduce $\IZ^*$-graded manifolds, in order to avoid having two types of $\caO_X$-elements of degree~0, cf.~\cite{Kotov:2108.13496,Kotov:2212.05579} for more details.}
    
    The \uline{algebra of smooth functions} on the graded manifold $\caX=(X,\caO_X)$ is the algebra of smooth global sections of $\caO_X$ and denoted by $C^\infty(\caX)$. We will regard all our graded manifolds as smooth.
    
    \paragraph{Morphisms of graded manifolds.} The definition of a morphism of graded manifolds is evident from the definition of graded manifolds: they are morphisms of sheaves of $\IZ$-graded rings. Explicitly, a \uline{morphism between graded manifolds} $\caX=(X,\caO_X)$ and $\caY=(Y,\caO_Y)$,
    \begin{equation}
        f: \caX\rightarrow \caY~,
    \end{equation}
    is given by a pair $f=(|f|,f^\sharp)$ consisting of a continuous map $|f|:X\rightarrow Y$ as well as a ring homomorphism $f^\sharp:\caO_Y\rightarrow |f|_*\caO_X$, which is local in the sense that the stalks of $\caO_Y$ at $f(x)$ are mapped to the stalks of $\caO_X$ at $x$. Here, $|f_*|\caO_X$ denotes the direct image sheaf.
    
    If $\caY$ is a graded vector space (e.g.~a local patch of a graded manifold), we can simply write 
    \begin{equation}
        \sfHom(\caX,\caY)=\sfHom(C^\infty(\caY),C^\infty(\caX))=(\caY\otimes C^\infty(\caX))_0~,
    \end{equation}
    where $C^\infty(\caX)$ and $(-)_0$ denote the smooth algebra of functions on $\caX$ and the restriction to maps of degree~$0$, respectively. In the last equality, we used the fact that the contained map $f^\sharp$ in a morphism of graded manifolds $f=(|f|,f^\sharp)$ is a ring homomorphism.
    
    \paragraph{N$\boldsymbol Q$-manifolds.} It will also be useful to  introduce the category $\CatNQMfd$ of differential graded manifolds concentrated in non-positive degrees, also known as \uline{N$Q$-manifolds}. We note that these correspond precisely to higher Lie algebroids.
    
    \paragraph{Internal homs of graded manifolds.} Recall that a presheaf on a category $\scC$ is a functor $\scC^{\rm op}\rightarrow \catSet$, and an interesting example is the presheaf 
    \begin{equation}
        \sfHom(-\times Y,Z): X\mapsto \sfHom(X\times Y,Z)
    \end{equation}
    with $X,Y,Z\in \scC$. We say that a presheaf is \uline{representable} if there is another object $\ihom(Y,Z)\in \scC$, the \uline{inner} or \uline{internal hom}, such that 
    \begin{equation}
        \sfHom(X\times Y,Z)=\sfHom(X,\ihom(Y,Z))
    \end{equation}
    for all $X,Y,Z\in \scC$. 
    
    In particular, in the category of sets, we simply have $\ihom(Y,Z)=\sfHom(Y,Z)$. This also holds in the category of vector spaces, but it fails in the category of graded vector spaces. For example, let $V$, $U$, and $W$ be graded vector spaces concentrated in positive degrees. We have
    \begin{equation}
        \sfHom(V,W)=\sfHom(V\times *,W)\neq \sfHom(V,\sfHom(*,W))=\sfHom(V,*)~,
    \end{equation}
    where $*$ denotes the single point with structure sheaf $\IR$. However, there is an internal hom consisting of linear maps of arbitrary degree,
    \begin{equation}
        \sfHom(V\times U,W)=\sfHom(V,\ihom(U,W))
        \ewith
        \ihom(U,W)\cong W\otimes C^\infty(U)~.
    \end{equation}
    This also gives us a local description of the internal hom on graded manifolds. 
    
    In certain cases, the internal hom of graded manifolds exists and is readily calculated. Particularly important in our discussion is
    \begin{equation}
        \ihom(\IR[-1],X)=\rmT[1]X~.
    \end{equation}
    For more details on this, see e.g.~\cite{Severa:2006aa,Cattaneo:2010re} as well as~\cite{Sachse:0802.4067,Alldridge:1109.3161,Bonavolonta:1304.0394} for the closely analogous case of supermanifolds.
    
    \paragraph{Differentials on graded manifolds.} A differential on a graded manifold $\caX=(X,\caO_X)$ is a differential on the algebra $C^\infty(\caX)$ of smooth functions on $\caX$. Equivalently, the differential can be seen as a nilquadratic vector field of degree~$1$. A graded manifold endowed with a differential is a \uline{differential graded (dg-)manifold}. Morphisms of dg-manifolds, or \uline{dg-maps} for short, are morphisms of graded manifolds respecting the differential. Together with dg-manifolds, they form the category $\catdgMfd$. 
    
    We are often interested in the internal hom in $\catdgMfd$, and we have, in particular for a manifold $X$
    \begin{equation}
        \ihom_\catdgMfd(\IR[-1],X)=(\rmT[1]X,\rmd)~,
    \end{equation}
    where we identified $C^\infty(\rmT[1]X)$ with the de~Rham complex $\Omega^\bullet(X)$ on $X$ and $\rmd$ is the de~Rham differential, cf.~\ref{def:grade-shifted_tangent_bundle}.
    
    Consider an internal hom $\ihom(\caX,\caY)$ of graded manifolds $\caX$ and $\caY$ which are both endowed with differential $\sfd_\caX$ and $\sfd_\caY$. Note that $\ihom(\caX,\caY)$ naturally extends to an internal hom in $\catdgMfd$, where the differential $\sfd$ is given by
    \begin{equation}
        \sfd: f^\sharp \rightarrow f^\sharp\circ \sfd_\caY+\sfd_\caX\circ f^\sharp
    \end{equation}
    for $f=(|f|,f^\sharp)\in \ihom(\caX,\caY)$.
    
    Very useful will be the following observation.
    \begin{remark}\label{rem:graded-map-to-differential-map}
        Consider a dg-manifold $(\rmT[1]\caN,\rmd)$, where the differential $\rmd$ contains the de~Rham differential on $\caN$. Then $\rmT[1]\caN$ is generated by the set
        \begin{equation}
            \{\phi,\rmd \phi~|~\phi\in C^\infty(\caN)\}~.
        \end{equation}
        This further implies that a morphism of dg-manifolds $f:(\caM,\rmd)\rightarrow (\rmT[1]\caN,\rmd)$ is uniquely defined by a morphism of graded manifolds $f_\flat:\caM\rightarrow \caN$: we have $|f|=|f_\flat|$, and $f^\sharp$ is uniquely fixed by 
        \begin{equation}
            f^\sharp(\phi)=f_\flat^\sharp(\phi)\eand f^\sharp(\rmd_\caN \phi)=\rmd_\caM f_\flat^\sharp(\phi)~.
        \end{equation}
    \end{remark}
    
    \section*{Acknowledgments}
    
    S.-R.F.\ wants to thank the National Center for Theoretical Sciences, and he especially thanks Camille Laurent-Gengoux for discussions and for sharing thoughts and ideas. M.J.F.\ was supported by the STFC PhD studentship ST/W507489/1. H.K.\ was partially supported by the Leverhulme Research Project Grant RPG-2021-092. C.S.\ was partially supported by the Leverhulme Research Project Grant RPG-2018-329.
    
    \section*{Data and License Management}
    
    No additional research data beyond the data presented and cited in this work are needed to validate the research findings in this work. For the purposes of open access, the authors have applied the \href{https://creativecommons.org/licenses/by/4.0/}{Creative Commons Attribution 4.0 International (CC~BY~4.0)} license to any author-accepted manuscript version arising from this work.

    \bibliography{bigone}
    
    \bibliographystyle{latexeu2}
    
\end{document}